\documentclass[10pt,a4paper]{article}
\usepackage[T1]{fontenc}
\usepackage[utf8]{inputenc}
\usepackage[sc]{mathpazo}
\usepackage{amsmath, amssymb, amsthm, mathrsfs, amsfonts, esint, mathtools}
\usepackage{accents}
\usepackage{cite}
\usepackage{a4wide}
\usepackage{xcolor}
\usepackage{tikz}
\usepackage{pgfplots}
\usepackage{float}
\pgfplotsset{compat=1.18}

\usetikzlibrary{arrows.meta,calc,positioning,decorations.pathreplacing,shapes.geometric,fit,shadows.blur}
\definecolor{AJblue}{RGB}{18,76,140}
\definecolor{AJblue2}{RGB}{10,50,105}
\definecolor{AJlight}{RGB}{235,243,255}
\definecolor{AJmid}{RGB}{170,205,255}
\definecolor{AJgray}{RGB}{110,110,110}
\newtheorem{mainthm}{Theorem}

\newtheorem{teorema}{Theorem}
\newtheorem{corolario}[teorema]{Corollary}
\newtheorem{defini}[teorema]{Definition}
\newtheorem{lema}[teorema]{Lemma}
\newtheorem{prop}[teorema]{Proposition}
\newtheorem{obs}[teorema]{Remark}
\newtheorem{exe}[teorema]{Example}

\newtheorem{claim}{{\sc Claim}}
\renewcommand{\theclaim}{\arabic{claim}} 
\newcommand{\A}{\mathcal{A}}
\newenvironment{claimproof}{%
	\begin{description}
		\item[\textit{{\sc Proof of Claim.}}]
	}{%
		\hfill $\diamond$
	\end{description}
}
\DeclareMathOperator{\sgn}{sgn}
\newcommand{\supp}{\operatorname{supp}}
\usepackage{fancyhdr}
\usepackage[colorlinks=true,linkcolor=blue,citecolor=blue,urlcolor=blue]{hyperref}
\usepackage{cleveref}
\usepackage{chngcntr} 
\counterwithin*{claim}{mainthm}
\counterwithin*{claim}{teorema}
\begin{document}
	\title{Instantaneous analytic smoothing of rough data for the modified and cubic gKdV equations}
	
	\author{
		A.J. Méndez \thanks{
			Argenis J. Méndez G: Instituto de Matemática, Estatística e Computação Científica (IMECC),
			Universidade Estadual de Campinas (UNICAMP), 13083-859, Campinas, SP, Brazil.
			Email: \texttt{ajmendez@unicamp.br}.
			Supported by FAPESP Grant 2024/20513-7 and CNPq Grant 304088/2024-2.\\
			Marcelo Nogueira: Universidade Federal de Ouro Preto (UFOP),
			Campus Mariana (ICSA), 35420-057, Mariana, Minas Gerais, Brazil.
			Email: \texttt{marcelo.nogueira@ufop.edu.br}.
		}
		\and
		Marcelo Nogueira\thanks{Supported by FAPESP Grant 2024/20513-7}
	}
	\maketitle
\begin{abstract}
	We consider the $k$-generalized Korteweg-de Vries equation
	\begin{equation*}
		\partial_{t}v+\partial_{x}^{3}v +\partial_{x}(v^{k+1})=0,
		\qquad (t,x)\in\mathbb R\times\mathbb R,
		\qquad k\in\mathbb Z_{+},
	\end{equation*}
	emphasizing the modified case $k=2$ and the cubic case $k=3$. We prove that
	solutions from low-regularity, possibly singular, data $u_{0}$ become real analytic
	in $(t,x)$ for all $t\neq0$, whenever $u_0$ satisfies a Nelson-type condition
	\begin{equation*}
		\sum_{k=0}^{\infty}\frac{\alpha^{k}}{k!}\,\big\|(x\partial_x)^{k}u_0\big\|_{X}<\infty,
	\end{equation*}
	for some $\alpha>0$. For the cubic equation, this includes data such as
	$u_0=x_{+}^{\lambda}$, singular at the origin; for mKdV even discontinuous
	data yield analytic solutions \emph{e.g} $u_{0}(x)=\sgn(x)e^{-x^{2}}$.For mKdV we work in the sharp well-posedness space
	$X=\widehat H^{r}_{s}(\mathbb R)$, $r\in(1,2]$, $s\geq\frac12-\frac1{2r}$,
	with $r=2$ recovering analyticity on $H^s(\mathbb R)$, $s\geq\frac14$, the
	best mKdV space in the sense of Kato; for the cubic equation we work in
	$X=H^{s}(\mathbb R)$, $s>-\frac16$, approaching the critical exponent
	$s=-\frac16$ from above. Analyticity thus holds on the largest known data class for which mKdV is well-posed, and on data approaching the corresponding threshold for the
	cubic equation, extending a known smoothing
	effect for KdV ($k=1$) to the modified and cubic nonlinearities and to a
	broader class of singular profiles, avoiding pseudo-differential calculus
	via Lorentz-space refinements replacing Bourgain-space localization. The
	mechanism is dispersive: analyticity is generated by the flow,
	symmetrically in time, and singular profiles become instantaneously
	analytic for $t\neq0$.
\end{abstract}
	\begin{center}
		\tableofcontents
	\end{center}
	\makeatother
\section{Introduction}
The following family of equations is called the \textit{$m$-generalized Korteweg-de Vries m-gKdV equation}
\begin{equation}\label{k-gKdv}
	\begin{cases}
		\partial_{t} u + \partial_{x}^{3}u +  \partial_{x}(u^{m+1})= 0, \\
		u(0,x) = u_{0}(x),
	\end{cases}
\end{equation}
where $x, t \in \mathbb{R},\,m \in \mathbb{Z}_{+}$ and $u=u(t,x)$ is a real-valued function. The case $m=1$ is the classical \textit{Korteweg-de Vries (KdV) equation}, the case $m=2$ is the \emph{modified Korteweg-de Vries (mKdV) equation}, and the case $m=3$ is the \emph{cubic gKdV equation}. Our results and analysis presented in this document will be restricted to the cases $m=2$ and $m=3$. Assuming sufficient smoothness and decay at infinity, equation \eqref{k-gKdv} conserves the $L^{1}$-\emph{mass}
\begin{equation}\label{Mass}
	P[u] (t):= \int_{\mathbb{R}} u(t,x)\,dx =  P[u](0),   
\end{equation}
the \emph{$L^{2}$-mass}
\begin{equation*}
	M[u](t):= \int_{\mathbb{R}} u^{2}(t,x)\,dx = M[u](0), 
\end{equation*}
and the \emph{energy}  
\begin{equation}\label{Hamiltonean}
	E[u] (t):= \int_{\mathbb{R}}\left(\frac{1}{2} u_{x}^{2}(t,x) 
	-  \frac{1}{m+2} u^{m+2}(t,x)\right)  \,dx = E[u](0). 
\end{equation}
For $m=1$ and $m=2$ the equation is completely integrable and possesses an infinite hierarchy of conserved quantities; for general $m\geq3,$ only the three quantities above are known to be conserved.

The $m$-gKdV equation admits \emph{solitary wave solutions} or \emph{solitons} with rapid decay at infinity, $v_{c,m}(x,t)=\varphi_{c,m}(x-ct)$, $c>0$, where
\begin{equation*}
	\varphi_{c,m}(x) = \left( \frac{m + 2}{2} c \, \mathrm{sech}^{2}\left( \frac{m}{2} \sqrt{c} \, x \right) \right)^{1/m}
\end{equation*}
is the unique, up to translation, positive decaying solution of $$-c\varphi+\varphi''+\varphi^{m+1}=0.$$

Notice already that these solitary waves are themselves real analytic functions of $x$: $\varphi_{c,m}$ is built out of $\operatorname{sech}$, an entire function, raised to a rational power, and remains real analytic since $\varphi_{c,m}>0$ everywhere. This elementary observation already suggests that analyticity, far from being a special feature of solitary-wave data, may be a structural property of the flow itself, at least for finite time. The main goal of the present paper is to establish exactly this for the modified ($m=2$) and cubic ($m=3$) equations: solutions arising from data in the sharp local well-posedness space for the modified equation, and in $H^s(\mathbb{R})$, $s>-\frac{1}{6}$, for the cubic equation approaching the critical exponent $s=-\frac16$ from above  become instantaneously real analytic in $(t,x)$ for every $t\neq0$ in their interval of existence. Solitary waves are then recovered as a special, already-analytic case of this general phenomenon, rather than an isolated example standing apart from it.

The dispersive smoothing mechanism exploited in this paper has classical roots. Kato~\cite{Kato83} and Kruzhkov and Faminskii~\cite{KF83} showed that solutions of the Korteweg-de Vries equation gain a finite, fixed number of derivatives locally in time, for $t \neq 0$, beyond the regularity of the initial datum. Kato and Ogawa \cite{KO2000} upgraded this finite local gain to full real analyticity under a suitable structural hypothesis on the datum, as we now recall.

The precedent for this type of result is due to Kato and Ogawa \cite{KO2000}, who addressed the analogous question for classical KdV. Given data at essentially the lowest Sobolev regularity for which the Cauchy problem is solvable, how much pointwise regularity can the corresponding solution instantaneously gain for $t\neq0$? Under a suitable structural condition on the datum, Kato and Ogawa \cite{KO2000} showed the answer is real analyticity: if the initial datum satisfies
\begin{equation*}
	\sum_{k = 0}^{\infty} \frac{A_{0}^{k}}{k!}\, \| (x \partial_{x})^{k} \phi \|_{H^{s}(\mathbb{R})} < \infty
\end{equation*}
for some $A_{0}>0$ and $s>-\tfrac{3}{4}$, and possesses only a single point singularity at $x=0$, then the corresponding solution of the focusing KdV equation is real analytic jointly in space and time. A related result from Gevrey-class data was obtained by de Bouard, Hayashi, and Kato \cite{BHK1995}, via operators constructed to commute, or almost commute, with the linear KdV flow. Similar dispersive-smoothing mechanisms, in the same spirit, have since been established for other models: Hayashi and Kato \cite{HK1997} proved an analogous analyticity-in-time and smoothing result for nonlinear Schr\"odinger equations, and Hayashi, Kato, and Ozawa \cite{HKO1996} obtained a related smoothing effect for the Benjamin-Ono equation via a dilation method, later extended by Kaikina, Kato, Naumkin, and Ogawa \cite{KKNO2002} to a full well-posedness and analytic smoothing theory for the same equation. We would like to call 
 particular attention to the work of Tarama \cite{Tarama2004} on the Korteweg-de Vries equation, whose elegant treatment rests on hypotheses on the initial datum that are remarkably rough, indicating that the analytic smoothing phenomenon can be driven by pure decay of the initial data alone, with no structural condition of Kato-Ogawa or Nelson type imposed. However, Tarama's method relies essentially on the inverse scattering transform, which is not available for the cubic equation treated here; this unavailability is precisely what leads us to adopt the unified, purely dispersive approach developed below, based on the commutation properties of $P=3t\partial_t+x\partial_x$ rather than on any integrable structure. The condition above is expressed through iterates of $x\partial_{x}$, the infinitesimal generator of the scaling symmetry $x\mapsto\lambda x$; a datum satisfying such a condition is, in the sense of Nelson's theory of analytic vectors for unbounded operators \cite{Nelson1959}, an \emph{analytic vector} for $x\partial_{x}$. It is precisely this type of condition, rather than Kato-Ogawa's pointwise single-singularity hypothesis, that we impose below for the mKdV and cubic gKdV equations, and this reformulation turns out to genuinely broaden the class of admissible singular data, as the examples at the end of this introduction illustrate.

The mechanism behind the Kato-Ogawa phenomenon is best understood by first examining the linear equation
\begin{equation}\label{linear}
	\begin{cases}
		\partial_{t}u + \partial_{x}^{3}u = 0, \\
		u(0,x) = u_{0}(x),
	\end{cases}
\end{equation}
whose solution
\begin{equation}\label{sol1}
	u(t,x)
	=
	\int_{\mathbb{R}}
	e^{2\pi i x \xi}
	e^{8\pi^{3} i \xi^{3} t}
	\widehat{u_{0}}(\xi)
	\,\mathrm{d}\xi
\end{equation}
depends on $(t,x)$ only through the phase $\Phi(t,x,\xi) = 2\pi x\xi + 8\pi^{3}t\xi^{3}$, which for each fixed $\xi$ is a polynomial in $(t,x)$. If $\widehat{u_{0}}$ decays exponentially, the integral \eqref{sol1} converges absolutely for complex $(t,x)$ near $\mathbb{R}^{2}$, so that the linear Airy flow already carries an intrinsic mechanism generating analyticity simultaneously in space and time.

We identify two different notions of analyticity that appear in this circle of ideas, and it is worth keeping them separate. \emph{Analyticity in Bourgain-Gevrey spaces}, which encodes spatial analyticity through exponential Fourier weights and temporal regularity through Bourgain norms~\cite{B1993a,B1993b}, and studies the analyticity of the solution map $u_{0}\mapsto u$ between suitable Banach spaces. This functional-analytic notion is somewhat different from the pointwise perspective adopted here; nevertheless, related issues in this direction have been treated in several works, see for instance~\cite{KM1986,FP2024,BGK2005}, and in particular the work of Kato and Masuda~\cite{KM1986} and of Foias and Temam~\cite{FT1989}.

Instead, following Kato and Ogawa, we treat analyticity as a \emph{pointwise} property of $(t,x)\mapsto u(t,x)$, detected by exploiting the dispersive smoothing of the flow together with the algebraic structure of the equation. The key tool is factorial control of derivatives generated by the dilation operator
\begin{equation*}
	P = 3t\partial_{t} + x\partial_{x},
\end{equation*}
which encodes the scaling invariance of the Airy equation: controlling every iterate $P^{k}u$ with factorial growth in $k$ yields real analyticity in both variables jointly.

For the linear equation, this can be verified directly. Expanding $e^{8\pi^{3}i\xi^{3}t} = \sum_{n\geq 0} \frac{(8\pi^{3}it)^{n}}{n!}\xi^{3n}$ in \eqref{sol1} and using the identity $\int_{\mathbb{R}} e^{2\pi i x\xi}(2\pi i \xi)^{k}\widehat{f}(\xi)\,d\xi = \partial_{x}^{k}f(x)$ yields the time-Taylor expansion
\begin{equation*}
	u(t,x)
	=
	\sum_{n=0}^{\infty}
	\frac{(-t)^{n}}{n!}\,\partial_{x}^{3n}u_{0}(x)
	=: e^{-t\partial_{x}^{3}}u_{0}(x).
\end{equation*}
If $u_{0}$ is itself real-analytic, the derivatives $\partial_{x}^{3n}u_{0}(x)$ obey factorial bounds, so the series converges near $t=0$, giving joint analyticity near any point $(x_{0},0)$.

The purpose of the present paper is to show that the same phenomenon persists for the nonlinear mKdV and cubic gKdV flows, where the nonlinear term destroys the explicit representation \eqref{sol1}. Our task is to show that dispersive smoothing, together with the commutation properties of $P$ with the equation, still allows factorial bounds on mixed space-time derivatives to be propagated along the nonlinear evolution.

The phenomenon is in fact already visible, in sharpened form, at the level of the fundamental solution of \eqref{linear}. Taking $u_{0}=\delta_{0}$, so that $\widehat{u_{0}}(\xi)=1$, the solution is the Airy kernel
\begin{equation*}
	K(x,t)
	=
	\int_{\mathbb{R}}
	e^{2\pi i x\xi}\,e^{8\pi^{3} i\xi^{3}t}\,
	\mathrm{d}\xi
	=
	(3t)^{-1/3}\,
	\operatorname{Ai}\!\left(\frac{x}{(3t)^{1/3}}\right),
	\qquad t>0.
\end{equation*}
Since $\operatorname{Ai}$ is entire, $x\mapsto K(x,t)$ is entire for each fixed $t>0$; and since $t\mapsto(3t)^{-1/3}\operatorname{Ai}(x(3t)^{-1/3})$ is, for fixed $x$, a composition of real-analytic functions away from $t=0$, it is real-analytic on $(0,\infty)$. Hence $K$ is jointly real-analytic on $\mathbb{R}\times(0,\infty)$: although $\delta_{0}$ is a distribution, singular at the origin, and lies into a quite rough Sobolev space, the solution is analytic in $(x, t)$ for every $t>0$, with only the prefactor $t^{-1/3}$ recording that $t=0$ itself is excluded. The same holds, by linearity, for any finite combination $u_{0}=\sum_{j=1}^{N}c_{j}\delta(\cdot-x_{j})$, since $u(x,t)=\sum_{j}c_{j}K(x-x_{j},t)$ remains entire in $x$ and real-analytic in $t>0$. These examples make the earlier distinction concrete: $K$ does not belong to any analytic Bourgain space uniformly down to $t=0$, nor does it arise from analytic, or even function-valued, initial data; its analyticity is not inherited from $u_{0}$ but created instantaneously by the dispersive dynamics. It is exactly this type of conclusion, pointwise analyticity for $t \neq 0$ generated dynamically from singular, low-regularity data, that we now establish for the nonlinear mKdV and cubic gKdV equations, where \eqref{sol1} is no longer available.

A substantial well-posedness theory underlies the thresholds appearing in Theorems~\ref{main1} and \ref{main2} below  sharp in the case of Theorem \ref{main1}, and approaching the critical exponent in the case of Theorem \ref{main2}. Refining the bilinear-estimate methods of Bourgain \cite{B1993a,B1993b}, Kenig, Ponce and Vega established local well-posedness for KdV in $H^{s}(\mathbb{R})$, $s>-\tfrac{3}{4}$ \cite{KPV2}, later extended to global well-posedness for $s>-\tfrac{3}{4}$ by Kishimoto \cite{K2010}, by Colliander, Keel, Staffilani, Takaoka and Tao \cite{CKSTT2003}, and by Guo \cite{Guo2009}, and shown to be sharp by Kenig, Ponce and Vega \cite{KPV2001}; well-posedness at $H^{-1}(\mathbb{R})$ was subsequently obtained by Buckmaster and Koch \cite{KB2015} and, in the full Hadamard sense using the integrable structure, by Killip and Vişan \cite{KV2019}, building on the earlier periodic result of Kappeler and Topalov \cite{KT2006}, and more recently extended below the $H^{-1}$-scale by Chapouto, Correia and Ramos \cite{CCR2026}. For mKdV, global well-posedness holds in $H^{s}(\mathbb{R})$ for $s>\tfrac{1}{4}$ by Colliander, Keel, Staffilani, Takaoka and Tao \cite{CKSTT2003}, recently lowered to the scaling-critical exponent $s>-\tfrac{1}{2}$ by Harrop-Griffiths, Killip and Vişan \cite{HKV2024}. On the Fourier-Lebesgue scale $\widehat{H}^{r}_{s}(\mathbb{R})$, $r\in(1,2]$ (which reduces to $H^{s}(\mathbb{R})$ when $r=2$, with norm $\|f\|_{\widehat{H}^{r}_{s}}:=\|\langle\xi\rangle^{s}\widehat{f}\|_{L^{r'}_{\xi}}$ recalled in Section~\ref{sec:notation}), local well-posedness of mKdV essentially down to $s\geq\tfrac{1}{2}-\tfrac{1}{2r}$ was established by Grünrock \cite{Gru2004-1} and by Grünrock and Vega \cite{GV2009}. For the (non-integrable) cubic gKdV equation, Grünrock \cite{Gru2005} obtained local well-posedness in $H^{s}(\mathbb{R})$ down to a threshold arbitrarily close to $s=-\tfrac{1}{6}$, via a bilinear Airy estimate adapted to the gKdV-3 non-linearity; this is exactly the regularity at which Theorem~\ref{main2} below establishes analytic smoothing. The local well-posedness threshold $s>-1/6$ established by Grünrock~\cite{Gru2005} was later shown by Tao~\cite{Tao2007} to extend to global well-posedness in the critical space $\dot{H}^{-1/6}(\mathbb{R})$ for small data, while Grünrock, Panthee and Drumond Silva \cite{GPS2007} obtained global well-posedness below $L^{2}(\mathbb{R})$ for the same equation.

\emph{We emphasize that our results are entirely local in time and do not draw on the completely integrable structure available for $m=1,2$: the mechanism is the local dispersive smoothing of the linear Airy flow together with the commutation properties of $P$, and the same conclusion is obtained for the non-integrable cubic equation. We therefore do not otherwise engage with the long-time, integrable-systems side of the theory (inverse scattering, soliton resolution, or breather asymptotic; see e.g. \cite{AC1991,AS1981,DZ1994,Tao2009,MM2001,DHKM2017,MM2008a,MM2008b,CL2021,DT1979}), beyond noting it here as the broader context in which the local regularity question of this paper sits.}

Writing $\varphi(x):=\varphi_{1,m}(x)$, so that $\varphi_{c,m}(x)=c^{1/m}\varphi(\sqrt{c}\,x)$, a change of variables in Fourier space gives
\begin{equation*}
	\| \varphi_{c,m} \|_{\dot{H}^{s}}^{2} = c^{s - s_{m}} \| \varphi \|_{\dot{H}^{s}}^{2},
	\qquad s_{m} := \frac{1}{2} - \frac{2}{m},
\end{equation*}
so that $\|\varphi_{c,m}\|_{\dot{H}^{s_{m}}}$ is independent of $c$, while $\|\varphi_{c,m}\|_{\dot{H}^{s}}\to0$ as $c\to0$ or $c\to\infty$ whenever $s\neq s_{m}$. The exponent $s_{m}$ is the scaling-critical Sobolev index for non-linearity $u^{m+1}$; for $m=2$ it equals $-\frac{1}{2}$, and it reappears below as the threshold in the sharpness remark following Theorem \ref{main1}.

We ask for conditions on $u_{0}$ guaranteeing that the corresponding solution is analytic jointly in space and time. Our first main result shows that the Nelson-type analytic vector condition of the preceding discussion, imposed on the Fourier-Lebesgue scale $\widehat{H}^{r}_{s}$, suffices.
\begin{mainthm}\label{main1}
	Let $r \in (1, 2]$ and $s \geq s(r):=\frac{1}{2} - \frac{1}{2r}$. Suppose the initial datum $u_{0}$ satisfies
	\begin{equation}\label{dec1}
		\sum_{k=0}^{\infty} \frac{\alpha^{k}}{k!} \big\| \left(x\partial_{x}\right)^{k} u_{0} \big\|_{\widehat{H}^{r}_{s}(\mathbb{R})} < \infty,
	\end{equation}
	for some constant $\alpha > 0$. Then, for any $b > \frac{1}{r}$, there exists $T>0$ such that the modified KdV equation admits a unique local-in-time solution
	\begin{equation*}
		u \in C\left(\big[-T,T\big]; \widehat{H}^{r}_{s}(\mathbb{R})\right) \cap X^{r}_{s,b},
	\end{equation*}
	where $X^{r}_{s,b}$ is the modified Bourgain space defined in Section~\ref{sec:notation} below, with continuous dependence on the data. Moreover, $u$ is real analytic in both space and time on $\big(\big[-T, T\big] \setminus \{0\}\big) \times \mathbb{R}$.
\end{mainthm}
\begin{obs}
	The two regimes covered by Theorem \ref{main1} correspond to the two
	sharp well-posedness results currently known for mKdV. On the classical
	Sobolev scale, the best local well-posedness result is due to Kenig,
	Ponce, and Vega \cite{KPV3}, at the scaling-critical regularity
	$s\geq\tfrac14$. On the Fourier-Lebesgue scale $\widehat H^r_s(\mathbb
	R)$, $r\in(1,2)$  a genuinely different refinement of $H^s(\mathbb
	R)$, \emph{not} a subspace of it, obtained by measuring $L^{r'}$ rather
	than $L^2$ integrability of the Fourier transform; the sharp local
	well-posedness threshold $s\geq\tfrac12-\tfrac1{2r}$ is due to Grünrock
	\cite{Gru2004-1} and Grünrock and Vega \cite{GV2009}. Theorem
	\ref{main1} establishes real analyticity on both scales simultaneously,
	through the single hypothesis \eqref{dec1}: as a byproduct, setting
	$r=2$ recovers analyticity precisely on the best space of
	well-posedness for mKdV proved by Kenig, Ponce, and Vega,
	\cite{KPV3,CKSTT2003}. In every case, the initial datum needs only
	exhibit mild decay at infinity, even within these rough function
	spaces; the corresponding solution nonetheless becomes analytic
	jointly in space and time, with infinite speed of propagation 
	(instantaneously), for every $t\neq0,$  the same phenomenon we
	establish for the cubic equation in Theorem \ref{main2}. We note that
	this sharpness is inherited from the sharpness of the underlying
	well-posedness theory; whether the analytic smoothing effect itself
	could persist below this threshold, if local well-posedness were
	eventually established there by other means, is not addressed by our
	argument.
\end{obs}
The reach of this hypothesis goes well beyond pointwise singularities of the Kato-Ogawa type, as the following example shows.
\begin{exe}
	We exhibit initial data for which analyticity is triggered. The following example intends to provide a family of examples from which the analyticity is generated; for such an end, we consider $\phi \in \mathcal{S}(\mathbb{R})$. Thus, it is straightforward to show that the function $u_{0}(x) = \operatorname{sgn}(x)\,\phi(x)$ satisfies
	\begin{equation*}
		(x\partial_{x})^{k}u_{0}(x) = \operatorname{sgn}(x)\,(x\partial_{x})^{k}\phi(x) \quad \text{in } \mathcal{S}'(\mathbb{R}), \qquad \text{for all } k \in \mathbb{N}.
	\end{equation*}
	Notice that, for each $k \in \mathbb{N}$, there exists a function $\psi_{k} \in \mathcal{S}(\mathbb{R})$ such that
	\begin{equation*}
		\mathcal{F}_{x}\big(\psi_{k}\big)(x) = (x\partial_{x})^{k}\phi(x), \qquad k \in \mathbb{N}.
	\end{equation*}
	Thus, $(x \partial_x)^{k}u_{0}$ can be rewritten as
	\begin{equation*}
	(x \partial_x)^{k}u_{0}(x) = i\,\mathcal{F}_{x}\big(\mathcal{H}(\psi_{k})\big)(x),
	\end{equation*}
	where $\mathcal{H}$ denotes the Hilbert transform\footnote{The \emph{Hilbert transform} of $f$ is defined as
		\begin{equation*}
			\mathcal{H}(f)(x) = \frac{1}{\pi}\,\mathrm{p.v.}\!\int_{\mathbb{R}} \frac{f(y)}{x-y}\,dy.
		\end{equation*}
		With the normalization $\mathcal{F}_{x}(v)(\xi) := \dfrac{1}{\sqrt{2\pi}} \displaystyle\int_{\mathbb{R}} e^{-ix\xi}\,v(x)\,dx$, this corresponds to the Fourier multiplier
		\begin{equation*}
			\mathcal{F}_{x}\big(\mathcal{H}(f)\big)(\xi) = -i\,\operatorname{sgn}(\xi)\,\mathcal{F}_{x}(f)(\xi).
	\end{equation*}}
	and $\mathcal{F}_{x}$ denotes the inverse Fourier transform (see Section~\ref{sec:notation} for the notation).
	
	Thus, for $r \in (1,2]$, $s \geq s(r)$, and $k \in \mathbb{N}$, restricting ourselves in what follows to the range $s \geq s(r)$ with $0 < s < \tfrac{1}{2}$, we obtain
\begin{equation}\label{controlfactorial}
	\begin{split}
		\left\|(x\partial_{x})^{k}u_{0}\right\|_{\widehat{H}^{r}_{s}(\mathbb{R})}
		&= \left\|\langle \xi \rangle^{s}\,\mathcal{H}\big(\psi_{k}\big)\right\|_{L^{r'}_{\xi}} \\
		&\lesssim \left\|\langle \xi \rangle^{s}\,\psi_{k}\right\|_{L^{r'}_{\xi}} \\
		&\sim \left\|\mathcal{F}_{x}\left(\langle \nabla_{x} \rangle^{s}\big((x\partial_{x})^{k}\phi\big)\right)\right\|_{L^{r'}_{\xi}}\\
		&\lesssim \left\|\langle \nabla_{x} \rangle^{s}\big((x\partial_{x})^{k}\phi\big)\right\|_{L^{r}_{x}} \\
		&< \infty,
	\end{split}
\end{equation}
The second estimate follows from the boundedness of the Hilbert transform $\mathcal{H}$ on the weighted space $L^{r'}_{\xi}(\langle\xi\rangle^{sr'}\,d\xi)$: by the classical $A_p$-weight theory of Hunt, Muckenhoupt, and Wheeden \cite{HMW1973}, the weight $\langle\xi\rangle^{sr'}$ belongs to the Muckenhoupt class $A_{r'}$ whenever $s\in\left(-\frac{1}{r'},\frac{1}{r'}\right)$; in our setting this translates into the range $0<s(r)=\frac{1}{2r'}\leq s<\frac{1}{r'}$ stated above. The third estimate follows from the Hausdorff-Young inequality, since $r\in(1,2]$ and its conjugate exponent $r'\in[2,\infty)$. Finally, the last bound holds because $(x\partial_{x})^{k}\phi\in\mathcal{S}(\mathbb{R})$ for every $k\in\mathbb{N}$, so that $\langle\nabla_{x}\rangle^{s}\big((x\partial_{x})^{k}\phi\big)\in\mathcal{S}(\mathbb{R})\subset L^{r}(\mathbb{R})$.

	To control the right-hand side uniformly in $k$ in \eqref{controlfactorial}, we take $s=1$ and assume that $\phi \in \mathcal{S}(\mathbb{R})$ satisfies a Gevrey-type bound: there exist constants $c_{0}, a > 0$ and $\sigma \geq 1$ such that
	\begin{equation*}
		\left\|\partial_{x}^{j}(x\partial_{x})^{k}\phi\right\|_{L^{r}_{x}(\mathbb{R})} \leq c_{0}\,a^{k+j}\,(k!)^{\sigma}\,j!, \qquad k \in \mathbb{N},\ j = 0,1.
	\end{equation*}
	In our case
	\begin{equation*}
		\left\|\langle \nabla_{x} \rangle\big((x\partial_{x})^{k}\phi\big)\right\|_{L^{r}_{x}} \lesssim c_{0}\,a^{k+1}\,(k!)^{\sigma}\,(1+k) \lesssim c_{1}\,b^{k}\,(k!)^{\sigma},
	\end{equation*}
	for suitable constants $c_{1}, b > 0$ depending on $c_{0}, a$. Consequently,
	\begin{equation*}
		\left\|(x\partial_{x})^{k}u_{0}\right\|_{\widehat{H}^{r}_{1}(\mathbb{R})} \lesssim c_{1}\,b^{k}\,(k!)^{\sigma}, \qquad k \in \mathbb{N},
	\end{equation*}
	which is precisely the Gevrey-$\sigma$ type estimate needed to conclude the corresponding regularity of $u_{0}$; the case $\sigma = 1$ corresponds to (real) analyticity.
	
A concrete class satisfying the hypothesis ($\sigma=1$). The prototypical example is the Gaussian
	\begin{equation*}
		\phi(x) = e^{-x^{2}/2}.
	\end{equation*}
	More generally, any $\phi$ belonging to the Gelfand-Shilov space $S^{1/2}_{1/2}(\mathbb{R})$~\cite{GS1968}, characterized equivalently 
	\begin{equation*}
		S^{1/2}_{1/2}(\mathbb{R}) \;=\; \Big\{ \phi \in C^{\infty}(\mathbb{R}) : \ \exists\, c, a, b > 0 \ \text{such that} \ 
		\left\| x^{m}\partial_{x}^{n}\phi \right\|_{L^{2}_{x}} \leq c\,a^{m}b^{n}\,(m!)^{1/2}(n!)^{1/2}, \ \ \forall\, m,n \in \mathbb{N} \Big\},
	\end{equation*}
	satisfies the required bound with $\sigma=1$; the equivalence between the $L^{\infty}_{x}$ formulation of Gelfand and Shilov's original definition and the $L^{2}_{x}$ formulation above is due to Chung, Chung, and Kim~\cite{CCK1996}. Fix such $\phi$ and let $c, a, b>0$ be constants for which the bound above holds.

	In the following, for simplicity, we restrict ourselves to the case $r=2$. Expanding $(x\partial_{x})^{k}$ via the Leibniz-type identity
	\begin{equation*}
		(x\partial_{x})^{k} = \sum_{j=1}^{k} S(k,j)\, x^{j}\partial_{x}^{j},
	\end{equation*}
	where $S(k,j)$ denotes the Stirling numbers\footnote{The \emph{Stirling numbers of the second kind} $S(n,k)$ are defined by
		\begin{equation*}
			S(n,k) = \frac{1}{k!}\sum_{j=0}^{k}(-1)^{j}\binom{k}{j}(k-j)^{n}.
	\end{equation*}} of the second kind, and applying the $S^{1/2}_{1/2}$ bound directly in $L^{2}_{x}$ with $m=n=j\le k$, we obtain
	\begin{equation*}
		\left\| (x\partial_{x})^{k}\phi \right\|_{L^{2}_{x}} \leq \sum_{j=1}^{k} S(k,j)\,\left\| x^{j}\partial_{x}^{j}\phi \right\|_{L^{2}_{x}} \leq c\sum_{j=1}^{k} S(k,j)\,(ab)^{j}\,j!.
	\end{equation*}
	The key point is that $S(k,j)j!\leq j^{k}$, which is a purely combinatorial fact: the total number of functions $f:\{1,\dots,k\}\to\{1,\dots,j\}$, $1\leq j\leq k$, equals $j^{k}$, and the surjections among them number exactly $S(k,j)j!$. Hence
	\begin{equation*}
		\sum_{j=1}^{k} S(k,j)(ab)^{j}j! \leq \sum_{j=1}^{k} j^{k}(ab)^{j} \leq k^{k}\sum_{j=1}^{k}(ab)^{j} = k^{k}\frac{(ab)^{k}-ab}{ab-1}, \qquad (ab\neq 1).
	\end{equation*}
	Since $k^{k}\leq e^{k}k!$ for all $k\in\mathbb{N}$, we conclude that
	\begin{equation*}
		\sum_{j=1}^{k} S(k,j)(ab)^{j}j! \lesssim_{a,b} e^{k}(ab)^{k}k! \leq (eab)^{k} k!, \qquad \forall k\in \mathbb{N}.
	\end{equation*}
	Combining the above estimates yields
	\begin{equation*}
		\left\| (x\partial_{x})^{k}\phi \right\|_{L^{2}_{x}} \leq c\,(eab)^{k}\,k!, \qquad k \in \mathbb{N},
	\end{equation*}
	with $a,b,c$ independent of $k$. Consequently,
	\begin{equation*}
		\left\| \langle \nabla_{x}\rangle(x\partial_{x})^{k}\phi \right\|_{L^{2}_{x}} \leq c_{1}\,\widetilde{a}^{k}\,k!, \qquad k \in \mathbb{N},
	\end{equation*}
	for suitable constants $c_{1}, \widetilde{a} > 0$ depending on $c,a,b$ but not on $k$.
	
	Besides the particular example above, another class of examples which can also be considered is the following:
	\begin{itemize}
		\item[(i)] the Hermite functions $h_{n}(x) = e^{-x^{2}/2}H_{n}(x)$, for each fixed $n \in \mathbb{N}$ (here $H_n$ is the $n$-th Hermite polynomial);
		\item[(ii)] finite linear combinations $\phi(x) = \sum_{m=0}^{n} c_{m}\,h_{m}(x)$;
		\item[(iii)] more general functions of the form $\phi(x) = P(x)\,e^{-x^{2}/2}$, with $P$ a polynomial.
	\end{itemize}
	In the following, we present the graph of our initial $u_{0}$, see Figure~\ref{sgn-gaussian} below.
	\begin{figure}[H]
		\centering
		\begin{tikzpicture}[xscale=1.5, yscale=1.3]
			\fill[red!6] (-0.15,-2.0) rectangle (0.15,2.1);
			
			\draw[-stealth, thick] (-3.5,0) -- (3.5,0) node[right] {$x$};
			\draw[-stealth, thick] (0,-2.0) -- (0,2.5) node[above] {$u_{0}$};
			\node[below left] at (-0.1,0) {$0$};
			
			\draw[blue!70!black, thick, domain=0.03:3, samples=200, smooth, variable=\x]
			plot ({\x},{exp(-(\x)*(\x))});
			\draw[blue!70!black, thick, domain=-3:-0.03, samples=200, smooth, variable=\x]
			plot ({\x},{-exp(-(\x)*(\x))});
			
			\draw[dashed, gray!50] (-3.5,1) -- (3.5,1);
			\draw[dashed, gray!50] (-3.5,-1) -- (3.5,-1);
			\node[gray!70, font=\small, left] at (-3.5,1) {$1$};
			\node[gray!70, font=\small, left] at (-3.5,-1) {$-1$};
			
			\node[circle, draw=red!70!black, fill=white, thick, minimum size=4pt, inner sep=0pt] at (0,1) {};
			\node[circle, draw=red!70!black, fill=white, thick, minimum size=4pt, inner sep=0pt] at (0,-1) {};
			
			\node[align=center, font=\small] at (2,1.99) {jump discontinuity at $x=0$\\ };
			\draw[-stealth, gray!70] (0,1.45) -- (0,1.08);
			
			\node[align=center, font=\small] at (2.6,0.6) {smooth, rapidly \\ decaying away from $x=0$\\ };
			\draw[-stealth, gray!70] (2.3,0.4) -- (2.05,0.08);
			
			\node[align=center, font=\small] at (1.9,-1.75) {no analytic (or even continuous)\\ extension across $x=0$};
			\draw[-stealth, gray!70] (0.6,-1.45) -- (0.08,-0.05);
		\end{tikzpicture}
		\caption{The function $\operatorname{sgn}(x)e^{-x^2}$. Since analyticity requires at least continuity, this jump
			already precludes any analytic or even continuous extension of
			the function across the origin.}
		\label{sgn-gaussian}
	\end{figure}
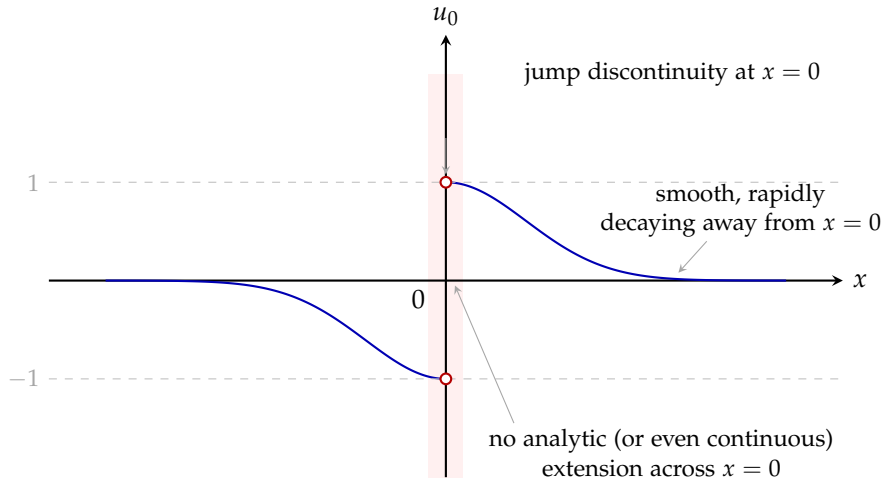
\end{exe}
The same smoothing phenomenon persists for the cubic gKdV equation, now on the classical Sobolev scale $H^s(\mathbb R)$, which corresponds to our second main result:
\begin{mainthm}\label{main2}
	Let $s >-\frac{1}{6}$. Suppose the initial data $u_{0}$ satisfy, for some $\alpha >0$,
	\begin{equation*}
		\sum_{ k = 0}^{\infty} \frac{\alpha^k}{k!} \big\| (x \partial_x)^{k} u_{0} \big\|_{H^s} < \infty. 
	\end{equation*}
	Then, for any $b > 1/2$, there exists a unique solution
	\begin{equation*}
		v \in C\left((-T,T); H^{s}(\mathbb{R})\right) \cap X^{s,b}
	\end{equation*}
	of the cubic gKdV equation on some interval $(-T,T)$, where $X^{s,b}$ denotes the classical Bourgain space, i.e.\ the case $r=2$ of $X^r_{s,b}$ defined in Section~\ref{sec:notation} below, depending continuously on the initial data. Moreover, $v$ is analytic at every point
	\begin{equation*}
		(t,x) \in \big[(-T, 0) \cup  (0,T)\big] \times \mathbb{R}. 
	\end{equation*}
\end{mainthm}
\begin{obs}
	For the cubic gKdV equation, our analyticity result holds for initial data
	in $H^s(\mathbb{R})$, $s>-\frac16$, approaching the critical regularity
	$\dot H^{-1/6}(\mathbb{R})$ where global well-posedness for small data
	was established by Tao \cite{Tao2007}. The endpoint $s=-\frac16$ itself
	remains outside the reach of our argument, and whether the analytic
	smoothing effect persists there is a question we do not address here
	and leave to a forthcoming work. Since the cubic equation is not
	completely integrable, this smoothing effect is not a consequence of
	any integrable structure, but rather of the dispersive mechanism
	described above.
\end{obs}
Proving Theorems \ref{main1} and \ref{main2} requires departing from the strategy of Kato and Ogawa \cite{KO2000} in two respects, both aimed at propagating the condition \eqref{dec1} along the nonlinear flow without the technical overhead a direct adaptation of \cite{KO2000} would carry. Commuting, or almost commuting, the operator $P=3t\partial_t+x\partial_{x}$ with the equation was achieved in \cite{KO2000} and \cite{BHK1995} through a pseudo-differential calculus adapted to the single-point singularity at $x=0$; this is effective for KdV, but becomes considerably more delicate for the modified and cubic nonlinearities treated here, whose relevant commutators involve products of several factors of the solution rather than a single one. We avoid pseudo-differential operators altogether, controlling the commutator between $P$ and the nonlinear term through elementary Leibniz-type identities for $(x\partial_{x})^k$ applied to products, combined with the dispersive linear estimates for the Airy propagator (in the spirit of the oscillatory integral bounds of \cite{KPV1991}) and the Kato-Ponce-type product estimates of \cite{GO2014}, which keeps the induction on $k$ used to propagate \eqref{dec1} self-contained. Separately, closing the estimates that propagate \eqref{dec1} requires controlling trilinear (for mKdV, in the spirit of \cite{KPV2}) and quadrilinear (for the cubic equation, building on the bilinear Airy estimate of \cite{Gru2005}) expressions built from $(x\partial_{x})^ku$ and its space-time derivatives. The standard route to such estimates in Bourgain spaces proceeds via dyadic Littlewood-Paley localization in frequency, which becomes increasingly cumbersome as the number of factors and the order $k$ grow. We instead establish these multilinear estimates through refinements based on Lorentz spaces $L^{p,q}$, exploiting the sharper embedding and interpolation properties of $L^{p,q}$ at the endpoints relevant to the k-gKdV non-linearity, following the classical interpolation-theoretic framework of Marcinkiewicz. This is not merely a change of technical convenience: propagating the Nelson-type condition \eqref{dec1} along the flow requires bounds on these multilinear expressions that are \emph{uniform in the order $k$} of differentiation, and dyadic localization becomes increasingly cumbersome as $k$ and the number of factors grow, whereas the sharper embedding and interpolation properties of $L^{p,q}$ at the relevant endpoints allow the estimates to be closed without explicit dyadic localization and, crucially, uniformly in $k$, exactly as required to sum the series \eqref{dec1}.

A further feature of this approach is that it treats the non-integrable cubic equation ($m=3$) with the same formalism used for the completely integrable mKdV equation ($m=2$): both are governed by the unified system (see \eqref{Pikm} below) for the iterated generator $v_{k} := P^{k} v$. Since the argument never makes use of the integrable structure available for $m=2$, but only of the dispersive smoothing of the linear Airy flow together with the commutation properties of $P$, this shows that the analytic smoothing effect established here is a structural, purely dispersive phenomenon, rather than a byproduct of integrability.

We close this introduction with a family of examples confirming that the class of data covered by Theorems \ref{main1} and \ref{main2} genuinely extends Kato-Ogawa's single-point-singularity condition to data singular at the origin that fail to belong to any classical space of analytic or Gevrey functions.
\begin{exe}
	This example exhibits a family of singular initial data in $H^{s}(\mathbb{R})$, for every $s\in(-\frac16,0)$, satisfying Nelson's condition \cite{Nelson1959} for the dilation generator $x\partial_{x}$  (see Figure \ref{x-plus-lambda} below ). The distributions in this family are generalized eigenfunctions of $x\partial_{x}$.
	
	Let $\lambda\in \big(-\frac{2}{3},-\frac{1}{2}\big)$, and define $x_{+}^{\lambda}\in\mathcal{S}'(\mathbb{R})$ by
	\begin{equation*}
		x_+^\lambda(x) := 
		\begin{cases}
			x^\lambda, & x > 0, \\
			0, & x \leq  0.
		\end{cases}
	\end{equation*}
	Its Fourier transform is
	\begin{equation*}
		\widehat{x_+^\lambda}(\xi) = -i\, e^{-i \lambda \pi/2}\,\Gamma(\lambda+1) \, (\xi - i0)^{-\lambda-1}, \qquad \xi\neq0,
	\end{equation*}
	where $(\xi-i0)^{-\lambda-1}:=\xi^{-\lambda-1}$ for $\xi>0$ and $(\xi-i0)^{-\lambda-1}:=|\xi|^{-\lambda-1}e^{i\pi(\lambda+1)}$ for $\xi<0$. One has $x_+^\lambda\in H^s(\mathbb R)$, with
	\begin{equation*}
		\|x_+^\lambda\|_{H^s} = \sqrt{2}\,\Gamma(\lambda+1)\, \left(\int_0^{\infty}(1+\xi^2)^s\,\xi^{-2 \lambda -2}\,d\xi \right)^{1/2} < \infty,
	\end{equation*}
	and, in the sense of tempered distributions, $(x \partial_x)^{k} x_+^\lambda = \lambda^{k}\, x_+^\lambda$, so that $\big\|(x\partial_x)^k x_+^\lambda\big\|_{H^s} = |\lambda|^k \,\big\|x_+^\lambda\big\|_{H^s}$ and, consequently,
	\begin{equation*}
		\sum_{k = 0}^{\infty} \frac{A_0^k}{k!} \big\| (x \partial_x)^k x_{+}^{\lambda}\big\|_{H^{s}(\mathbb{R})} = \|x_{+}^{\lambda}\|_{H^{s}(\mathbb{R})} \, e^{|\lambda| A_0} < \infty
	\end{equation*}
	for every $A_0>0$. The same holds for any finite linear combination $\beta_1 x_+^{\lambda_1}+\cdots+\beta_m x_+^{\lambda_m}$, $\lambda_j\in(-\frac23,-\frac12)$, $\beta_j\in\mathbb R$, since
	\begin{equation*}
		\left\|(x\partial_x)^k \Big(\sum_{j=1}^{m}\beta_j x_+^{\lambda_j}\Big)\right\|_{H^s(\mathbb{R})}
		\leq
		\Big(\max_{1\leq j \leq m}|\lambda_j|\Big)^k
		\sum_{j=1}^{m}|\beta_j|\,\big\|x_+^{\lambda_j}\big\|_{H^s(\mathbb{R})}.
	\end{equation*}
	A similar result also holds for $x_{-}^{\lambda}(x):=x_{+}^{\lambda}(-x)$; the proof requires only minor modifications.
	
	Theorem \ref{main2} therefore applies to $u_0=x_{+}^\lambda$, and to any such finite combination: although $u_0$ is singular at the origin, failing to be continuous, let alone smooth or analytic, at $x=0$, the corresponding cubic gKdV solution becomes instantaneously real analytic in $(t,x)$ for every $t\neq0$ in its interval of existence. This is a direct nonlinear counterpart of the instantaneous analytic smoothing exhibited by the linear Airy kernel above, and it strictly enlarges the class of admissible singular profiles beyond the single point singularity treated in \cite{KO2000}.
	
	This example isolates, in particularly sharp form, the sense in which the analyticity produced by Theorem \ref{main2} is genuinely a smoothing effect of the flow rather than a property carried over from the datum: $x_+^\lambda$ admits no analytic extension across $x=0$, so no analyticity of the solution could be anticipated from $u_0$ itself, and the real analyticity obtained for $t\neq0$ is encoded entirely in the decay, in $s$ and in $k$, of the evolving solution rather than in any regularity of $u_0$. Structurally, the example is best understood spectrally: $x_+^\lambda$ is a generalized eigenfunction of the dilation generator $x\partial_x$, with $(x\partial_x)^kx_+^\lambda=\lambda^kx_+^\lambda$ recording the corresponding eigenvalue $\lambda$. Nelson's condition \eqref{dec1} amounts precisely to requiring that $u_0$, decomposed along the (generalized) eigenfunctions of $x\partial_x$, be an analytic vector for this operator in the sense of \cite{Nelson1959}. Equivalently, writing $x=e^{y}$ for $x>0$, one has $x_+^\lambda=e^{\lambda y}$, so that being an eigenfunction of $x\partial_x$ is the same as being (entire) analytic in the logarithmic variable $y=\ln x$; the eigenfunction condition imposed here and the logarithmic-scale analyticity condition built directly into the initial datum of the preceding example are thus two instances of the same mechanism.
	\begin{figure}[H]
		\centering
		\begin{tikzpicture}[xscale=2.6, yscale=1.3]
			\fill[red!6] (0,0) rectangle (0.35,4.7);
			
			\draw[-stealth, thick] (-1.3,0) -- (3.4,0) node[right] {$x$};
			\draw[-stealth, thick] (0,-0.3) -- (0,4.7) node[above] {$x_+^{\lambda}(x)$};
			\node[below left] at (0,0) {$0$};
			
			\draw[blue!70!black, thick, domain=0.15:3, samples=200, smooth, variable=\x]
			plot ({\x},{(\x)^(-0.55)});
			
			\draw[blue!70!black, thick] (-1.3,0) -- (0,0);
			
			\draw[dashed, gray!70] (0,0) -- (0,4.7);
			\node[align=center, font=\small] at (1.15,3.4) {$x^{\lambda}_{+}\to\infty$ \\ as $x\to0^+$};
			\draw[-stealth, gray!80] (0.85,3.15) to[bend right=10] (0.32,2.55);
			
			\node[align=center, font=\small] at (2.5,-0.05) [below] {mild algebraic \\ decay as $x\to\infty$};
			\draw[-stealth, gray!80] (2.5,0.35) to[bend left=10] (2.15,0.62);
			
			\node[align=center, font=\small] at (-0.75,1.9) {$x_+^\lambda\equiv0$};
			\draw[-stealth, gray!80] (-0.75,1.55) to[bend left=10] (-0.4,0.08);
			
			\node[circle, draw=red!70!black, fill=white, thick, minimum size=4pt, inner sep=0pt] at (0,0) {};
		\end{tikzpicture}
		\caption{The tempered distribution $x_+^\lambda\in\mathcal S'(\mathbb R)$ for
			$\lambda\in\big(-\tfrac23,-\tfrac12\big)$: identically zero on $(-\infty,0]$,
			and equal to $x^\lambda$ on $(0,\infty)$, where it blows up at the origin
			while remaining locally integrable and decaying, in the mild
			algebraic sense, as $x\to\infty$. Although discontinuous indeed
			unbounded at $x=0$, the distribution lies in $H^s(\mathbb R)$ for
			suitable $s$, and is a generalized eigenfunction of the dilation
			generator $x\partial_x$, with eigenvalue $\lambda$.}
		\label{x-plus-lambda}
	\end{figure}
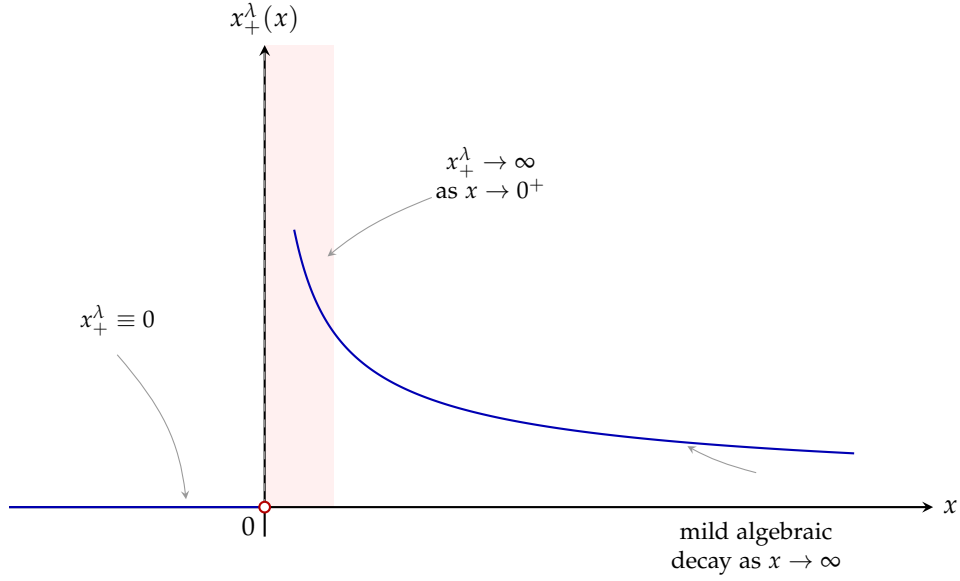
\end{exe}
	\subsection{Organization}
	This paper is organized as follows. 
	In Section~\ref{sec:notation}, we introduce our notation and the functional framework, including the Fourier-Lebesgue and modified Bourgain spaces $X^{r}_{s,b}$. 
	Section~\ref{sect2} presents the unified framework for the dilation operator $P = 3t\partial_t + x\partial_x$, establishing the system \eqref{SYS1} that tracks the growth of iterated derivatives. 
	The proof of Theorem \ref{main1} is detailed in Section \ref{sectA}, where we verify the uniform factorial bounds for the mKdV flow through a series of claims involving localized regularity bootstrap and trilinear estimates. 
	Finally, Section \ref{proofthmB} addresses the cubic gKdV equation, adapting the established framework to the $H^s$ scale. Technical estimates, including new trace theorems and commutator bounds in our modified spaces, are collected in Appendix \ref{apendiceA}.
	\subsection{Notation}\label{sec:notation}
	Throughout this work, $c$ denotes a generic positive constant that may change from line to line. Significant dependence on parameters is indicated by subscripts, such as $c_{t_{0}}$ or $c_{\epsilon}$. For two quantities $A$ and $B$, $A \lesssim B$ (or $A \lesssim_{t_{0}} B$) signifies $A \le c B$ (or $A \le c_{t_{0}} B$). We denote by $L^{2}(\mathbb{R}^{n})$ the \emph{standard Lebesgue space} of square-integrable functions, and for $s \in \mathbb{R}$, $H^{s}(\mathbb{R}^{n})$ represents the \emph{classical Sobolev space} with the norm $\|u\|_{H^{s}} = \|\langle \xi \rangle^{s} \hat{u}(\xi)\|_{L^{2}}$, where $\langle \xi \rangle = (1 + |\xi|^{2})^{1/2}$ and $\langle(\tau, \xi)\rangle = (1+\tau^{2}+\xi^{2})^{1/2}$ are the standard brackets. For any two operators $A$ and $B$, their \emph{commutator} is defined as $[A; B]: = AB - BA$.
	
	Central to our analysis is the \emph{Fourier transform} of a tempered distribution $v \in \mathcal{S}'(\mathbb{R})$. We consider the \emph{spatial Fourier transform} $\mathcal{F}_{x}$ for $v(x)$ as:
	$$\mathcal{F}_{x}(v)(\xi) :=\frac{1}{\sqrt{2\pi}} \int_{\mathbb{R}} e^{-i x \xi} v(x) \, dx,$$
	and the \emph{spacetime Fourier transform} $\mathcal{F}_{t,x}$ for $v(t,x)$, with respect to the dual variables $(\tau, \xi)$, as:
	$$\mathcal{F}_{t,x}(v)(\tau, \xi) := \frac{1}{2\pi}\int_{\mathbb{R}} \int_{\mathbb{R}} e^{-i(t\tau + x \xi)} v(t,x) \, dx \, dt.$$
	
	For $s \in \mathbb{R}$, we define the \emph{homogeneous fractional derivative operator} $|\nabla_x|^s$ (and its spacetime counterpart $|\nabla_{t,x}|^s$) via the Fourier multiplier $|\xi|^s$ (or $|(\tau, \xi)|^s$):
	$$\mathcal{F}_x(|\nabla_x|^s v)(\xi) := |\xi|^s \mathcal{F}_x(v)(\xi), \quad \text{and} \quad \mathcal{F}_{t,x}(|\nabla_{t,x}|^s v)(\tau, \xi) := |(\tau, \xi)|^s \mathcal{F}_{t,x}(v)(\tau, \xi),$$
	where $|(\tau, \xi)| = (\tau^2 + \xi^2)^{1/2}$. Consistent with this, the \emph{homogeneous Sobolev spaces} $\dot{H}^s(\mathbb{R})$ are defined by the norm $\|u\|_{\dot{H}^s} := \||\nabla_x|^s u\|_{L^2(\mathbb{R})}$.

	To study localization properties adapted to the linear part of the modified Korteweg-de Vries equation $\partial_{t} u + \partial_{x}^{3} u = 0$, we define for $ \sigma\in\mathbb{R}$, the \emph{spacetime Bessel potential operator} $\langle \nabla_{t,x}\rangle^{\sigma}$ as the Fourier multiplier:
	$$\langle \nabla_{t,x}\rangle^{\sigma} f := \mathcal{F}^{-1}_{t,x} \Big( \langle(\tau, \xi)\rangle^{\sigma} \mathcal{F}_{t,x}(f)(\tau, \xi) \Big).$$
	
	With these notations established, we define the functional framework necessary for our analysis, prioritizing the \emph{modified Fourier-Lebesgue} and \emph{Bourgain spaces}. For $s, b \in \mathbb{R}$ and $1 \le r \le \infty$ (with $r'$ as the \emph{conjugate exponent} satisfying $1/r + 1/r' = 1$), we consider the \emph{Fourier-Lebesgue Sobolev spaces} $\widehat{H}^{r}_{s}(\mathbb{R})$ and the \emph{anisotropic Fourier-Lebesgue Sobolev spaces} $\widehat{H}^{r}_{s, b}(\mathbb{R}^{2})$ as the spaces of tempered distributions equipped with the norms:
	$$\|f\|_{\widehat{H}^{r}_{s}(\mathbb{R})} := \| \langle\xi\rangle^{s} \mathcal{F}_{x}(f)(\xi) \|_{L^{r'}_{\xi}} \quad \text{and} \quad \|f\|_{\widehat{H}^{r}_{s, b}(\mathbb{R}^{2})} := \| \langle(\tau, \xi)\rangle^{b} \langle\xi\rangle^{s} \widehat{f}(\tau, \xi) \|_{L^{r'}_{\tau, \xi}}.$$
	Consistent with this, for $\mu \in \mathbb{R}$, we recognize $\widehat{H}^{r}_{\mu}(\mathbb{R}^{2})$ as the space of tempered distributions  equipped with the norm $\|\langle(\tau,\xi)\rangle^{\mu} \mathcal{F}_{t,x}(f)\|_{L^{r'}_{\tau\xi}}$. 
	
	These spaces, originally introduced by Grünrock and Vega in \cite{GV2009}, were specifically designed to exploit the dispersive properties of the equation in a more refined harmonic analysis setting. To account for the dispersive nature of the problem, we consider the \emph{modified Bourgain space} $X^{r}_{s,b}$, defined by the norm:
	\begin{equation*}
		\|f\|_{X^{r}_{s,b}} := \left( \int_{\mathbb{R}^{2}} \langle \xi \rangle^{sr'} \langle \tau - \xi^{3} \rangle^{br'} | \mathcal{F}_{t,x}(f)(\tau,\xi) |^{r'} \, \mathrm{d}\tau \, \mathrm{d}\xi \right)^{1/r'}
	\end{equation*}
	As established in \cite{GV2009}, these norms are well-defined across the full range of $r$:
	\begin{itemize}
		\item[(i)] For $r=1$ ($r'=\infty$), the norm characterizes the decay rate of the Fourier transform via the essential supremum.
		\item[(ii)] For $r=\infty$ ($r'=1$), the norm $\|f\|_{\widehat{H}^{\infty}_{s}(\mathbb{R})}$ reduces to the weighted $L^{1}$ norm $\|\langle\xi\rangle^{s}\mathcal{F}_{x}(f)\|_{L^{1}_{\xi}}$. 
		\item[(iii)] In the particular case $r=2$, we recover the standard Sobolev space $H^{s}(\mathbb{R})$ and the classical Bourgain space $X^{s,b}$.
	\end{itemize}
	Following the approach in \cite{GV2009}, we also use the time-restricted spaces $X^{r}_{s,b}(\delta)$ to facilitate the local-in-time theory. For any $\delta > 0$, we consider:
	\begin{equation*}
		X^{r}_{s,b}(\delta) := \left\{f = \tilde{f}|_{[-\delta,\delta]\times\mathbb{R}} : \tilde{f} \in X^{r}_{s,b} \right\}
	\end{equation*}
	equipped with the natural restriction norm
	\begin{equation*}
		\|f\|_{X^{r}_{s,b}(\delta)} := \inf \left\{ \|\tilde{f}\|_{X^{r}_{s,b}} : \tilde{f}|_{[-\delta,\delta]\times\mathbb{R}} = f \right\}.
	\end{equation*}
	To establish the analyticity of the flow, we define a sequence-based framework designed to characterize the growth of the Taylor coefficients. Let $\mathbf{f} = (f_k)_{k=0}^\infty$ be a sequence of tempered distributions where, for a given radius of analyticity $\rho \geq 0$, we monitor the convergence of the associated power series. Specifically, for $s \in \mathbb{R}$ and $1 < r \leq 2$, we define the \emph{spatial analytic space} $\mathcal{A}_{\rho}(\widehat{H}^{r}_{s})$ as the space of sequences where each $f_k \in \mathcal{S}'(\mathbb{R})$ is a spatial tempered distribution satisfying the norm condition
	\begin{equation*}
		\|\mathbf{f}\|_{\mathcal{A}_{\rho}(\widehat{H}^{r}_{s})} := \sum_{k=0}^{\infty} \frac{\rho^{k}}{k!} \|f_{k}\|_{\widehat{H}^{r}_{s}(\mathbb{R})} < \infty.
	\end{equation*}
	Similarly, to handle the dispersive evolution, we define the \emph{spacetime analytic space} $\mathcal{A}_{\rho}(X^{r}_{s,b})$ for $s, b \in \mathbb{R}$. This space consists of sequences of spacetime tempered distributions $f_k \in \mathcal{S}'(\mathbb{R}^2)$ acting on $\mathbb{R}_t \times \mathbb{R}_x$, equipped with the norm 
	\begin{equation*}
		\|\mathbf{f}\|_{\mathcal{A}_{\rho}(X_{s,b}^{r})} := \sum_{k=0}^{\infty} \frac{\rho^{k}}{k!} \|f_{k}\|_{X^{r}_{s,b}} < \infty.
	\end{equation*}
	In order to localize the suitable regions where we describe the analyticity of the solutions, we consider the following cut-off functions. Let $\chi \in C_{0}^{\infty}(\mathbb{R}^{2})$ be a smooth function such that $\chi(t, x) = 1$ for $|(t, x)| \leq 1$, $0 \leq \chi \leq 1$ and $\text{supp}(\chi) \subseteq \{(t, x) : |(t, x)| \leq 2\}$. 
	
	For a point $(t_{0}, x_{0}) \in \mathbb{R}^{2}$ and a vector of parameters $\vec{\boldsymbol{\epsilon}}= (\epsilon_{1}, \epsilon_{2}) \in (0, \infty)^{2}$, we define the \emph{localized function} $\chi_{(t_{0}, x_{0}), \vec{\boldsymbol{\epsilon}}}$ by:
	\begin{equation}\label{cutoff1}
		\chi_{(t_{0}, x_{0}), \vec{\boldsymbol{\epsilon}}}(t, x) := \chi \left( \frac{t - t_{0}}{\epsilon_{1}}, \frac{x - x_{0}}{\epsilon_{2}} \right).
	\end{equation}
	To decouple the variables in our directional estimates, we introduce corresponding 
	one-dimensional cut-off functions. Let $\chi_{1}, \chi_{2} \in C_{0}^{\infty}(\mathbb{R})$ 
	be standard bump functions satisfying $0 \le \chi_{i}(s) \le 1$ for $i \in \{1, 2\}$, 
	such that $\chi_{i}(s) = 1$ for $|s| \le 1$ and $\text{supp}(\chi_{i}) \subseteq \{s \in \mathbb{R} : |s| \le 2\}$. 
	
	Given arbitrary centers $t_{0}, x_{0} \in \mathbb{R}$ and positive scales $\epsilon_{1}, \epsilon_{2} > 0$, 
	we define the \emph{component-wise localized functions} for the time variable $t$ and space 
	variable $x$, respectively, by:
	\begin{equation}\label{cutoff1A}
		\chi_{1, t_{0}, \epsilon_{1}}(t) := \chi_{1} \left( \frac{t - t_{0}}{\epsilon_{1}} \right) 
		\quad \text{and} \quad 
		\chi_{2, x_{0}, \epsilon_{2}}(x) := \chi_{2} \left( \frac{x - x_{0}}{\epsilon_{2}} \right).
	\end{equation}
	\section{A Unified Framework for gKdV Analyticity}\label{sect2}
	In this section, we establish a common framework to treat the analyticity of solutions for the generalized Korteweg-de Vries (gKdV) equation of order $m\in \{3, 4\}$:
	\begin{equation}\label{mgkdv}
		\begin{cases} 
			\partial_{t} v + \partial_{x}^{3} v + \partial_{x} (v^{m}) = 0, \\ 
			v(0,x) = \phi_{0}(x), 
		\end{cases}
	\end{equation}
	where $m=3$ corresponds to the modified KdV, $m=4$ to the $3$-gKdV. To exploit the scaling symmetries of the equation, we consider the \emph{generator of dilation} $P$ associated with the \emph{linear operator of the KdV flow}, which we denote by $L.$ More precisely,
	\begin{equation*}
		P = x \partial_{x} + 3t \partial_{t}, \qquad L = \partial_{t} + \partial_{x}^3.
	\end{equation*}
	The fundamental commutation relations between these operators are:
	\begin{equation*}
		[L; P] = 3L, \qquad LP^{k}= (P+3I)^{k} L, \quad \text{and} \quad (P+3I)^{k} \partial_{x} = \partial_{x} (P+2I)^{k},\quad \forall \,\,k\in \mathbb{N}_{0}.
	\end{equation*}
	Given that $L v = -\partial_x (v^{m})$, we study the evolution of the iterated generator $v_{k} := P^{k} v$. Applying $L$ to $v_{k}$ yields:
	\begin{equation*}
		L v_{k} = (P+3I)^{k} L v = -\partial_{x} (P+2I)^{k} (v^{m}).
	\end{equation*}
	By applying the binomial expansion to $(P+2I)^k$ and the multinomial expansion to the $m$-th power of $v$, we collapse the nested sums into a single unified system:
	\begin{equation}\label{SYS1}
		\begin{cases} 
			\partial_{t} v_{k} + \partial_{x}^3 v_{k} + \Pi_k^{(m)}(v_0, v_1, \ldots, v_k) = 0, \\ 
		v_k(0,x) = \phi_{k}(x) = (x \partial_{x})^{k} \phi_{0}(x), \quad k = 0, 1, 2 , \ldots
		\end{cases}
	\end{equation}
	where $v_{0} := v$, and 
	\begin{equation}\label{Pikm}
		\Pi_{k}^{(m)}(v_0, v_1, \ldots, v_k) := \partial_{x} \left( \sum_{\substack{ k_{1} + \dots +k_{m+1} = k \\ 0\leq k_{1}\,k_{2},\dots,k_{m+1}\leq k }} \frac{k!\,2^{k_{1}}}{ k_{1}!k_{2}! \dots k_{m+1}!} \prod_{p=2}^{m+1} v_{k_{p}} \right).
	\end{equation}
	To establish the results asserted in Theorem \ref{main1}, we employ a strategy based on the local-in-time analysis of the flow. We organize the argument into several distinct stages, beginning with the fundamental local well-posedness result for the modified Korteweg-de Vries (mKdV) equation. This corresponds to the case $m=3$ in the generalized equation \eqref{mgkdv} and serves as the technical foundation for the estimates in more general settings.
	\begin{teorema}
		Let $1 < r \leq 2$, $s \geq s(r) \coloneqq \frac{1}{2} - \frac{1}{2r}$, and $u_0 \in \widehat{H}^r_s(\mathbb{R})$. Then there exist $b > \frac{1}{r}$, a time $\delta = \delta(\|u_0\|_{\widehat{H}^r_s}) > 0$, and a unique solution $u \in X_{s,b}^{r}(\delta)$ to \eqref{mgkdv} (with $m=3$). This solution is persistent, and for any $\delta_0 \in (0, \delta)$, the flow map 
		\begin{equation*}
			\begin{array}{rccc}
				\mathcal{W}\colon & \widehat{H}^r_s(\mathbb{R}) & \longrightarrow & X_{s,b}^{r}(\delta_0) \\
				& u_0 & \longmapsto & u
			\end{array}
		\end{equation*}
		is locally Lipschitz continuous.
	\end{teorema}
	\begin{proof}
		This result is a consequence of the bilinear estimates in Bourgain-type spaces adapted to the $\widehat{H}^r_s$ framework; for a comprehensive treatment, we refer the reader to \cite[Theorem 1]{GV2009}.
	\end{proof}
	We also present the linear estimates associated with the Airy group, distinguishing between classical $H^s$ and $X^{s,b}$ structures and the modified $\widehat{H}^r_s$ and $X^r_{s,b}$ spaces required for the Fourier-Lebesgue theory.
	\begin{lema}\label{Lemma6}
		Let $1 < r < \infty$ and $s \in \mathbb{R}$.
		\begin{enumerate}
			\item For any $b \in \mathbb{R}$ and any cut-off function $\psi \in C_{0}^{\infty}(\mathbb{R})$, the linear propagator $S(t)$ satisfies:
			\begin{equation*}
				\|\psi(t) S(t) u_{0}\|_{X^{r}_{s,b}} = \|\psi\|_{\widehat{H}^{r}_{b}} \|u_{0}\|_{\widehat{H}^{r}_{s}} \le c_{\psi} \|u_{0}\|_{\widehat{H}^{r}_{s}},
			\end{equation*}
			for any initial data $u_{0} \in \widehat{H}^{r}_{s}(\mathbb{R})$.
			\item  Assume $-\frac{1}{r'} < b' \leq 0 \leq b \leq 1+b'$. Then for any $\delta \in (0, 1]$, we have:
			\begin{equation*}
				\left\|\psi_{\delta} \int_{0}^{t} S(t-t') F(t') dt' \right\|_{X^{r}_{s,b}} \leq c \delta^{1+b'-b} \|F\|_{X^{r}_{s,b'}},
			\end{equation*}
			where $c$ is a constant independent of $\delta$.
		\end{enumerate}
	\end{lema}
	\begin{proof}
		See Gr\"unrock \cite{Gru2004-1}.
	\end{proof}
	The primary contribution of Grünrock and Vega in \cite{GV2009} lies in deriving  multi-linear estimates within the framework of the modified Bourgain spaces $X^{r}_{s,b}$. These estimates were deduced to push the well-posedness theory down to remarkably rough Sobolev spaces. In our analysis, we rely on these technical tools not only to establish the existence of solutions in low-regularity regimes but also to prove their analyticity.
	
	The core of this approach is summarized in the following multi-linear estimate, which handles the derivative non-linearity of the equation:
	\begin{teorema}\label{thm2}
		Let $1 < r \leq 2$ and $s \geq s(r) = \frac{1}{2} - \frac{1}{2r}$. Then, for any $b' < 0$ and $b > \frac{1}{r}$, the following estimate holds:
		\begin{equation}\label{est3}
			\left\| \partial_{x} \left( \prod_{j=1}^{3} u_{j} \right) \right\|_{X^{r}_{s,b'}} \leq c \prod_{j=1}^{3} \|u_{j}\|_{X^{r}_{s,b}}.
		\end{equation}
	\end{teorema}
	\begin{proof}
		See \cite{GV2009}.
	\end{proof}
	Additionally, we include the following remarks regarding the lifespan of solutions, as derived in \cite{GV2009}:
	\begin{obs}\label{Remark9} For functions $u_{1}, u_{2}, u_{3}$ supported in $[-\delta, \delta] \times \mathbb{R}$ with $0 < \delta \leq 1$, we have the estimate:
		\begin{equation*}
			\left\| \partial_{x} \left( \prod_{j=1}^{3} u_{j} \right) \right\|_{X^{r}_{s,b-1}} \leq c \delta^{1 - \frac{1}{r} - \epsilon} \prod_{j=1}^{3} \|u_{j}\|_{X^{r}_{s,b}},
		\end{equation*}
		provided $1 < r \leq 2$, $s \geq s(r)$, $b > \frac{1}{r}$, and $\epsilon > 0$. Incorporating this estimate into the fixed-point argument yields a lifespan of size:
		\begin{equation*}
			\delta \sim \|u_{0}\|_{\widehat{H}^{r}_{s}}^{-\left( \frac{2r}{r-1} \right) - \epsilon'}.
		\end{equation*}
		For the case $r = 2$, this coincides up to the factor $\epsilon'$ with the classical result established in \cite{KPV3} (see also \cite{FLP1999}).
	\end{obs}
	In what follows, we present the main ideas of the proof of Theorem \ref{main1}.
	\section{Proof of Theorem \ref{main1}}\label{sectA}
	\begin{proof}[Proof of Theorem \ref{main1}]
		We organize the proof into a series of claims that contain the essential information required to establish the analyticity of the solutions within these rough Fourier-Lebesgue spaces. 
        \setcounter{claim}{0}
	\begin{claim}\label{c0}
	Let $1 < r \leq 2$, $ \frac{1}{r} < b < 1$, and $s \geq s(r):= \frac{1}{2}\left(1 - \frac{1}{r}\right)$. Suppose there exists a constant $\alpha > 0$ such that the initial data $\boldsymbol{\phi} = (\phi_{k})_{k=0}^{\infty}$ for the system \eqref{SYS1} satisfies
	\begin{equation*}
		\|\boldsymbol{\phi}\|_{\mathcal{A}_{\alpha}\left(\widehat{H}^{r}_{s}(\mathbb{R})\right)} := \sum_{k=0}^{\infty} \frac{\alpha^k}{k!} \|\phi_k\|_{\widehat{H}^{r}_{s}(\mathbb{R})} < \infty.
	\end{equation*}
	Then, there exists a lifespan $T = T\left(\|\boldsymbol{\phi}\|_{ \mathcal{A}_{\alpha}\left(\widehat{H}^{r}_{s}(\mathbb{R})\right)}\right) > 0$ such that the system \eqref{SYS1} with $m = 3$ satisfies:
	\begin{enumerate}
		\item \underline{Existence and Uniqueness:} There exists a unique solution $\mathbf{v}(t) = (v_k(t))_{k \geq 0}$ belonging to the class
		\begin{equation*}
			\mathbf{v} \in C\left([-T, T]; \mathcal{A}_{\alpha}\left(\widehat{H}^{r}_{s}(\mathbb{R})\right)\right) \cap \mathcal{A}_{\alpha}(X^{r}_{s,b}(T)).
		\end{equation*}
		\item \underline{Uniform Bound:} The solution $\mathbf{v}$ satisfies the analytical estimate:
		\begin{equation*}
			\|\mathbf{v}\|_{\mathcal{A}_{\alpha}(X^{r}_{s,b}(T))} \leq c \|\boldsymbol{\phi}\|_{\mathcal{A}_{\alpha}\left(\widehat{H}^{r}_{s}(\mathbb{R})\right)}.
		\end{equation*}
 \item \underline{Lipschitz continuity:} for every $M>0$ there is $T(M)>0$ (as in Item~1) such that: for any $\boldsymbol\phi_1,\boldsymbol\phi_2 \in \mathcal A_{\alpha}(\widehat H_s^r(\mathbb R))$ with $\|\boldsymbol{\phi}_1\|_{\mathcal A_{\alpha}(\widehat H_s^r)},\|\boldsymbol{\phi}_2\|_{\mathcal A_{\alpha}(\widehat H_s^r)}\le M$, the corresponding solutions $\mathbf{v}_1,\mathbf{v}_2$ satisfy
\[
\|\mathbf{v}_1-\mathbf{v}_2\|_{C([-T(M),T(M)];\mathcal A_{\alpha}(\widehat H_s^r))}\le c_{T(M)}\|\boldsymbol{\phi}_1-\boldsymbol{\phi}_2\|_{\mathcal A_{\alpha}(\widehat H_s^r)}.
\]
\end{enumerate}
\end{claim}
		\begin{claimproof}
			For a sequence  $\mathbf v=(v_k)_{k\ge0}$ in the  Banach space $\A_{\alpha}\big(X^r_{s,b}\big)$, we define the operator  $\Phi(\mathbf{v}) := \bigl(\Phi_k(\mathbf{v})\bigr)_{k\geq 0}$, where  
\[
\Phi_k(\mathbf{v})(t):=\psi(t)S(t)\phi_k-\psi(t) \int_0^t  S(t-t')\; \psi_{\delta}(t') \;\Pi_k^{(3)}(v_0,\dots,v_k)(t')\,dt',
\]
and  $\psi \in C_{0}^{\infty}(\mathbb{R})$ is a smooth cut-off function with $\psi \equiv 1$ on $[-1,1]$ and $\supp (\psi) \subset [-2,2]$, and $\psi_{\delta}(t):= \psi(t/\delta)$. 

By Lemma \ref{Lemma6}, Theorem \ref{thm2} and Remark \ref{Remark9} we obtain, for every $k\ge0$,
\[
\|\Phi_k(\mathbf{v})\|_{X^r_{s,b}}\lesssim\|\phi_k\|_{\widehat{H}_s^{r}(\mathbb{R})}+\delta^{1-\frac1r-\epsilon}\!\!\sum_{k_1+k_2+k_3+k_4=k}\!\!2^{k_1}\binom{k}{k_1,k_2,k_3,k_4}\prod_{\ell=2}^{4}\|v_{k_\ell}\|_{X^r_{s,b}},
\]
valid for $1<r\le2$, $s\ge s(r)$, $b>\tfrac1r$, and any $\epsilon>0$. Multiplying by $\alpha^k/k!$, summing in $k$, we obtain

\begin{equation}\label{EqXXXX}
\|\Phi(\mathbf v)\|_{\A_{\alpha}(X^r_{s,b})}\le c_0\|\boldsymbol\phi\|_{\A_{\alpha}(\widehat{H}_{s}^{r})}+c_1e^{2\alpha}\delta^{1-\frac1r-\epsilon}\|\mathbf v\|_{\A_{\alpha}(X^r_{s,b})}^3.    
\end{equation}

\emph{Fixed point.} Given the initial data
$\boldsymbol\phi\in\A_{\alpha}(\widehat H_s^r)$, let $R':=\|\boldsymbol\phi\|_{\A_{\alpha}(\widehat H_s^r)}$, $R:= 2c_0 R'$ and choose
\[
0<\delta<\delta(R'):=\big( 24 c_0^2 c_1 \,e^{2\alpha}(R')^{2}\big)^{-1/(1-\frac1r-\epsilon)}.
\]
Let $B_R:=\{\mathbf v\in\A_{\alpha}(X^r_{s,b}):
\|\mathbf v\|_{\A_{\alpha}(X^r_{s,b})}\le R\}$.  To show that $\Phi(B_R)\subset B_R$, note that
$c_1 e^{2\alpha}\delta^{1-\frac1r-\epsilon}R^2\le\frac12$. Hence, by \eqref{EqXXXX},
\[
\begin{aligned}
\|\Phi(\mathbf v)\|_{\A_{\alpha}(X^r_{s,b})}
&\le c_0R'
   +c_1e^{2\alpha}\delta^{1-\frac{1}{r}-\epsilon}R^3 \\
&\le c_0R'
   +\frac{1}{2}R = R.
\end{aligned}
\]
 Furthermore, the same multilinear estimates imply
\[
\|\Phi(\mathbf v)-\Phi(\mathbf w)\|_{\A_{\alpha}(X^r_{s,b})}\le L\,\|\mathbf v-\mathbf w\|_{\A_{\alpha}(X^r_{s,b})},\qquad L:= 3c_1e^{2\alpha}\delta^{1-\frac1r-\epsilon}R^2<1,
\]
for $\mathbf v,\mathbf w\in B_R$. From the definition of $\delta(R')$, we have $\delta^{1-\frac1r-\epsilon}\sim (R')^{-2}$. Since $R=2c_0R'$, it follows that $R^2\sim (R')^2$, and therefore $\delta^{1-\frac1r-\epsilon}R^2\sim1$. Consequently,
\[
\delta(R')
\sim
(R')^{-\frac{2r}{r-1-r\epsilon}}
=
(R')^{-\frac{2r}{r-1}
-
\frac{2r^{2}\epsilon}{(r-1)(r-1-r\epsilon)}}.
\]
Since $\frac{2r^{2}\epsilon}{(r-1)(r-1-r\epsilon)}
\longrightarrow0 $ as $\epsilon\to0^+$, it follows that, for every $\epsilon'>0$, one can choose $\epsilon>0$ sufficiently small so that $\delta(R') \sim (R')^{-\frac{2r}{r-1}-\epsilon'}$, which agrees with the lifespan stated in Remark~\ref{Remark9}.

Banach's fixed point theorem gives a unique $\mathbf v\in B_R$ with $\Phi(\mathbf v)=\mathbf v$, and setting $T:=\delta$ yields the asserted lifespan. It is easy to see when $t \in [-\delta, \delta]$ , 
\[
v_{k}(t)= S(t)\phi_k-\int_0^t  S(t-t')\; \;\Pi_k^{(3)}(v_0,\dots,v_k)(t')\,dt',
\]
hence $v_k$ is a solution to  the system \eqref{SYS1}
for $k = 0, 1, 2, \ldots$ and $v_k \in X_{s,b}^{r}(\delta)$. Now, since  $\mathbf v\in B_R$, we have  $\|\mathbf v\|_{\A_{\alpha}(X^r_{s,b}(\delta))}\le R=2c_0R'=2c_0\|\boldsymbol\phi\|_{\A_{\alpha}(\widehat{H}_s^{r})}$. Moreover, uniqueness extends to the whole space
$\A_{\alpha}(X^r_{s,b}(\delta))$ by adapting the standard argument in \cite[pp.~380--382]{BOP1998}.

\emph{Lipschitz continuity.}
Let $M:=\max\!\left(
\|\boldsymbol{\phi}_1\|_{\A_{\alpha}(\widehat H_s^r)},
\|\boldsymbol{\phi}_2\|_{\A_{\alpha}(\widehat H_s^r)}
\right)$. Apply the fixed-point construction with this bound \(M\). The corresponding lifespan
\(\delta=\delta(M)\) and radius $R:=2c_0M$ ensure that the associated solutions \(\mathbf{v}_1\) and \(\mathbf{v}_2\) both belong to the same ball
\(B_R\subset\A_{\alpha}(X^r_{s,b}(\delta))\) on the same interval
\([-\delta,\delta]\). Write
\[
\mathbf{v}_1 -\mathbf{v}_2 
=
\big[\Phi_{\boldsymbol{\phi}_1}(\mathbf{v}_1)-\Phi_{\boldsymbol{\phi}_1}(\mathbf{v}_2)\big]
+
\big[\Phi_{\boldsymbol{\phi}_1}(\mathbf{v}_2)-\Phi_{\boldsymbol{\phi}_2}(\mathbf{v}_2)\big]
=: I+II.
\]
By the contraction estimate, we have $  \|I\|_{\A_{\alpha}(X^r_{s,b}(\delta))}\le L \cdot \|\mathbf{v}_1-\mathbf{v}_2\|_{\A_{\alpha}(X^r_{s,b}(\delta))}$, while 
\[
 \|II\|_{\A_{\alpha}(X^r_{s,b}(\delta))} = \| \psi(t)S(t)(\boldsymbol{\phi}_1-\boldsymbol{\phi}_2) \|_{\A_{\alpha}(X^r_{s,b}(\delta))}\le c_0\|\boldsymbol{\phi}_1-\boldsymbol{\phi}_2\|_{\A_{\alpha}(\widehat{H}_s^{r})}
\]
by Lemma \ref{Lemma6}, part (1). Since $L<1$,
\[
\|\mathbf{v}_1 -\mathbf{v}_2 \|_{\A_{\alpha}(X^r_{s,b}(\delta))}\le\frac{c_0}{1-L}\|\boldsymbol{\phi}_1-\boldsymbol{\phi}_2\|_{\A_{\alpha}(\widehat{H}_s^{r})},
\]
and the continuous embedding $X^r_{s,b}(\delta)\hookrightarrow C([-\delta,\delta];\widehat{H}_s^{r})$ (valid for $b>1/r\ge1/2$), summed termwise with weights $\alpha^k/k!$, yields 
\[
\|\mathbf{v}_1 -\mathbf{v}_2\|_{C([-\delta,\delta];\A_{\alpha}(\widehat{H}_s^{r}))} \leq \frac{c_0 C_b}{1-L} \|\boldsymbol{\phi}_1-\boldsymbol{\phi}_2\|_{\A_{\alpha}(\widehat{H}_s^{r})}
\]
 where  $C_b$ denotes the constant from the embedding; this completes the proof of Claim 1. 
\end{claimproof}
	
\setcounter{claim}{1}	
		\begin{claim}\label{c2.1}
			Let $b$ and $r$ be numbers such that $1 < r \leq 2$, then 
			\begin{equation*}
				\|v_{k}\|_{\widehat{L}^r_{xt}(\mathbb{R}^2)}  \leq c \alpha^{k} k!, \quad \text{for } k = 0, 1, 2, \dots.
			\end{equation*}
		\end{claim}
		\begin{claimproof}
			By the definition of the Fourier-Lebesgue and Bourgain norms, we have for each $k \in \mathbb{N}_0$:
			\begin{equation*}
				\|v_{k}\|_{\widehat{L}^r_{xt}(\mathbb{R}^2)} \leq \left\| \langle \tau-\xi^{3}\rangle^{b} \mathcal{F}_{t,x}(v_{k})(\tau,\xi) \right\|_{L^{r'}_{\tau\xi}} = \|v_{k}\|_{X^{r}_{0,b}},
			\end{equation*}
			Using the estimate obtained in {\sc claim }\ref{c0}, it follows that
			$
			\|v_{k}\|_{X^{r}_{0,b}} \leq \alpha^{k} k!, \quad k \in \mathbb{N}_0,$  which finish the proof.
		\end{claimproof}
        \setcounter{claim}{2}
		\begin{claim}\label{c2.2}
			Let $b$ and $r$ be numbers such that $1 < r \leq 2$ and $b > \frac{1}{r}$. Then 
			\begin{equation*}
				\|v_{k}\|_{L^{8}_{xt}} \leq c \alpha^{k} k!, \quad \text{for } k = 0, 1, 2, \dots.
			\end{equation*}
		\end{claim}
		\begin{claimproof}
			Notice that by Theorem \ref{kpvstrichartz} in Appendix \ref{apendiceA}, the group $S(t) = e^{-t\partial_{x}^{3}}$ associated to the Airy equation satisfies:
			\begin{equation*}
				\|S(t)f\|_{L^{8}_{xt}} \leq c \|f\|_{L^{2}_{x}}, \quad \text{for any } f \in L^{2}(\mathbb{R}),
			\end{equation*}
			for some positive constant $c$. Now, by Plancherel's Theorem, it implies that 
			\begin{equation*}
				\|S(t)f\|_{L^{8}_{xt}} \leq c \|f\|_{\widehat{L}^{2}_{x}}, \quad \text{for any } f \in L^{2}(\mathbb{R}).
			\end{equation*}
			Thus, by Lemma \ref{stabilitylinfty} in Appendix \ref{apendiceA}, we obtain that for any $b > \frac{1}{r}$, it holds that
			\begin{equation*}
				\|f\|_{L^{8}_{xt}} \leq c(b) \|f\|_{X^{r}_{0,b}}, \quad \text{whenever } f \in X^{r}_{0,b}.
			\end{equation*}
			In particular, by the result proved in {\sc claim }\ref{c0}, it follows that each $P^k v$ satisfies the bound 
			\begin{equation*}
				\|v_{k}\|_{L^{8}_{xt}} \leq c(b) \|v_{k}\|_{X^{r}_{0,b}} \leq c \alpha^{k} k!, \quad k \in \mathbb{N}_{0}.
			\end{equation*}
		\end{claimproof}
		\begin{obs}
			For the sake of simplicity in the notation, we shall henceforth write $\Pi_{k}^{(m)}(v)$ instead of
			$\Pi_{k}^{(m)}(v_0,v_1,\ldots,v_k)$, whenever no ambiguity arises.
		\end{obs}
		\setcounter{claim}{3}
		\begin{claim} \label{c3}
			Let $\vec{\boldsymbol{\epsilon}} = (\epsilon, \epsilon) \in (0, \infty)^{2}$ and let $\chi_{(t_{0},x_{0}),\boldsymbol{\vec{\epsilon}}}$ be the cutoff function defined in \eqref{cutoff1}. For $1 < r \leq 2$ with $\frac{1}{r}+\frac{1}{r'}=1$, there exist constants $c > 0,$  such that the following estimate holds for all $k \in \mathbb{N}_0$:
			\begin{equation*}
				\left\| \chi_{(t_{0},x_{0}),\boldsymbol{\vec{\epsilon}}} \,v_{k} \right\|_{\widehat{H}^{r}_{1}(\mathbb{R}^{2})} \leq c \alpha^{k} k!.
			\end{equation*}
		\end{claim}
		\begin{claimproof}
			Applying Lemma \ref{interpo1}, we obtain the following inequality:
			\begin{equation*}
				\begin{aligned}
					\| \chi_{(t_{0},x_{0}),\boldsymbol{\vec{\epsilon}}} \, v_{k} \|_{\widehat{H}^r_{1}} \lesssim_{t_0, x_0} \,\, &\| \chi_{(t_{0},x_{0}),\boldsymbol{\vec{\epsilon}}} \, v_{k} \|_{\widehat{H}^r_{-2}(\mathbb{R}^{2})} 
					+ \left\| t\partial_x^3 \left( \chi_{(t_{0},x_{0}),\boldsymbol{\vec{\epsilon}}} \, v_{k} \right) \right\|_{\widehat{H}^r_{-2}(\mathbb{R}^{2})} \\
					&+ \left\| P^{3} \left( \chi_{(t_{0},x_{0}),\boldsymbol{\vec{\epsilon}}} \, v_{k} \right) \right\|_{\widehat{H}^r_{-2}(\mathbb{R}^{2})} \\
					=: \,\, &\Xi_{1} + \Xi_{2} + \Xi_{3}.
				\end{aligned}
			\end{equation*}
			We now estimate each term individually. For the first term, using Lemma \ref{bach1}, we have:
			\begin{equation*}
				\begin{aligned}
					\Xi_{1} &= \| \chi_{(t_{0},x_{0}),\vec{\boldsymbol{\epsilon}}} \, v_{k} \|_{\widehat{H}^r_{-2}(\mathbb{R}^{2})} \\
					&\lesssim \| \chi_{(t_{0},x_{0}),\vec{\boldsymbol{\epsilon}}}\|_{\widehat H^{\infty}_{2}(\mathbb R^{2})} \|v_{k}\|_{\widehat H^{r}_{-2}(\mathbb R^{2})} \\
					&\lesssim \| v_{k} \|_{\widehat{L}^r_{xt}(\mathbb{R}^2)} \\
					&\leq c \alpha^{k} k!.
				\end{aligned}
			\end{equation*}
			The penultimate inequality holds due to the continuous embedding $\widehat{L}^r_{xt}(\mathbb{R}^2) \hookrightarrow \widehat H^{r}_{-2}(\mathbb R^{2})$, and the final bound follows from \textbf{{\sc Claim}} \ref{c2.1}.
			
			Next, we handle $\Xi_{2}$. 	In this sense, we first rewrite 
			\begin{equation}\label{identity2}
				\begin{split}
					t\partial_{x}^{3}\left(\chi_{(t_{0},x_{0}),\boldsymbol{\vec{\epsilon}}} \,v_{k} \right)&=t\chi_{(t_{0},x_{0}),\boldsymbol{\vec{\epsilon}}} \, \partial_{x}^{3}v_{k}  +\partial_{x}^{2}\left(3t\partial_{x}\chi_{(t_{0},x_{0}),\boldsymbol{\vec{\epsilon}}} \,v_{k} \right)
					-\partial_{x}\left(3t\partial_{x}^{2}\chi_{(t_{0},x_{0}),\boldsymbol{\vec{\epsilon}}} \,v_{k} \right)\\
					&\quad +t\partial_{x}^{3}\chi_{(t_{0},x_{0}),\boldsymbol{\vec{\epsilon}}} \,v_{k}
				\end{split}
			\end{equation}
			Thus,
			\begin{equation*}
				\begin{split}
					\Xi_{2}&\leq  \left\|t\chi_{(t_{0},x_{0}),\boldsymbol{\vec{\epsilon}}} \, \partial_{x}^{3}v_{k}\right\|_{\widehat{H}^r_{-2}(\mathbb{R}^{2})}+
					\left\|\partial_{x}^{2}\left(3t\partial_{x}\chi_{(t_{0},x_{0}),\boldsymbol{\vec{\epsilon}}} \,v_{k} \right)\right\|_{\widehat{H}^r_{-2}(\mathbb{R}^{2})}
					\\
					&\quad+\left\|\partial_{x}\left(3t\partial_{x}^{2}\chi_{(t_{0},x_{0}),\boldsymbol{\vec{\epsilon}}} \,v_{k} \right)\right\|_{\widehat{H}^r_{-2}(\mathbb{R}^{2})} +\left\|t\partial_{x}^{3}\chi_{(t_{0},x_{0}),\boldsymbol{\vec{\epsilon}}} \,v_{k}\right\|_{\widehat{H}^r_{-2}(\mathbb{R}^{2})}\\
					&=\Xi_{2,1}+\Xi_{2,2}+\Xi_{2,3}+\Xi_{2,4}.
				\end{split}
			\end{equation*}
			Regarding $\Xi_{2,1}$, we use the following identity:
			\begin{equation}\label{identity1}
				t \chi_{(t_{0},x_{0}),\boldsymbol{\vec{\epsilon}}} \partial_{x}^{3}v_{k} = -t \chi_{(t_{0},x_{0}),\boldsymbol{\vec{\epsilon}}} \Pi_{k}^{(3)}(v) - \frac{\chi_{(t_{0},x_{0}),\boldsymbol{\vec{\epsilon}}}}{3} Pv_{k} + \frac{x \chi_{(t_{0},x_{0}),\boldsymbol{\vec{\epsilon}}}}{3} \partial_{x}v_{k},
			\end{equation}
			for all $k \in \mathbb{N}_{0}$. Here, we recall that $\Pi_{k}^{(3)}(v)$ is defined as:
			\begin{equation*}
				\Pi_{k}^{(3)}(v) = \partial_{x} \left( \sum_{\substack{k=k_{1}+k_{2}+k_{3}+k_{4}\\ 0 \leq k_{1},k_{2},k_{3},k_{4} \leq k}} \frac{k! 2^{k_{1}}}{k_{1}!k_{2}!k_{3}!k_{4}!} \, v_{k_{2}}v_{k_{3}}v_{k_{4}} \right), \quad k \in \mathbb{N}_0.
			\end{equation*}
			Consequently, we obtain the estimate:
			\begin{equation*}
				\begin{aligned}
					\Xi_{2,1} &\lesssim \left\| t \chi_{(t_{0},x_{0}),\boldsymbol{\vec{\epsilon}}} \Pi_{k}^{(3)}(v) \right\|_{\widehat{H}^r_{-2}(\mathbb{R}^{2})} + \left\| \chi_{(t_{0},x_{0}),\boldsymbol{\vec{\epsilon}}} v_{k+1} \right\|_{\widehat{H}^r_{-2}(\mathbb{R}^{2})} + \left\| x \chi_{(t_{0},x_{0}),\boldsymbol{\vec{\epsilon}}} \partial_{x} v_{k} \right\|_{\widehat{H}^r_{-2}(\mathbb{R}^{2})} \\
					&= \mathcal{T}_{1} + \mathcal{T}_{2} + \mathcal{T}_{3}.
				\end{aligned}
			\end{equation*}
			Thus, we have the following estimate:
			\begin{equation*}
				\begin{split}
					\mathcal{T}_{1}
					&\lesssim  \sum_{\substack{k=k_{1}+k_{2}+k_{3}+k_{4} \\ 0 \leq k_{1},k_{2},k_{3},k_{4}\leq k}} \frac{k! 2^{k_{1}}}{k_{1}!k_{2}!k_{3}!k_{4}!} \,\left\| t\chi_{(t_{0},x_{0}),\boldsymbol{\vec{\epsilon}}}\, v_{k_{2}}v_{k_{3}}v_{k_{4}}  \right\|_{\widehat{H}^r_{-1}(\mathbb{R}^{2})}\\
					&\quad+\sum_{\substack{k=k_{1}+k_{2}+k_{3}+k_{4} \\ 0 \leq k_{1},k_{2},k_{3},k_{4} \leq k}} \frac{k! 2^{k_{1}}}{k_{1}!k_{2}!k_{3}!k_{4}!} \,\left\| t\partial_{x}\chi_{(t_{0},x_{0}),\boldsymbol{\vec{\epsilon}}}\, v_{k_{2}}v_{k_{3}}v_{k_{4}}  \right\|_{\widehat{H}^r_{-2}(\mathbb{R}^{2})}\\
					&=\mathcal{T}_{1,1}+\mathcal{T}_{1,2}.
				\end{split}
			\end{equation*}
			By the Hausdorff-Young inequality together with the continuity of the Bessel potential and {\sc Claim} \ref{c2.1}, we obtain the following estimate for $\mathcal{T}_{1,1}$:
			\small
			\begin{align*}
				\mathcal{T}_{1,1}
				&\sim \sum_{\substack{k=k_{1}+k_{2}+k_{3}+k_{4} \\ 0 \leq k_{1},k_{2},k_{3},k_{4} \leq k}} \frac{k! 2^{k_{1}}}{k_{1}!k_{2}!k_{3}!k_{4}!} \\
				&\quad \times \left\|\langle (\tau,\xi)\rangle^{-1}\,\mathcal{F}_{t,x}\left( t\chi_{(t_{0},x_{0}),\boldsymbol{\vec{\epsilon}}}\, v_{k_{2}}v_{k_{3}}v_{k_{4}} \right)(\tau,\xi) \right\|_{L^{r'}_{\tau\xi}(\mathbb{R}^{2})}\\
				&\sim \sum_{\substack{k=k_{1}+k_{2}+k_{3}+k_{4} \\ 0 \leq k_{1},k_{2},k_{3},k_{4} \leq k}} \frac{k! 2^{k_{1}}}{k_{1}!k_{2}!k_{3}!k_{4}!} \\
				&\quad \times \left\|\mathcal{F}_{t,x}\left(\langle \nabla_{t,x}\rangle^{-1}\left( t\chi_{(t_{0},x_{0}),\boldsymbol{\vec{\epsilon}}}\, v_{k_{2}}v_{k_{3}}v_{k_{4}} \right)\right)(\tau,\xi)\right\|_{L^{r'}_{\tau\xi}(\mathbb{R}^{2})}\\
				&\lesssim \sum_{\substack{k=k_{1}+k_{2}+k_{3}+k_{4} \\ 0 \leq k_{1},k_{2},k_{3},k_{4} \leq k}} \frac{k! 2^{k_{1}}}{k_{1}!k_{2}!k_{3}!k_{4}!} \,\left\|\langle \nabla_{t,x}\rangle^{-1}\left( t\chi_{(t_{0},x_{0}),\boldsymbol{\vec{\epsilon}}}\, v_{k_{2}}v_{k_{3}}v_{k_{4}} \right) \right\|_{L^{r}_{xt}(\mathbb{R}^{2})}\\
				&\lesssim \sum_{\substack{k=k_{1}+k_{2}+k_{3}+k_{4} \\ 0 \leq k_{1},k_{2},k_{3},k_{4} \leq k}} \frac{k! 2^{k_{1}}}{k_{1}!k_{2}!k_{3}!k_{4}!} \\
				&\quad \times \left\|\langle \nabla_{t,x}\rangle^{-1}\left( t\chi_{(t_{0},x_{0}),\boldsymbol{\vec{\epsilon}}}\, v_{k_{2}}v_{k_{3}}v_{k_{4}} \right) \right\|_{L^{1,\infty}_{xt}(\mathbb{R}^{2})}^{1-\theta}\,\left\|\langle \nabla_{t,x}\rangle^{-1}\left( t\chi_{(t_{0},x_{0}),\boldsymbol{\vec{\epsilon}}}\, v_{k_{2}}v_{k_{3}}v_{k_{4}} \right) \right\|_{L^{p,\infty}_{xt}(\mathbb{R}^{2})}^{\theta}\\
				&\sim_{r,s} \sum_{\substack{k=k_{1}+k_{2}+k_{3}+k_{4} \\ 0 \leq k_{1},k_{2},k_{3},k_{4} \leq k}} \frac{k! 2^{k_{1}}}{k_{1}!k_{2}!k_{3}!k_{4}!} \,\left\| t\chi_{(t_{0},x_{0}),\boldsymbol{\vec{\epsilon}}}\, v_{k_{2}}v_{k_{3}}v_{k_{4}} \right\|_{L^{1}_{xt}(\mathbb{R}^{2})}\\
				&\lesssim \sum_{\substack{k=k_{1}+k_{2}+k_{3}+k_{4} \\ 0 \leq k_{1},k_{2},k_{3},k_{4} \leq k}} \frac{k! 2^{k_{1}}}{k_{1}!k_{2}!k_{3}!k_{4}!} \\
				&\quad \times \left\|t\chi_{(t_{0},x_{0}),\boldsymbol{\vec{\epsilon}}}\right\|_{L^{8/5}_{xt}(\mathbb{R}^{2})}\,\| v_{k_{2}}\|_{L^{8}_{xt}(\mathbb{R}^{2})}\,\| v_{k_{3}}\|_{L^{8}_{xt}(\mathbb{R}^{2})} \,\| v_{k_{4}}\|_{L^{8}_{xt}(\mathbb{R}^{2})}\\
				&\lesssim_{\epsilon} k!\sum_{ \substack{k=k_{1}+k_{2}+k_{3}+k_{4} \\ 0 \leq k_{1},k_{2},k_{3},k_{4} \leq k} }\frac{ 2^{k_{1}}}{k_{1}!}\alpha^{k_{2}+k_{3}+k_{4}}\\
				&\sim k!\alpha^{k}\sum_{k_{1}=0}^{k}\frac{2^{k_{1}}}{k_{1}!}\alpha^{-k_{1}}(k-k_{1})\left((k-k_{1}+1)(k-k_{1}+2)-3(k-k_{1})\right)\\
				&\lesssim e^{\frac{2}{\alpha}}(k+4)!\alpha^{k},
			\end{align*}
			where we applied the interpolation inequality
			\begin{equation*}
				\|f\|_{L^{r}} \leq \left(\frac{r}{r-1} + \frac{r}{p-r}\right)^{\frac{1}{r}} \|f\|_{L^{1,\infty}}^{1-\theta} \, \|f\|_{L^{p,\infty}}^{\theta},
			\end{equation*}
			which holds for $1 < r < p \leq \infty$ and $\theta \in (0,1)$ such that $\frac{1}{r} = 1 - \theta + \frac{\theta}{p}$.\footnote{Recall that $L^{p,\infty}_{xt}(\mathbb{R}^{2})$ denotes the standard \emph{weak $L^p$ space}.}
			
			Thus
			\begin{equation*}
				\begin{split}
					\left\|\langle\nabla_{t,x}\rangle^{-1}(f_{1}f_{2}f_{3}f_{4})\right\|_{L^{r}_{xt}}&\lesssim_{p,r}\left\|\langle\nabla_{t,x}\rangle^{-1}(f_{1}f_{2}f_{3}f_{4})\right\|_{L^{1,\infty}_{xt}}^{1-\theta}\, \left\|\langle\nabla_{t,x}\rangle^{-1}(f_{1}f_{2}f_{3}f_{4})\right\|_{L^{p,\infty}_{xt}}^{\theta}.
				\end{split}
			\end{equation*}
			Since  $L^{1}_{xt}\hookrightarrow L^{1,\infty}_{xt}$ then  we have 
			\begin{equation*}
				\begin{split}
					\left\|\langle\nabla_{t,x}\rangle^{-1}(f_{1}f_{2}f_{3}f_{4})\right\|_{L^{1,\infty}_{xt}}&\lesssim 	\left\|\langle\nabla_{t,x}\rangle^{-1}(f_{1}f_{2}f_{3}f_{4})\right\|_{L^{1}_{xt}}\\
					&\sim 	\left\|G_{1}*\left(f_{1}f_{2}f_{3}f_{4}\right)\right\|_{L^{1}_{xt}}\\
					&\lesssim \|G_{1}\|_{L^{1}}\left\|f_{1}f_{2}f_{3}f_{4}\right\|_{L^{1}_{xt}},
				\end{split}
			\end{equation*}
			where we have used Young's convolution inequality and the fact that $\|G_{1}\|_{L^{1}}\sim 1.$	 For the remainder term, we use the continuity of Bessel's potentials, which is the Hardy-Littlewood-Sobolev inequality for Bessel's potentials. More precisely, we have
			\begin{equation*}
				\left\|\langle\nabla_{t,x}\rangle^{-1}(f_{1}f_{2}f_{3}f_{4})\right\|_{L^{p,\infty}_{xt}}\lesssim_{p}\|f_{1}f_{2}f_{3}f_{4}\|_{L^{1}_{xt}},\quad \mbox{whenever}\,\, p\in (1,\infty).
			\end{equation*}
			Now, we choose in particular $p>r>1,$ so that 
			${\displaystyle
				\theta:=\left(\frac{p}{p-1}\right)\,\left(\frac{r-1}{r}\right)\in (0,1).}$
			
			Gathering the estimates above, we finally get the following estimate
			\begin{equation*}
				\begin{split}
					\left\|\langle\nabla_{t,x}\rangle^{-1}(f_{1}f_{2}f_{3}f_{4})\right\|_{L^{r}_{xt}}&\lesssim \left\|f_{1}f_{2}f_{3}f_{4}\right\|_{L^{1}_{xt}}\\
					&\leq c\|f_{1}\|_{L^{8/5}_{xt}}\|f_{2}\|_{L^{8}_{xt}}\|f_{3}\|_{L^{8}_{xt}}\|f_{4}\|_{L^{8}_{xt}}.
				\end{split}
			\end{equation*}
			Next, we estimate the term $\mathcal{T}_{1,2}$. By employing an argument analogous to the one used above, we obtain the following upper bound:
			\begin{equation*}
				\begin{aligned}
					\mathcal{T}_{1,2} &\lesssim \sum_{\substack{k=k_{1}+k_{2}+k_{3}+k_{4} \\ 0 \leq k_{1},k_{2},k_{3},k_{4} \leq k}} \frac{k! 2^{k_{1}}}{k_{1}!k_{2}!k_{3}!k_{4}!} \left\| t\partial_{x}\chi_{(t_{0},x_{0}),\boldsymbol{\vec{\epsilon}}}\, v_{k_{2}}v_{k_{3}}v_{k_{4}} \right\|_{\widehat{H}^r_{-1}(\mathbb{R}^{2})} \\
					&\lesssim_{\epsilon} k! \sum_{\substack{k=k_{1}+k_{2}+k_{3}+k_{4} \\ 0 \leq k_{1},k_{2},k_{3},k_{4} \leq k}} \frac{2^{k_{1}}}{k_{1}!} \alpha^{k_{2}+k_{3}+k_{4}}\\
					&\lesssim e^{\frac{2}{\alpha}}(k+4)!\alpha^{k}.
				\end{aligned}
			\end{equation*}
			Next, we address the term $\mathcal{T}_{2}$. By invoking Lemma \ref{bach1} that combined with \textbf{{\sc Claim}} \ref{c2.1},  allows us to obtain 
			\begin{equation*}
				\begin{aligned}
					\mathcal{T}_{2} &\lesssim \left\| \chi_{(t_{0},x_{0}),\boldsymbol{\vec{\epsilon}}}\right\|_{\widehat{H}^{\infty}_{2}(\mathbb{R}^{2})} \left\|v_{k+1}\right\|_{\widehat{H}^r_{-2}(\mathbb{R}^{2})} \\
					&\lesssim_{\epsilon} \left\|v_{k+1}\right\|_{\widehat{L}^{r}_{xt}(\mathbb{R}^{2})} \\
					&\leq c\alpha^{k+1} (k+1)!, \quad \text{for all } k \in \mathbb{N}_{0}.
				\end{aligned}
			\end{equation*}
			We  estimate the last contribution of the term $\Xi_{2,1}$, that is given by $\mathcal{T}_{3}.$ By combining  Lemma \ref{bach1} and \textbf{{\sc Claim}} \ref{c2.1}, we obtain:
			\begin{equation*}
				\begin{aligned}
					\mathcal{T}_{3} &\lesssim \| \partial_{x} ( x \chi_{(t_{0},x_{0}), \boldsymbol{\vec{\epsilon}}} \, v_{k} ) \|_{\widehat{H}^{r}_{-2}(\mathbb{R}^{2})} + \| \partial_{x}(x\chi_{(t_{0},x_{0}),\boldsymbol{\vec{\epsilon}}}) \, v_{k} \|_{\widehat{H}^{r}_{-2}(\mathbb{R}^{2})} \\
					&\lesssim \left( \| x \chi_{(t_{0},x_{0}),\boldsymbol{\vec{\epsilon}}} \|_{\widehat{H}^{\infty}_{1}(\mathbb{R}^{2})} + \| \partial_{x}(x \chi_{(t_{0},x_{0}), \epsilon}) \|_{\widehat{H}^{\infty}_{2}(\mathbb{R}^{2})} \right) \| v_{k} \|_{\widehat{H}^{r}_{-1}(\mathbb{R}^{2})} \\
					&\lesssim \left( \| x \chi_{(t_{0},x_{0}) ,\boldsymbol{\vec{\epsilon}}} \|_{\widehat{H}^{\infty}_{1}(\mathbb{R}^{2})} + \| \partial_{x}(x \chi_{(t_{0},x_{0}) ,\boldsymbol{\vec{\epsilon}}}) \|_{\widehat{H}^{\infty}_{2}(\mathbb{R}^{2})} \right) \| v_{k} \|_{\widehat{L}^{r}_{xt}(\mathbb{R}^{2})} \\
					&\lesssim_{\epsilon} \alpha^{k}k!, \quad \text{for all } k \in \mathbb{N}_{0}.
				\end{aligned}
			\end{equation*}
			Next, the terms $\Xi_{2,2}, \Xi_{2,3},$ and  $\Xi_{2,4}$ are estimated by using a similar argument as the one used to handle the term $\mathcal{T}_{3}$. For the sake of brevity, we omit the details; nevertheless, these terms satisfy the bound after combining Lemma \ref{bach1} and {\sc Claim } \ref{c2.1} from where we get 
			\begin{equation*}
				\Xi_{2,j}\lesssim_{\epsilon} \alpha^{k} k!,\quad \mbox{for all } \,\, k\in \mathbb{N}_{0}.
			\end{equation*}
			Finally, we address the term $\Xi_{3}$. More precisely, by combining Lemma \ref{bach1} and \textbf{{\sc Claim}} \ref{c2.1}, we establish the following bound:
			\begin{equation*}
				\begin{aligned}
					\Xi_{3} &\lesssim \sum_{j=0}^{3} \binom{3}{j} \left\| P^{j} v_{k} P^{3-j} \big( \chi_{(t_{0},x_{0}),\boldsymbol{\vec{\epsilon}}} \big) \right\|_{\widehat{H}^r_{-2}(\mathbb{R}^{2})} \\
					&\lesssim \sum_{j=0}^{3} \| v_{k+j} \|_{\widehat{H}^r_{-2}(\mathbb{R}^{2})} \left\| P^{3-j} \big( \chi_{(t_{0},x_{0}),\boldsymbol{\vec{\epsilon}}} \big) \right\|_{\widehat{H}^{\infty}_{2}(\mathbb{R}^{2})} \\
					&\lesssim \sum_{j=0}^{3} \alpha^{k+j}(k+j)! \left\| P^{3-j} \big( \chi_{(t_{0},x_{0}),\boldsymbol{\vec{\epsilon}}} \big) \right\|_{\widehat{H}^{\infty}_{2}(\mathbb{R}^{2})} \\
					&\lesssim_{\epsilon} \alpha^{k+3}(k+3)!, \quad \text{for any } k \in \mathbb{N} \setminus \{0\}.
				\end{aligned}
			\end{equation*}
		\end{claimproof}
        \setcounter{claim}{4}
		\begin{claim}\label{c4}
			Let $\vec{\boldsymbol{\epsilon}} = (\epsilon, \epsilon) \in (0, \infty)^2$ and let $\chi_{(t_0, x_0), \vec{\boldsymbol{\epsilon}}}$ be the cutoff function defined in \eqref{cutoff1}. For $1 < r \leq 2$ with $\frac{1}{r} + \frac{1}{r'} = 1$, there exists a constant $c > 0$ such that for all $k \in \mathbb{N}_0$, the following estimate holds:
			\begin{equation*}
				\left\| \chi_{(t_0, x_0), \boldsymbol{\vec{\epsilon}}} \, v_k \right\|_{\widehat{H}^{r}_{1+\theta}(\mathbb{R}^2)} \leq c \alpha^k k!, \quad \mbox{with}\, \theta\in(0,1).
			\end{equation*}
		\end{claim}
		\begin{claimproof}
			Applying Lemma \ref{interpo1}, we obtain the following inequality:
			\begin{equation*}
				\begin{aligned}
					\| \chi_{(t_{0},x_{0}),\vec{\boldsymbol{\epsilon}}} \, v_{k} \|_{\widehat{H}^{r}_{1+\theta}} \lesssim_{t_0, x_0} \,\, &\| \chi_{(t_{0},x_{0}),\boldsymbol{\vec{\epsilon}}} \, v_{k} \|_{\widehat{H}^r_{\theta-2}(\mathbb{R}^{2})} 
					+ \left\| t\partial_x^3 \left( \chi_{(t_{0},x_{0}),\boldsymbol{\vec{\epsilon}}} \, v_{k} \right) \right\|_{\widehat{H}^r_{\theta-2}(\mathbb{R}^{2})} \\
					&+ \left\| P^{3} \left( \chi_{(t_{0},x_{0}),\boldsymbol{\vec{\epsilon}}} \, v_{k} \right) \right\|_{\widehat{H}^r_{\theta-2}(\mathbb{R}^{2})} \\
					=: \,\, &	\overset{\sim}{\Xi_{1}} + 	\overset{\sim}{\Xi_{2}} + 	\overset{\sim}{\Xi_{3}}.
				\end{aligned}
			\end{equation*}
			The term $	\overset{\sim}{\Xi}_{1}$ may be readily estimated by applying \textbf{{\sc Claim}} \ref{c2.1} and \textbf{{\sc Claim}} \ref{c3} joint with  Lemma~\ref{bach1}. From these, we arrive at
			\begin{equation*}
				\overset{\sim}{\Xi}_{1}\leq c\alpha^{k}k!,\quad \mbox{for all}\,\, k\in\mathbb{N}_{0}, 
			\end{equation*}
			for some $c>0$.
			
			In virtue of  identity \eqref{identity2}, it is straightforward to obtain 
			\begin{equation*}
				\begin{split}
					\overset{\sim}{\Xi_{2}}&\leq  \left\|t\chi_{(t_{0},x_{0}),\boldsymbol{\vec{\epsilon}}} \, \partial_{x}^{3}v_{k}\right\|_{\widehat{H}^r_{\theta-2}(\mathbb{R}^{2})}+
					\left\|\partial_{x}^{2}\left(3t\partial_{x}\chi_{(t_{0},x_{0}),\boldsymbol{\vec{\epsilon}}} \,v_{k} \right)\right\|_{\widehat{H}^r_{\theta-2}(\mathbb{R}^{2})}
					\\
					&\quad+\left\|\partial_{x}\left(3t\partial_{x}^{2}\chi_{(t_{0},x_{0}),\boldsymbol{\vec{\epsilon}}} \,v_{k} \right)\right\|_{\widehat{H}^r_{\theta-2}(\mathbb{R}^{2})} +\left\|t\partial_{x}^{3}\chi_{(t_{0},x_{0}),\boldsymbol{\vec{\epsilon}}} \,v_{k}\right\|_{\widehat{H}^r_{\theta-2}(\mathbb{R}^{2})}\\
					&=	\overset{\sim}{\Xi_{2,1}}+	\overset{\sim}{\Xi_{2,2}}+	\overset{\sim}{\Xi_{2,3}}+	\overset{\sim}{\Xi_{2,4}}.
				\end{split}
			\end{equation*}
			Now, by identity \eqref{identity1} we obtain 
			\begin{equation*}
				\begin{aligned}
					\overset{\sim}{\Xi_{2,1}} &\lesssim \left\| t \chi_{(t_{0},x_{0}),\boldsymbol{\vec{\epsilon}}} \Pi_{k}^{(3)}(v) \right\|_{\widehat{H}^r_{\theta-2}(\mathbb{R}^{2})} + \left\| \chi_{(t_{0},x_{0}),\boldsymbol{\vec{\epsilon}}}v_{k+1} \right\|_{\widehat{H}^r_{\theta-2}(\mathbb{R}^{2})} + \left\| x \chi_{(t_{0},x_{0}),\boldsymbol{\vec{\epsilon}}} \partial_{x} v_{k} \right\|_{\widehat{H}^r_{\theta-2}(\mathbb{R}^{2})} \\
					&= 	\overset{\sim}{\mathcal{T}_{1}} + 	\overset{\sim}{\mathcal{T}_{2}} + 	\overset{\sim}{\mathcal{T}_{3}}.
				\end{aligned}
			\end{equation*}
			Thus, we have the following estimate:
			\begin{equation*}
				\begin{split}
					\overset{\sim}{\mathcal{T}_{1}}
					&\lesssim  \sum_{\substack{k=k_{1}+k_{2}+k_{3}+k_{4} \\ 0 \leq k_{1},k_{2},k_{3},k_{4}\leq k}} \frac{k! 2^{k_{1}}}{k_{1}!k_{2}!k_{3}!k_{4}!} \,\left\| t\chi_{(t_{0},x_{0}),\boldsymbol{\vec{\epsilon}}}\, v_{k_{2}}v_{k_{3}}v_{k_{4}}  \right\|_{\widehat{H}^r_{\theta-1}(\mathbb{R}^{2})}\\
					&\quad+\sum_{\substack{k=k_{1}+k_{2}+k_{3}+k_{4} \\ 0 \leq k_{1},k_{2},k_{3},k_{4} \leq k}} \frac{k! 2^{k_{1}}}{k_{1}!k_{2}!k_{3}!k_{4}!} \,\left\| t\partial_{x}\chi_{(t_{0},x_{0}),\boldsymbol{\vec{\epsilon}}}\, v_{k_{2}}v_{k_{3}}v_{k_{4}}  \right\|_{\widehat{H}^r_{\theta-2}(\mathbb{R}^{2})}\\
					&=	\overset{\sim}{\mathcal{T}_{1,1}}+	\overset{\sim}{\mathcal{T}_{1,2}}.
				\end{split}
			\end{equation*}
			A similar argument as the used  to handle $\mathcal{T}_{1,1}$ in previous {\sc Claim} \ref{c3}  yields 
			\begin{align*}
				\overset{\sim}{\mathcal{T}}_{1,1} &\sim \sum_{\substack{k=k_{1}+k_{2}+k_{3}+k_{4} \\ 0 \leq k_{1},k_{2},k_{3},k_{4}\leq k}} \frac{k! \, 2^{k_{1}}}{k_{1}!k_{2}!k_{3}!k_{4}!} \\
				&\quad \times \left\| \langle (\tau,\xi) \rangle^{\theta-1} \mathcal{F}_{t,x} \bigg( t \chi_{(t_{0},x_{0}),\boldsymbol{\vec{\epsilon}}} \, v_{k_{2}}v_{k_{3}}v_{k_{4}} \bigg)(\tau,\xi) \right\|_{L^{r'}_{\tau\xi}(\mathbb{R}^{2})} \\[2ex]
				&\lesssim_{r,s} \sum_{\substack{k=k_{1}+k_{2}+k_{3}+k_{4} \\ 0 \leq k_{1},k_{2},k_{3},k_{4} \leq k}} \frac{k! 2^{k_{1}}}{k_{1}!k_{2}!k_{3}!k_{4}!} \,\left\| t\chi_{(t_{0},x_{0}),\boldsymbol{\vec{\epsilon}}}\, v_{k_{2}}v_{k_{3}}v_{k_{4}} \right\|_{L^{1}_{xt}(\mathbb{R}^{2})} \\[2ex]
				&\lesssim \sum_{\substack{k=k_{1}+k_{2}+k_{3}+k_{4} \\ 0 \leq k_{1},k_{2},k_{3},k_{4} \leq k}} \frac{k! 2^{k_{1}}}{k_{1}!k_{2}!k_{3}!k_{4}!} \\
				&\quad \times \left\|t\chi_{(t_{0},x_{0}),\boldsymbol{\vec{\epsilon}}}\right\|_{L^{8/5}_{xt}(\mathbb{R}^{2})}\,\| v_{k_{2}}\|_{L^{8}_{xt}(\mathbb{R}^{2})}\,\| v_{k_{3}}\|_{L^{8}_{xt}(\mathbb{R}^{2})} \,\| v_{k_{4}}\|_{L^{8}_{xt}(\mathbb{R}^{2})} \\[2ex]
				&\lesssim_{\epsilon} k! \sum_{\substack{k=k_{1}+k_{2}+k_{3}+k_{4} \\ 0 \leq k_{1},k_{2},k_{3},k_{4} \leq k}} \frac{ 2^{k_{1}}}{k_{1}!}\alpha^{k_{2}+k_{3}+k_{4}}\\
				&\lesssim e^{\frac{2}{\alpha}}(k+4)!\alpha^{k}.
			\end{align*}
			The situation for $	\overset{\sim}{\mathcal{T}_{1,2}}$ is different, and it is handled by using H\"older's inequality joint with {\sc Claim }\ref{c2.1}. More precisely, we obtain the following upper bound 
			\begin{equation*}
				\begin{split}
					\overset{\sim}{\mathcal{T}}_{1,2}& \lesssim\sum_{\substack{k=k_{1}+k_{2}+k_{3}+k_{4} \\ 0 \leq k_{1},k_{2},k_{3},k_{4} \leq k}} \frac{k! 2^{k_{1}}}{k_{1}!k_{2}!k_{3}!k_{4}!} \,\left\| t\partial_{x}\chi_{(t_{0},x_{0}),\boldsymbol{\vec{\epsilon}}}\, v_{k_{2}}v_{k_{3}}v_{k_{4}}  \right\|_{\widehat{H}^r_{\theta-2}(\mathbb{R}^{2})}\\
					&\lesssim_{\epsilon} k! \sum_{\substack{k=k_{1}+k_{2}+k_{3}+k_{4} \\ 0 \leq k_{1},k_{2},k_{3},k_{4} \leq k}} \frac{2^{k_{1}}}{k_{1}!} \alpha^{k_{2}+k_{3}+k_{4}}\\
					&\lesssim e^{\frac{2}{\alpha}}(k+4)!\alpha^{k}.
				\end{split}
			\end{equation*}
			In the case of $	\overset{\sim}{\mathcal{T}_{3}}$ is straightforward to handle. Specifically, we combine 
			Lemma \ref{bach1} and \textbf{{\sc Claim}} \ref{c2.1} to obtain:
			\begin{equation*}
				\begin{aligned}
					\overset{\sim}{\mathcal{T}_{3}} &\lesssim \| \partial_{x} ( x \chi_{(t_{0},x_{0}), \boldsymbol{\vec{\epsilon}}} \, v_{k} ) \|_{\widehat{H}^{r}_{\theta-2}(\mathbb{R}^{2})} + \| \partial_{x}(x\chi_{(t_{0},x_{0}),\boldsymbol{\vec{\epsilon}}}) \, v_{k} \|_{\widehat{H}^{r}_{\theta-2}(\mathbb{R}^{2})} \\
					&\lesssim \left( \| x \chi_{(t_{0},x_{0}) ,\boldsymbol{\vec{\epsilon}}} \|_{\widehat{H}^{\infty}_{1-\theta}(\mathbb{R}^{2})} + \| \partial_{x}(x \chi_{(t_{0},x_{0}) ,\boldsymbol{\vec{\epsilon}}}) \|_{\widehat{H}^{\infty}_{2-\theta}(\mathbb{R}^{2})} \right) \| v_{k} \|_{\widehat{L}^{r}_{xt}(\mathbb{R}^{2})} \\
					&\lesssim_{\epsilon} \alpha^{k}k!, \quad \text{for all } k \in \mathbb{N}_{0}.
				\end{aligned}
			\end{equation*}
			Next, we focus on estimate $	\overset{\sim}{\Xi_{2,2}}$, that is handled after using {\sc Claim } \ref{c3} as follows
			\begin{equation*}
				\begin{split}
					\overset{\sim}{\Xi_{2,2}}&\lesssim \left\|t\partial_{x}\chi_{(t_{0},x_{0}),\boldsymbol{\vec{\epsilon}}} \,v_{k} \right\|_{\widehat{H}^r_{\theta}(\mathbb{R}^{2})},
				\end{split}
			\end{equation*}
			which is straightforward to handle by using {\sc claim } \ref{c3}. More precisely,
			we employ the cut-off function $\chi_{(t_{0},x_{0}),\boldsymbol{	\overset{\sim}{\vec{\epsilon}}}}$ 
			as defined in \eqref{cutoff1}, specializing to the case 
			$\boldsymbol{	\overset{\sim}{\vec{\epsilon}}} = (2\epsilon, 2\epsilon)$. Since 
			$\chi_{(t_{0},x_{0}),\boldsymbol{	\overset{\sim}{\vec{\epsilon}}}} \equiv 1$ on 
			$\operatorname{supp}(\chi_{(t_{0},x_{0}),\boldsymbol{	\vec{\epsilon}}})$ we can apply again {\sc claim } \ref{c3} to this cut-off function to obtain 
			\begin{equation*}
				\begin{split}
					\left\|t\partial_{x}\chi_{(t_{0},x_{0}),\boldsymbol{\vec{\epsilon}}} \,v_{k} \right\|_{\widehat{H}^{r}_{\theta}(\mathbb{R}^{2})}&\leq 
					\left\|\chi_{(t_{0},x_{0}),\boldsymbol{	\overset{\sim}{\vec{\epsilon}}}} \,v_{k}\right\|_{\widehat{H}^{r}_{1}(\mathbb{R}^{2})}\\
					&\lesssim_{\epsilon} \alpha^{k}k!,\quad k\in\mathbb{N}_{0}.
				\end{split}
			\end{equation*}
			Next,  we focus on $	\overset{\sim}{\Xi_{2,3}},$ which satisfies the following bound 
			\begin{equation*}
				\begin{split}
					\overset{\sim}{\Xi_{2,3}}&\lesssim \left\|t\partial_{x}^{2}\chi_{(t_{0},x_{0}),\boldsymbol{\vec{\epsilon}}} \,v_{k} \right\|_{\widehat{H}^r_{\theta-1}(\mathbb{R}^{2})}\\
					&=c\left\|t\partial_{x}^{2}\chi_{(t_{0},x_{0}),\boldsymbol{\vec{\epsilon}}} \,\chi_{(t_{0},x_{0}),\boldsymbol{	\overset{\sim}{\vec{\epsilon}}}} \,v_{k} \right\|_{\widehat{H}^r_{\theta-1}(\mathbb{R}^{2})}\\
					&\lesssim_{\epsilon}\left\|\chi_{(t_{0},x_{0}),\boldsymbol{	\overset{\sim}{\vec{\epsilon}}}} \,v_{k} \right\|_{\widehat{H}^r_{\theta-1}(\mathbb{R}^{2})}\\
					&\lesssim \alpha^{k} k!,\quad k\in\mathbb{N}_{0},
				\end{split}
			\end{equation*}
			where  we have used the {\sc Claims} \ref{c2.1} and \ref{c3}.
			
			The next term in the argument is $	\overset{\sim}{\Xi_{2,4}}$, which we control by using a similar argument as the one employed above. More precisely, it satisfies
			\begin{equation*}
				\overset{\sim}{\Xi_{2,4}}\lesssim \alpha^{k}k!,\quad k\in\mathbb{N}_{0}.
			\end{equation*}
			Finally we deal with $	\overset{\sim}{\Xi_{3}},$ which in virtue of Leibniz rule, Lemma \ref{bach1} and {\sc Claim } \ref{c2.1} it satisfies 
			\begin{equation*}
				\begin{aligned}
					\overset{\sim}{\Xi_{3}} &\lesssim \sum_{j=0}^{3} \binom{3}{j} \left\| P^{j} v_{k} P^{3-j} \big( \chi_{(t_{0},x_{0}),\boldsymbol{\vec{\epsilon}}} \big) \right\|_{\widehat{H}^r_{\theta-2}(\mathbb{R}^{2})} \\
					&\lesssim \sum_{j=0}^{3} \alpha^{k+j}(k+j)! \left\| P^{3-j} \big( \chi_{(t_{0},x_{0}),\boldsymbol{\vec{\epsilon}}} \big) \right\|_{\widehat{H}^{\infty}_{2-\theta}(\mathbb{R}^{2})} \\
					&\lesssim_{\epsilon} \alpha^{k+3}(k+3)!, \quad \text{for any } k \in \mathbb{N} \setminus \{0\}.
				\end{aligned}
			\end{equation*}
			
		\end{claimproof}
        \setcounter{claim}{5}
		\begin{claim}\label{c5}
			Let $\vec{\boldsymbol{\epsilon}} = (\epsilon, \epsilon) \in (0, \infty)^2$ and let $\chi_{(t_0, x_0), \vec{\boldsymbol{\epsilon}}}$ be the cutoff function defined in \eqref{cutoff1}. For $1 < r \leq 2$ with $\frac{1}{r} + \frac{1}{r'} = 1$, there exists a constant $c > 0$ such that for all $k \in \mathbb{N}_0$, the following estimate holds:
			\begin{equation*}
				\left\| \chi_{(t_0, x_0), \vec{\boldsymbol{\epsilon}}} \, v_k \right\|_{\widehat{H}^{r}_{2+\theta}(\mathbb{R}^2)} \leq c \alpha^k k!, \quad \mbox{with}\, \theta\in\left(\frac{2}{3},1\right).
			\end{equation*}
		\end{claim}
		\begin{claimproof}
			We apply  Lemma \ref{interpo1} to obtain
			\begin{equation*}
				\begin{aligned}
					\| \chi_{(t_{0},x_{0}),\vec{\boldsymbol{\epsilon}}} \, v_{k} \|_{\widehat{H}^{r}_{2+\theta}} \lesssim_{t_0, x_0} \,\, &\| \chi_{(t_{0},x_{0}),\boldsymbol{\vec{\epsilon}}} \, v_{k} \|_{\widehat{H}^r_{\theta-1}(\mathbb{R}^{2})} 
					+ \left\| t\partial_x^3 \left( \chi_{(t_{0},x_{0}),\boldsymbol{\vec{\epsilon}}} \, v_{k} \right) \right\|_{\widehat{H}^r_{\theta-1}(\mathbb{R}^{2})} \\
					&+ \left\| P^{3} \left( \chi_{(t_{0},x_{0}),\boldsymbol{\vec{\epsilon}}} \, v_{k} \right) \right\|_{\widehat{H}^r_{\theta-1}(\mathbb{R}^{2})} \\
					&	=\overset{\approx}{\Xi_{1}}+ \overset{\approx}{\Xi_{2}} +  \overset{\approx}{\Xi_{3}}.
				\end{aligned}
			\end{equation*}
			Notice that by {\sc Claim } \ref{c2.1}, we have for all $k\in\mathbb{N}_{0}:$
			\begin{equation*}
				\overset{\approx}{\Xi_{1}}\lesssim \alpha^{k}k!.
			\end{equation*}
			Now,  we use decomposition in \cref{identity1,identity2} to obtain 
			\begin{equation*}
				\begin{aligned}
					\overset{\approx}{\Xi_{2}} &\lesssim \left\| t \chi_{(t_{0},x_{0}),\boldsymbol{\vec{\epsilon}}} \Pi_{k}^{(3)}(v) \right\|_{\widehat{H}^r_{\theta-1}(\mathbb{R}^{2})} + \left\| \chi_{(t_{0},x_{0}),\boldsymbol{\vec{\epsilon}}} v_{k+1} \right\|_{\widehat{H}^r_{\theta-1}(\mathbb{R}^{2})} + \left\| x \chi_{(t_{0},x_{0}),\boldsymbol{\vec{\epsilon}}} \partial_{x} v_{k} \right\|_{\widehat{H}^r_{\theta-1}(\mathbb{R}^{2})} \\
					&= 	\overset{\approx}{\mathcal{T}_{1}} + \overset{\approx}{\mathcal{T}_{2}} + \overset{\approx}{\mathcal{T}_{3}}.
				\end{aligned}
			\end{equation*}
			Hence
			\begin{equation*}
				\begin{split}
					\overset{\approx}{\mathcal{T}_{1}}
					&\lesssim  \sum_{\substack{k=k_{1}+k_{2}+k_{3}+k_{4} \\ 0 \leq k_{1},k_{2},k_{3},k_{4}\leq k}} \frac{k! 2^{k_{1}}}{k_{1}!k_{2}!k_{3}!k_{4}!} \,\left\| t\chi_{(t_{0},x_{0}),\boldsymbol{\vec{\epsilon}}}\, v_{k_{2}}v_{k_{3}}v_{k_{4}}  \right\|_{\widehat{H}^r_{\theta}(\mathbb{R}^{2})}\\
					&\quad+\sum_{\substack{k=k_{1}+k_{2}+k_{3}+k_{4} \\ 0 \leq k_{1},k_{2},k_{3},k_{4} \leq k}} \frac{k! 2^{k_{1}}}{k_{1}!k_{2}!k_{3}!k_{4}!} \,\left\| t\partial_{x}\chi_{(t_{0},x_{0}),\boldsymbol{\vec{\epsilon}}}\, v_{k_{2}}v_{k_{3}}v_{k_{4}}  \right\|_{\widehat{H}^r_{\theta-1}(\mathbb{R}^{2})}\\
					&=	\overset{\approx}{\mathcal{T}_{1,1}}+	\overset{\approx}{\mathcal{T}_{1,2}}.
				\end{split}
			\end{equation*}
			To handle  $\overset{\approx}{\mathcal{T}_{1,1}}$ we  shall emphasize several issues that allows to handle the trilinear estimate with the level of regularity required at this step of the argument.
			
			First, we consider  the cut-off function $\chi_{(t_{0},x_{0}),\boldsymbol{	\overset{\sim}{\vec{\epsilon}}}}$ 
			as defined in \eqref{cutoff1}, specializing to the case 
			$\boldsymbol{	\overset{\sim}{\vec{\epsilon}}} = (2\epsilon, 2\epsilon)$. Since 
			$\chi_{(t_{0},x_{0}),\boldsymbol{	\overset{\sim}{\vec{\epsilon}}}} \equiv 1$ on 
			$\operatorname{supp}(\chi_{(t_{0},x_{0}),\vec{\boldsymbol{\epsilon}}})$  then for all $(\tau,\xi)\in \mathbb{R}^{2}$ the following identity holds
			\begin{equation*}
				\begin{split}
					\langle (\tau,\xi)\rangle^{\theta}\mathcal{F}_{t,x}\left( t\chi_{(t_{0},x_{0}),\boldsymbol{\vec{\epsilon}}}\, v_{k_{2}}v_{k_{3}}v_{k_{4}} \right)(\tau,\xi)&=\langle (\tau,\xi)\rangle^{\theta}\mathcal{F}_{t,x}\left( t\chi_{(t_{0},x_{0}),\boldsymbol{\vec{\epsilon}}}\, \prod_{j=1}^{4}\chi_{(t_{0},x_{0}),\boldsymbol{	\overset{\sim}{\vec{\epsilon}}}} v_{k_{j}} \right)(\tau,\xi).
				\end{split}
			\end{equation*}
			Additionally,
			\begin{equation*}
				\begin{split}
					&\mathcal{F}_{t,x}\left( t\chi_{(t_{0},x_{0}),\boldsymbol{\vec{\epsilon}}}\, \prod_{j=2}^{4}\chi_{(t_{0},x_{0}),\boldsymbol{\overset{\sim}{\vec{\epsilon}}}} v_{k_{j}} \right)(\tau,\xi)\\
					&\quad = \iiint\limits_{\mathbb{R}^{6}} \mathcal{F}_{t,x} (t\chi_{(t_{0},x_{0}),\boldsymbol{\vec{\epsilon}}})(\eta_{1}) \, \mathcal{F}_{t,x}(\chi_{(t_{0},x_{0}),\boldsymbol{\overset{\sim}{\vec{\epsilon}}}} v_{k_{2}})(\eta_{2})  \, \mathcal{F}_{t,x} (\chi_{(t_{0},x_{0}),\boldsymbol{\overset{\sim}{\vec{\epsilon}}}} v_{k_{3}})(\eta_{3}) \\
					&\quad\qquad \times \mathcal{F}_{t,x} (\chi_{(t_{0},x_{0}),\boldsymbol{\overset{\sim}{\vec{\epsilon}}}} v_{k_{4}})((\tau,\xi)-\eta_{1}-\eta_{2}-\eta_{3}) \,\mathrm{d}\eta_{1} \,\mathrm{d}\eta_{2} \,\mathrm{d}\eta_{3},
				\end{split}
			\end{equation*}
			where $\eta_{j} := (\tilde{\eta}_{j1}, \tilde{\eta}_{j2}) \in \mathbb{R}^{2}$ for $j=1,2,3$.
			
			Notice that for $\theta\in(0,1)$ we have
			\begin{equation}\label{deco1}
				\begin{split}
					\langle(\tau,\xi) \rangle^{\theta} 
					&\lesssim \langle(\tau,\xi)-\eta_{1}-\eta_{2}-\eta_{3}\rangle^{\theta} + \langle \eta_{1}\rangle^{\theta} + \langle \eta_{2}\rangle^{\theta} + \langle \eta_{3}\rangle^{\theta}.
				\end{split}
			\end{equation}
			Consequently, 
			\begin{equation}\label{deco3}
				\begin{split}
					\overset{\approx}{\mathcal{T}_{1,1}}
					&\sim \sum_{\substack{k=k_{1}+k_{2}+k_{3}+k_{4} \\ 0 \leq k_{1},k_{2},k_{3},k_{4} \leq k}} \frac{k! 2^{k_{1}}}{k_{1}!k_{2}!k_{3}!k_{4}!}\,\left\|\langle (\tau,\xi)\rangle^{\theta}\,\mathcal{F}_{t,x}\left( t\chi_{(t_{0},x_{0}),\boldsymbol{\vec{\epsilon}}}\, v_{k_{2}}v_{k_{3}}v_{k_{4}} \right)(\tau,\xi) \right\|_{L^{r'}_{\tau\xi}(\mathbb{R}^{2})}\\
					&\lesssim \sum_{\substack{k=k_{1}+k_{2}+k_{3}+k_{4} \\ 0 \leq k_{1},k_{2},k_{3},k_{4} \leq k}} \frac{k! 2^{k_{1}}}{k_{1}!k_{2}!k_{3}!k_{4}!}\, \Bigg\{ \\
					& \sum_{j=1}^{3} \left\|\iiint\limits_{\mathbb{R}^{6}} \langle \eta_{j}\rangle^{\theta} \left| \mathcal{F}_{t,x} (t\chi_{(t_{0},x_{0}),\boldsymbol{\vec{\epsilon}}})(\eta_{1}) \, \mathcal{F}_{t,x} (\chi_{(t_{0},x_{0}),\boldsymbol{\overset{\sim}{\vec{\epsilon}}}} v_{k_{2}}) (\eta_{2}) \,  \mathcal{F}_{t,x} (\chi_{(t_{0},x_{0}),\boldsymbol{\overset{\sim}{\vec{\epsilon}}}} v_{k_{3}}) (\eta_{3}) \right| \right.\\
					&\left. \times \left| \mathcal{F}_{t,x} (\chi_{(t_{0},x_{0}),\boldsymbol{\overset{\sim}{\vec{\epsilon}}}} v_{k_{4}}) ((\tau,\xi)-\eta_{1}-\eta_{2}-\eta_{3}) \right| \,\mathrm{d}\eta_{1} \,\mathrm{d}\eta_{2} \,\mathrm{d}\eta_{3} \right\|_{L^{r'}_{\tau\xi}(\mathbb{R}^{2})}\\
					&+\left\| \iiint\limits_{\mathbb{R}^{6}} \left| \mathcal{F}_{t,x} (t\chi_{(t_{0},x_{0}),\boldsymbol{\vec{\epsilon}}})(\eta_{1}) \,  \mathcal{F}_{t,x} (\chi_{(t_{0},x_{0}),\boldsymbol{\overset{\sim}{\vec{\epsilon}}}} v_{k_{2}}) (\eta_{2}) \right|  \, \langle (\tau,\eta)-\eta_{1}-\eta_{2}-\eta_{3}\rangle^{\theta} \left| \mathcal{F}_{t,x} (\chi_{(t_{0},x_{0}),\boldsymbol{\overset{\sim}{\vec{\epsilon}}}} v_{k_{3}}) (\eta_{3}) \right| \right.\\
					&\left.\times \left| \mathcal{F}_{t,x} (\chi_{(t_{0},x_{0}),\boldsymbol{\overset{\sim}{\vec{\epsilon}}}} v_{k_{4}}) ((\tau,\xi)-\eta_{1}-\eta_{2}-\eta_{3}) \right| \,\mathrm{d}\eta_{1} \,\mathrm{d}\eta_{2} \,\mathrm{d}\eta_{3}\right\|_{L^{r'}_{\tau\xi}(\mathbb{R}^{2})} \Bigg\}\\
				\end{split}
			\end{equation}
			The expression above can be represented as a sum of fourfold convolutions. More precisely, the term reduces to:
			\begin{equation}\label{deco2}
				\begin{aligned}
					\overset{\approx}{\mathcal{T}_{1,1}}
					&\lesssim\sum_{\substack{k = k_1+k_2+k_3+k_4 \\ 0 \leq k_1,k_2,k_3,k_4 \leq k}}  \frac{k! 2^{k_1}}{k_1!k_2!k_3!k_4!} \Biggl\{ \\
					& \left\| \big( \langle \cdot \rangle^{\theta} |\mathcal{F}_{t,x}(t\chi_{(t_{0},x_{0}),\boldsymbol{\vec{\epsilon}}})| \big) * |\mathcal{F}_{t,x}(\chi_{(t_{0},x_{0}),\boldsymbol{\overset{\sim}{\vec{\epsilon}}}} v_{k_2})| * |\mathcal{F}_{t,x}(\chi_{(t_{0},x_{0}),\boldsymbol{\overset{\sim}{\vec{\epsilon}}}} v_{k_3})| * |\mathcal{F}_{t,x}(\chi_{(t_{0},x_{0}),\boldsymbol{\overset{\sim}{\vec{\epsilon}}}} v_{k_4})| \right\|_{L^{r'}_{\tau\xi}} \\
					& + \left\| |\mathcal{F}_{t,x}(t\chi_{(t_{0},x_{0}),\boldsymbol{\vec{\epsilon}}})| * \big( \langle \cdot \rangle^{\theta} |\mathcal{F}_{t,x}(\chi_{(t_{0},x_{0}),\boldsymbol{\overset{\sim}{\vec{\epsilon}}}} v_{k_2})| \big) * |\mathcal{F}_{t,x}(\chi_{(t_{0},x_{0}),\boldsymbol{\overset{\sim}{\vec{\epsilon}}}} v_{k_3})| * |\mathcal{F}_{t,x}(\chi_{(t_{0},x_{0}),\boldsymbol{\overset{\sim}{\vec{\epsilon}}}} v_{k_4})| \right\|_{L^{r'}_{\tau\xi}} \\
					& + \left\| |\mathcal{F}_{t,x}(t\chi_{(t_{0},x_{0}),\boldsymbol{\vec{\epsilon}}})| * |\mathcal{F}_{t,x}(\chi_{(t_{0},x_{0}),\boldsymbol{\overset{\sim}{\vec{\epsilon}}}} v_{k_2})| * \big( \langle \cdot \rangle^{\theta} |\mathcal{F}_{t,x}(\chi_{(t_{0},x_{0}),\boldsymbol{\overset{\sim}{\vec{\epsilon}}}} v_{k_3})| \big) * |\mathcal{F}_{t,x}(\chi_{(t_{0},x_{0}),\boldsymbol{\overset{\sim}{\vec{\epsilon}}}} v_{k_4})| \right\|_{L^{r'}_{\tau\xi}} \\
					& + \left\| |\mathcal{F}_{t,x}(t\chi_{(t_{0},x_{0}),\boldsymbol{\vec{\epsilon}}})| * |\mathcal{F}_{t,x}(\chi_{(t_{0},x_{0}),\boldsymbol{\overset{\sim}{\vec{\epsilon}}}} v_{k_2})| * |\mathcal{F}_{t,x}(\chi_{(t_{0},x_{0}),\boldsymbol{\overset{\sim}{\vec{\epsilon}}}} v_{k_3})| * \big( \langle \cdot \rangle^{\theta} |\mathcal{F}_{t,x}(\chi_{(t_{0},x_{0}),\boldsymbol{\overset{\sim}{\vec{\epsilon}}}} v_{k_4})| \big) \right\|_{L^{r'}_{\tau\xi}} \Biggr\}\\
					&=\Upsilon_{1}+\Upsilon_{2}+\Upsilon_{3}+\Upsilon_{4}.
				\end{aligned}
			\end{equation}
			By symmetry, the analysis for each term is essentially identical; thus, we shall focus our estimation on a single representative term, with the others following \emph{mutatis mutandis}.
			\underline{{\sc  Estimate for $\Upsilon_{1}$:}}
			
			For $q_{1} := \frac{r(4+3\theta)}{3} > 2$, we define $p_{1}$ via the relation
			${\displaystyle
				\frac{1}{p_{1}} := \frac{1}{q_{1}} + \frac{1}{r'}}$
			which ensures $p_{1} \in (1, \infty)$. Applying H\"older's inequality, we obtain:
			\begin{equation*}
				\begin{split}
					\left\|\langle \cdot \rangle^{\theta} \mathcal{F}_{t,x} (t\chi_{(t_{0},x_{0}),\boldsymbol{\vec{\epsilon}}})\right\|_{L^{p_{1}}_{\tau\xi}(\mathbb{R}^{2})} 
					&\leq \left\|\frac{1}{\langle\cdot\rangle}\right\|_{L^{q_{1}}(\mathbb{R}^{2})} \left\|\langle \cdot
					\rangle^{1+\theta} \mathcal{F}_{t,x} (t\chi_{(t_{0},x_{0}),\boldsymbol{\vec{\epsilon}}})\right\|_{L^{r'}_{\tau\xi}(\mathbb{R}^{2})} \\
					&\sim \left\|\langle \cdot\rangle^{1+\theta} \mathcal{F}_{t,x} (t\chi_{(t_{0},x_{0}),\boldsymbol{\vec{\epsilon}}})\right\|_{L^{r'}_{\tau\xi}(\mathbb{R}^{2})}.
				\end{split}
			\end{equation*}
			\underline{{\sc  Estimates for the remainder terms:} }
			
			A similar analysis applies to the remaining factors. More precisely, for $j=2, 3, 4$, we set ${\displaystyle
				q_{j} := \frac{r(4+3\theta)}{3(1+\theta)}}$
			and define the corresponding exponents $p_{j}$ such that
			${\displaystyle 
				\frac{1}{p_{j}} := \frac{1}{q_{j}} + \frac{1}{r'},\quad j=2,3,4.}$
			
			Thus, by H\"older's inequality, we get 
			\begin{equation*}
				\begin{split}
					\left\|\mathcal{F}_{t,x} (\chi_{(t_{0},x_{0}),\boldsymbol{\overset{\sim}{\vec{\epsilon}}}} v_{k_{j}}) \right\|_{L^{p_{j}}_{\tau\xi}(\mathbb{R}^{2})}&\lesssim \left\|\frac{1}{\langle \cdot\rangle^{1+\theta}}\right\|_{L^{q_{j}}_{\tau\xi}(\mathbb{R}^{2})}\left\| \langle \cdot
					\rangle^{1+\theta} \mathcal{F}_{t,x} (\chi_{(t_{0},x_{0}),\boldsymbol{\overset{\sim}{\vec{\epsilon}}}} v_{k_{j}}) \right\|_{L^{r'}_{\tau\xi}(\mathbb{R}^{2})}\\
					&\sim \left\| \langle \cdot
					\rangle^{1+\theta} \mathcal{F}_{t,x} (\chi_{(t_{0},x_{0}),\boldsymbol{\overset{\sim}{\vec{\epsilon}}}} v_{k_{j}}) \right\|_{L^{r'}_{\tau\xi}(\mathbb{R}^{2})},
				\end{split}
			\end{equation*}
			for $j=2,3,4.$
			
			Gathering these results and employing Young’s convolution inequality, we arrive at the following estimate
			\begin{equation*}
				\begin{split}
					\Upsilon_{1}&\lesssim \left\|\langle \cdot\rangle^{1+\theta} \mathcal{F}_{t,x} (t\chi_{(t_{0},x_{0}),\boldsymbol{\vec{\epsilon}}})\right\|_{L^{r'}_{\tau\xi}(\mathbb{R}^{2})}\prod_{j=2}^{4} \left\| \langle \cdot\rangle^{1+\theta} \mathcal{F}_{t,x} (\chi_{(t_{0},x_{0}),\boldsymbol{\overset{\sim}{\vec{\epsilon}}}} v_{k_{j}}) \right\|_{L^{r'}_{\tau\xi}(\mathbb{R}^{2})},
				\end{split}
			\end{equation*}
			whenever
			${\displaystyle
				\frac{1}{p_{1}}+\frac{1}{p_{2}}+\frac{1}{p_{3}}+\frac{1}{p_{4}}=\frac{3}{r}+\frac{4}{r'}=3+\frac{1}{r'}.}$
			
			Now,  we use  {\sc claim } \ref{c4}  adapting  the weight $\chi_{(t_{0},x_{0}),\boldsymbol{	\overset{\sim}{\epsilon}}}$ 
			as defined in \eqref{cutoff1}, to the case 
			$\boldsymbol{	\overset{\sim}{\vec{\epsilon}}} = (2\epsilon, 2\epsilon)$  ($\epsilon>0$) to obtain
			\begin{equation*}
				\begin{split}
					\Upsilon_{1}&\lesssim  k! \sum_{\substack{k=k_{1}+k_{2}+k_{3}+k_{4} \\ 0 \leq k_{1},k_{2},k_{3},k_{4}\leq k}} \,\frac{2^{k_{1}}}{k_{1}!} \,\alpha^{k_{2}+k_{3}+k_{4}}\\
					&\lesssim e^{\frac{2}{\alpha}}(k+4)!\alpha^{k}.
				\end{split}
			\end{equation*}
			An analogous analysis applies to the terms $\Upsilon_{2}, \Upsilon_{3},$ and $\Upsilon_{4}$. For the sake of brevity, we omit the explicit details; nevertheless, these terms satisfy the following bound for $j=2,3,4$:
			\begin{equation*}
				\begin{split}
					\Upsilon_{j} &\lesssim k! \sum_{\substack{k=k_{1}+k_{2}+k_3+k_4 \\ 0 \leq k_{1},k_{2},k_{3},k_{4}\leq k}} \frac{2^{k_{1}}}{k_{1}!} \alpha^{k_{2}+k_{3}+k_{4}}\\
					&\lesssim e^{\frac{2}{\alpha}}(k+4)!\alpha^{k}.
				\end{split}
			\end{equation*}
			For the term $\overset{\approx}{\mathcal{T}_{1,2}}$, a combination of the Hausdorff-Young inequality and {\sc claim }\ref{c2.2} yields:
			\begin{equation*}
				\begin{split}
					\overset{\approx}{\mathcal{T}_{1,2}}&\lesssim	\sum_{\substack{k=k_{1}+k_{2}+k_{3}+k_{4} \\ 0 \leq k_{1},k_{2},k_{3},k_{4} \leq k}} \frac{k! 2^{k_{1}}}{k_{1}!k_{2}!k_{3}!k_{4}!} \,\left\|\langle (\tau,\xi)\rangle^{\theta-1} \,\mathcal{F}_{t,x}\left(\partial_{x}\chi_{(t_{0},x_{0}),\boldsymbol{\vec{\epsilon}}}\, v_{k_{2}}v_{k_{3}}v_{k_{4}} \right)(\tau,\xi)
					\right\|_{L^{r'}_{\tau\xi}(\mathbb{R}^{2})}\\
					&\lesssim \sum_{\substack{k=k_{1}+k_{2}+k_{3}+k_{4} \\ 0 \leq k_{1},k_{2},k_{3},k_{4} \leq k}} \frac{k! 2^{k_{1}}}{k_{1}!k_{2}!k_{3}!k_{4}!} \,\left\|\langle \nabla_{t,x}\rangle^{\theta-1}\left( t\chi_{(t_{0},x_{0}),\boldsymbol{\vec{\epsilon}}}\, v_{k_{2}}v_{k_{3}}v_{k_{4}} \right) \right\|_{L^{r}_{xt}(\mathbb{R}^{2})}\\
					&\lesssim \sum_{\substack{k=k_{1}+k_{2}+k_{3}+k_{4} \\ 0 \leq k_{1},k_{2},k_{3},k_{4} \leq k}} \frac{k! 2^{k_{1}}}{k_{1}!k_{2}!k_{3}!k_{4}!} \\
					&\quad \times \left\|\langle \nabla_{t,x}\rangle^{\theta-1}\left( t\chi_{(t_{0},x_{0}),\boldsymbol{\vec{\epsilon}}}\, v_{k_{2}}v_{k_{3}}v_{k_{4}} \right) \right\|_{L^{1,\infty}_{xt}(\mathbb{R}^{2})}^{1-\gamma}\,\left\|\langle \nabla_{t,x}\rangle^{\theta-1}\left( t\chi_{(t_{0},x_{0}),\boldsymbol{\vec{\epsilon}}}\, v_{k_{2}}v_{k_{3}}v_{k_{4}} \right) \right\|_{L^{p,\infty}_{xt}(\mathbb{R}^{2})}^{\gamma
					}\\
					&\sim_{r,s} \sum_{\substack{k=k_{1}+k_{2}+k_{3}+k_{4} \\ 0 \leq k_{1},k_{2},k_{3},k_{4} \leq k}} \frac{k! 2^{k_{1}}}{k_{1}!k_{2}!k_{3}!k_{4}!} \,\left\| t\chi_{(t_{0},x_{0}),\boldsymbol{\vec{\epsilon}}}\, v_{k_{2}}v_{k_{3}}v_{k_{4}} \right\|_{L^{1}_{xt}(\mathbb{R}^{2})}\\
					&\lesssim \sum_{\substack{k=k_{1}+k_{2}+k_{3}+k_{4} \\ 0 \leq k_{1},k_{2},k_{3},k_{4} \leq k}} \frac{k! 2^{k_{1}}}{k_{1}!k_{2}!k_{3}!k_{4}!} \\
					&\quad \times\left\|t\chi_{(t_{0},x_{0}),\boldsymbol{\vec{\epsilon}}}\right\|_{L^{8/5}_{xt}(\mathbb{R}^{2})}\,\| v_{k_{2}}\|_{L^{8}_{xt}(\mathbb{R}^{2})}\,\| v_{k_{3}}\|_{L^{8}_{xt}(\mathbb{R}^{2})} \,\| v_{k_{4}}\|_{L^{8}_{xt}(\mathbb{R}^{2})}\\
					&\lesssim_{\epsilon} k! \sum_{\substack{k=k_{1}+k_{2}+k_{3}+k_{4} \\ 0 \leq k_{1},k_{2},k_{3},k_{4} \leq k}} \frac{ 2^{k_{1}}}{k_{1}!}\alpha^{k_{2}+k_{3}+k_{4}}\\
					&\lesssim e^{\frac{2}{\alpha}}(k+4)!\alpha^{k}.
				\end{split}
			\end{equation*}
			We have used the following interpolation inequality for Lorentz spaces:
			\begin{equation*}
				\|f\|_{L^{r}_{xt}(\mathbb{R}^{2})}\leq \left(\frac{r}{r-1}+\frac{r}{p-r}\right)^{\frac{1}{r}}\|f\|_{L^{1,\infty}_{xt}(\mathbb{R}^{2})}^{1-\gamma}\,\|f\|_{L^{p,\infty}_{xt}(\mathbb{R}^{2})}^{\gamma},\quad \frac{1}{r}=1-\gamma+\frac{\gamma}{p}, \quad \gamma\in(0,1),
			\end{equation*}
			with $1<r<p\leq \infty.$ Thus, it follows that
			\begin{equation*}
				\begin{split}
					\left\|\langle\nabla_{t,x}\rangle^{\theta-1}(f_{1}f_{2}f_{3}f_{4})\right\|_{L^{r}_{xt}(\mathbb{R}^{2})} &\lesssim_{p,r} \left\|\langle\nabla_{t,x}\rangle^{\theta-1}(f_{1}f_{2}f_{3}f_{4})\right\|_{L^{1,\infty}_{xt}(\mathbb{R}^{2})}^{1-\gamma} \left\|\langle\nabla_{t,x}\rangle^{\theta-1}(f_{1}f_{2}f_{3}f_{4})\right\|_{L^{p,\infty}_{xt}(\mathbb{R}^{2})}^{\gamma}.
				\end{split}
			\end{equation*}
			Thus,
			\begin{equation*}
				\begin{split}
					\left\|\langle\nabla_{t,x}\rangle^{\theta-1}(f_{1}f_{2}f_{3}f_{4})\right\|_{L^{1,\infty}_{xt}(\mathbb{R}^{2})}&\lesssim \left\|\langle\nabla_{t,x}\rangle^{\theta-1}(f_{1}f_{2}f_{3}f_{4})\right\|_{L^{1}_{xt}(\mathbb{R}^{2})}\\
					&\sim \left\|G_{1-\theta}*(f_{1}f_{2}f_{3}f_{4})\right\|_{L^{1}_{xt}(\mathbb{R}^{2})}\\
					&\lesssim \|G_{1-\theta}\|_{L^{1}}\left\|f_{1}f_{2}f_{3}f_{4}\right\|_{L^{1}_{xt}(\mathbb{R}^{2})},
				\end{split}
			\end{equation*}
			where we have used Young's convolution inequality and the fact that $\|G_{1-\theta}\|_{L^{1}}\sim 1.$ For the remainder term, we appeal to the continuity of Bessel potentials, specifically the Hardy-Littlewood-Sobolev inequality for Bessel potentials. More precisely, we have
			\begin{equation*}
				\left\|\langle\nabla_{t,x}\rangle^{\theta-1}(f_{1}f_{2}f_{3}f_{4})\right\|_{L^{p,\infty}_{xt}}\lesssim_{p}\|f_{1}f_{2}f_{3}f_{4}\|_{L^{1}_{xt}},\quad \text{whenever } p\in (1,\infty).
			\end{equation*}
			Now, choosing $p>r>1$ such that
			${\displaystyle 		\gamma := \left(\frac{p}{p-1}\right)\left(\frac{r-1}{r}\right) \in (0,1),}$
			and gathering the estimates above, we finally obtain the following bound:
			\begin{equation*}
				\begin{split}
					\left\|\langle\nabla_{t,x}\rangle^{\theta-1}(f_{1}f_{2}f_{3}f_{4})\right\|_{L^{r}_{xt}}&\lesssim \left\|f_{1}f_{2}f_{3}f_{4}\right\|_{L^{1}_{xt}}\\
					&\leq c\|f_{1}\|_{L^{8/5}_{xt}}\|f_{2}\|_{L^{8}_{xt}}\|f_{3}\|_{L^{8}_{xt}}\|f_{4}\|_{L^{8}_{xt}}.
				\end{split}
			\end{equation*}
			Next, we handle $ \overset{\approx}{\mathcal{T}_{2}}$ which  in virtue of {\sc claim } \ref{c2.1} satisfies
			\begin{equation*}
				\begin{split}
					\overset{\approx}{\mathcal{T}_{2}}&\sim	\left\| \langle (\tau,\xi)\rangle^{\theta-1}\mathcal{F}_{t,x}\left(\chi_{(t_{0},x_{0}),\boldsymbol{\vec{\epsilon}}} v_{k+1} \right)\right\|_{L^{r'}_{\tau\xi}(\mathbb{R}^{2})}\\
					&\lesssim \left\|\chi_{(t_{0},x_{0}),\boldsymbol{\vec{\epsilon}}} v_{k+1} \right\|_{\widehat{L^{r}}_{xt}(\mathbb{R}^{2})}\\
					&\lesssim \alpha^{k+1}(k+1)!
				\end{split}
			\end{equation*}
			Subsequently, we combine the results of {\sc  Claims} \ref{c2.1} and \ref{c3}  and Lemma \ref{bach1} as applied to the weight $\chi_{(t_{0},x_{0}),\boldsymbol{\overset{\sim}{\vec{\epsilon}}}}$, where we set $\boldsymbol{\overset{\sim}{\vec{\epsilon}}}=(2\epsilon,2\epsilon)$ for some $\epsilon>0$. Indeed, we obtain that for all $k\in\mathbb{N}_{0}:$
			\begin{equation*}
				\begin{split}
					\overset{\approx}{\mathcal{T}_{3}}&\leq \left\| x \chi_{(t_{0},x_{0}),\boldsymbol{\vec{\epsilon}}}\,\chi_{(t_{0},x_{0}),\boldsymbol{\overset{\sim}{\vec{\epsilon}}}}\, v_{k} \right\|_{\widehat{H}^r_{\theta}(\mathbb{R}^{2})} +\left\| \partial_{x}\left(x \chi_{(t_{0},x_{0}),\boldsymbol{\vec{\epsilon}}}\right)\,
					\chi_{(t_{0},x_{0}),\boldsymbol{\overset{\sim}{\vec{\epsilon}}}}v_{k} \right\|_{\widehat{H}^r_{\theta-1}(\mathbb{R}^{2})}\\
					&\lesssim \left\| x \chi_{(t_{0},x_{0}),\boldsymbol{\vec{\epsilon}}}\right\|_{\widehat H^{\,\infty}_{\theta}(\mathbb R^{2})}\,\|\chi_{(t_{0},x_{0}),\boldsymbol{\overset{\sim}{\vec{\epsilon}}}}\, v_{k} \|_{\widehat H^{\,r}_{\theta}(\mathbb R^{2})}\\
					&\quad+\left\| \partial_{x}\left(x \chi_{(t_{0},x_{0}),\boldsymbol{\vec{\epsilon}}}\right) \right\|_{\widehat H^{\,\infty}_{1-\theta}(\mathbb R^{2})}\,\|\chi_{(t_{0},x_{0}),\boldsymbol{\overset{\sim}{\vec{\epsilon}}}}\, v_{k} \|_{ \widehat{L^{r}}_{xt}(\mathbb R^{2})}\\
					&\lesssim \alpha^{k}k!.
				\end{split}
			\end{equation*}
			In  virtue of Leibniz rule, Lemma \ref{bach1} and {\sc Claim } \ref{c2.1} we get
			\begin{equation*}
				\begin{aligned}
					\overset{\approx}{\Xi_{3}} &\lesssim \sum_{j=0}^{3} \binom{3}{j} \left\| P^{j} v_{k} P^{3-j} \big( \chi_{(t_{0},x_{0}),\boldsymbol{\vec{\epsilon}}} \big) \right\|_{\widehat{H}^r_{\theta-1}(\mathbb{R}^{2})} \\
					&\lesssim \sum_{j=0}^{3} \alpha^{k+j}(k+j)! \left\| P^{3-j} \big( \chi_{(t_{0},x_{0}),\boldsymbol{\vec{\epsilon}}} \big) \right\|_{\widehat{H}^{\infty}_{2-\theta}(\mathbb{R}^{2})} \\
					&\lesssim_{\epsilon} \alpha^{k+3}(k+3)!, \quad \text{for any } k \in \mathbb{N} \setminus \{0\}.
				\end{aligned}
			\end{equation*}
		\end{claimproof}
        \setcounter{claim}{6}
		\begin{claim}\label{c6}
			Let $\vec{\boldsymbol{\epsilon}} = (\epsilon, \epsilon) \in (0, \infty)^2$ and let $\chi_{(t_0, x_0), \vec{\boldsymbol{\epsilon}}}$ be the cutoff function defined in \eqref{cutoff1}. For $1 < r \leq 2$ with $\frac{1}{r} + \frac{1}{r'} = 1$, there exists a constant $c > 0$ such that for all $k \in \mathbb{N}_0$, the following estimate holds:
			\begin{equation*}
				\left\| \chi_{(t_0, x_0), \vec{\boldsymbol{\epsilon}}} \, v_k \right\|_{\widehat{H}^{r}_{3+\theta}(\mathbb{R}^2)} \leq c \alpha^k k!, \quad \mbox{with}\, \theta\in\left(\frac{2}{3},1\right).
			\end{equation*}
		\end{claim}
		\begin{claimproof}
			We apply  Lemma \ref{interpo1} to obtain
			\begin{equation*}
				\begin{aligned}
					\| \chi_{(t_{0},x_{0}),\boldsymbol{\epsilon}} \, v_{k} \|_{\widehat{H}^{r}_{3+\theta}} \lesssim_{t_0, x_0} \,\, &\| \chi_{(t_{0},x_{0}),\boldsymbol{\vec{\epsilon}}} \, v_{k} \|_{\widehat{H}^r_{\theta}(\mathbb{R}^{2})} 
					+ \left\| t\partial_x^3 \left( \chi_{(t_{0},x_{0}),\boldsymbol{\vec{\epsilon}}} \, v_{k} \right) \right\|_{\widehat{H}^r_{\theta}(\mathbb{R}^{2})} \\
					&+ \left\| P^{3} \left( \chi_{(t_{0},x_{0}),\boldsymbol{\vec{\epsilon}}} \, v_{k} \right) \right\|_{\widehat{H}^r_{\theta}(\mathbb{R}^{2})} \\
					&	=\overset{\approx \approx \approx}{\Xi_{1}}+ \overset{\approx \approx \approx}{\Xi_{2}} +  \overset{\approx \approx \approx}{\Xi_{3}}.
				\end{aligned}
			\end{equation*}
			{\sc Claim } \ref{c3} implies that 
			\begin{equation*}
				\overset{\approx \approx \approx}{\Xi_{1}}\lesssim \alpha^{k}k!,\quad \forall k\in \mathbb{N}_{0}.
			\end{equation*}
			Now,  we use decomposition in \cref{identity1,identity2} to obtain 
			\begin{equation*}
				\begin{aligned}
					\overset{\approx \approx \approx}{\Xi_{2}} &\lesssim \left\| t \chi_{(t_{0},x_{0}),\boldsymbol{\vec{\epsilon}}} \Pi_{k}^{(3)}(v) \right\|_{\widehat{H}^r_{\theta}(\mathbb{R}^{2})} + \left\| \chi_{(t_{0},x_{0}),\boldsymbol{\vec{\epsilon}}} v_{k+1} \right\|_{\widehat{H}^r_{\theta}(\mathbb{R}^{2})} + \left\| x \chi_{(t_{0},x_{0}),\boldsymbol{\vec{\epsilon}}} \partial_{x} v_{k} \right\|_{\widehat{H}^r_{\theta}(\mathbb{R}^{2})} \\
					&= 	\overset{\approx \approx \approx}{\mathcal{T}_{1}} +\overset{\approx \approx \approx}{\mathcal{T}_{2}} +\overset{\approx \approx \approx}{\mathcal{T}_{3}}.
				\end{aligned}
			\end{equation*}
			Thus
			\begin{equation*}
				\begin{split}
					\overset{\approx \approx \approx}{\mathcal{T}_{1}}
					&\lesssim  \sum_{\substack{k=k_{1}+k_{2}+k_{3}+k_{4} \\ 0 \leq k_{1},k_{2},k_{3},k_{4}\leq k}} \frac{k! 2^{k_{1}}}{k_{1}!k_{2}!k_{3}!k_{4}!} \,\left\| t\chi_{(t_{0},x_{0}),\boldsymbol{\vec{\epsilon}}}\, v_{k_{2}}v_{k_{3}}v_{k_{4}}  \right\|_{\widehat{H}^r_{\theta+1}(\mathbb{R}^{2})}\\
					&\quad+\sum_{\substack{k=k_{1}+k_{2}+k_{3}+k_{4} \\ 0 \leq k_{1},k_{2},k_{3},k_{4} \leq k}} \frac{k! 2^{k_{1}}}{k_{1}!k_{2}!k_{3}!k_{4}!} \,\left\| t\partial_{x}\chi_{(t_{0},x_{0}),\boldsymbol{\vec{\epsilon}}}\, v_{k_{2}}v_{k_{3}}v_{k_{4}}  \right\|_{\widehat{H}^r_{\theta}(\mathbb{R}^{2})}\\
					&=\overset{\approx \approx \approx}{\mathcal{T}_{1,1}}+	\overset{\approx \approx \approx}{\mathcal{T}_{1,2}}.
				\end{split}
			\end{equation*}
			The same analysis  described in this term in the {\sc claim} \ref{c5} implies that 
			\begin{equation*}
				\begin{aligned}
					\overset{\approx \approx \approx}{\mathcal{T}_{1,1}}
					&\lesssim\sum_{\substack{k = k_1+k_2+k_3+k_4 \\ 0 \leq k_1,k_2,k_3,k_4 \leq k}}  \frac{k! 2^{k_1}}{k_1!k_2!k_3!k_4!} \Biggl\{ \\
					& \left\| \big( \langle \cdot \rangle^{\theta+1} |\mathcal{F}_{t,x}(t\chi_{(t_{0},x_{0}),\boldsymbol{\vec{\epsilon}}})| \big) * |\mathcal{F}_{t,x}(\chi_{(t_{0},x_{0}),\boldsymbol{\overset{\sim}{\vec{\epsilon}}}} v_{k_2})| * |\mathcal{F}_{t,x}(\chi_{(t_{0},x_{0}),\boldsymbol{\overset{\sim}{\vec{\epsilon}}}} v_{k_3})| * |\mathcal{F}_{t,x}(\chi_{(t_{0},x_{0}),\boldsymbol{\overset{\sim}{\vec{\epsilon}}}} v_{k_4})| \right\|_{L^{r'}_{\tau\xi}} \\
					& + \left\| |\mathcal{F}_{t,x}(t\chi_{(t_{0},x_{0}),\boldsymbol{\vec{\epsilon}}})| * \big( \langle \cdot \rangle^{\theta+1} |\mathcal{F}_{t,x}(\chi_{(t_{0},x_{0}),\boldsymbol{\overset{\sim}{\vec{\epsilon}}}} v_{k_2})| \big) * |\mathcal{F}_{t,x}(\chi_{(t_{0},x_{0}),\boldsymbol{\overset{\sim}{\vec{\epsilon}}}} v_{k_3})| * |\mathcal{F}_{t,x}(\chi_{(t_{0},x_{0}),\boldsymbol{\overset{\sim}{\vec{\epsilon}}}} v_{k_4})| \right\|_{L^{r'}_{\tau\xi}} \\
					& + \left\| |\mathcal{F}_{t,x}(t\chi_{(t_{0},x_{0}),\boldsymbol{\vec{\epsilon}}})| * |\mathcal{F}_{t,x}(\chi_{(t_{0},x_{0}),\boldsymbol{\overset{\sim}{\vec{\epsilon}}}} v_{k_2})| * \big( \langle \cdot \rangle^{\theta+1} |\mathcal{F}_{t,x}(\chi_{(t_{0},x_{0}),\boldsymbol{\overset{\sim}{\vec{\epsilon}}}} v_{k_3})| \big) * |\mathcal{F}_{t,x}(\chi_{(t_{0},x_{0}),\boldsymbol{\overset{\sim}{\vec{\epsilon}}}} v_{k_4})| \right\|_{L^{r'}_{\tau\xi}} \\
					& + \left\| |\mathcal{F}_{t,x}(t\chi_{(t_{0},x_{0}),\boldsymbol{\vec{\epsilon}}})| * |\mathcal{F}_{t,x}(\chi_{(t_{0},x_{0}),\boldsymbol{\overset{\sim}{\vec{\epsilon}}}} v_{k_2})| * |\mathcal{F}_{t,x}(\chi_{(t_{0},x_{0}),\boldsymbol{\overset{\sim}{\vec{\epsilon}}}} v_{k_3})| * \big( \langle \cdot \rangle^{\theta+1} |\mathcal{F}_{t,x}(\chi_{(t_{0},x_{0}),\boldsymbol{\overset{\sim}{\vec{\epsilon}}}} v_{k_4})| \big) \right\|_{L^{r'}_{\tau\xi}} \Biggr\}\\
					&=	\overset{\sim}{\Upsilon_{1}}+	\overset{\sim}{\Upsilon_{2}}+	\overset{\sim}{\Upsilon_{3}}+	\overset{\sim}{\Upsilon_{4}}.
				\end{aligned}
			\end{equation*}
			We shall focus our estimation on a single representative term, with the others following \emph{Mutatis mutandis}.

			\underline{{\sc  Estimate for $\Upsilon_{1}$:}}
			
			For $q_{1} := \frac{r(4+3\theta)}{3} > 2$, we define $p_{1}$ via the relation ${\displaystyle 
				\frac{1}{p_{1}} := \frac{1}{q_{1}} + \frac{1}{r'}},$
			which ensures $p_{1} \in (1, \infty)$. Applying H\"older's inequality, we obtain:
			\begin{equation*}
				\begin{split}
					\left\|\langle \cdot \rangle^{\theta+1} \mathcal{F}_{t,x} (t\chi_{(t_{0},x_{0}),\boldsymbol{\vec{\epsilon}}})\right\|_{L^{p_{1}}_{\tau\xi}(\mathbb{R}^{2})} 
					&\leq \left\|\frac{1}{\langle\cdot\rangle}\right\|_{L^{q_{1}}(\mathbb{R}^{2})} \left\|\langle \cdot
					\rangle^{2+\theta} \mathcal{F}_{t,x} (t\chi_{(t_{0},x_{0}),\boldsymbol{\vec{\epsilon}}})\right\|_{L^{r'}_{\tau\xi}(\mathbb{R}^{2})} \\
					&\sim \left\|\langle \cdot\rangle^{2+\theta} \mathcal{F}_{t,x} (t\chi_{(t_{0},x_{0}),\boldsymbol{\vec{\epsilon}}})\right\|_{L^{r'}_{\tau\xi}(\mathbb{R}^{2})}.
				\end{split}
			\end{equation*}
			\underline{{\sc  Estimates for the remainder terms:} }
			
			A similar analysis applies to the remaining factors. More precisely, for $j=2, 3, 4$, we set ${\displaystyle 
				q_{j} := \frac{r(4+3\theta)}{3(1+\theta)}}$
			and define the corresponding exponents $p_{j}$ such that
			${\displaystyle 
				\frac{1}{p_{j}} := \frac{1}{q_{j}} + \frac{1}{r'},\quad j=2,3,4.}$
			
			Thus, by H\"older's inequality, we get 
			\begin{equation*}
				\begin{split}
					\left\|\mathcal{F}_{t,x} (\chi_{(t_{0},x_{0}),\boldsymbol{\overset{\sim}{\vec{\epsilon}}}} v_{k_{j}}) \right\|_{L^{p_{j}}_{\tau\xi}(\mathbb{R}^{2})}&\lesssim \left\|\frac{1}{\langle \cdot\rangle^{1+\theta}}\right\|_{L^{q_{j}}_{\tau\xi}(\mathbb{R}^{2})}\left\| \langle \cdot
					\rangle^{1+\theta} \mathcal{F}_{t,x} (\chi_{(t_{0},x_{0}),\boldsymbol{\overset{\sim}{\vec{\epsilon}}}} v_{k_{j}}) \right\|_{L^{r'}_{\tau\xi}(\mathbb{R}^{2})}\\
					&\lesssim \left\| \langle \cdot
					\rangle^{2+\theta} \mathcal{F}_{t,x} (\chi_{(t_{0},x_{0}),\boldsymbol{\overset{\sim}{\vec{\epsilon}}}} v_{k_{j}}) \right\|_{L^{r'}_{\tau\xi}(\mathbb{R}^{2})},
				\end{split}
			\end{equation*}
			for $j=2,3,4.$
			
			Gathering these results a combined with  Young’s convolution inequality  allows us to obtain 
			\begin{equation*}
				\begin{split}
					\overset{\sim}{\Upsilon_{1}}&\lesssim \left\|\langle \cdot\rangle^{2+\theta} \mathcal{F}_{t,x} (t\chi_{(t_{0},x_{0}),\boldsymbol{\vec{\epsilon}}})\right\|_{L^{r'}_{\tau\xi}(\mathbb{R}^{2})}\prod_{j=2}^{4} \left\| \langle \cdot\rangle^{2+\theta} \mathcal{F}_{t,x} (\chi_{(t_{0},x_{0}),\boldsymbol{\overset{\sim}{\vec{\epsilon}}}} v_{k_{j}}) \right\|_{L^{r'}_{\tau\xi}(\mathbb{R}^{2})},
				\end{split}
			\end{equation*}
			whenever ${\displaystyle 
				\frac{1}{p_{1}}+\frac{1}{p_{2}}+\frac{1}{p_{3}}+\frac{1}{p_{4}}=\frac{3}{r}+\frac{4}{r'}=3+\frac{1}{r'}.}$
			
			Now, we use  {\sc claim } \ref{c5}  to obtain that 
			\begin{equation*}
				\begin{split}
					\overset{\sim}{\Upsilon_{1}}&\lesssim  k! \sum_{\substack{k=k_{1}+k_{2}+k_{3}+k_{4} \\ 0 \leq k_{1},k_{2},k_{3},k_{4}\leq k}} \,\frac{2^{k_{1}}}{k_{1}!} \,\alpha^{k_{2}+k_{3}+k_{4}}\\
					&\lesssim e^{\frac{2}{\alpha}}(k+4)!\alpha^{k}.
				\end{split}
			\end{equation*}
			An analogous analysis applies to the terms $\Upsilon_{2}, \Upsilon_{3},$ and $\Upsilon_{4}$. For the sake of brevity, we omit the explicit details; nevertheless, these terms satisfy the following bound for $j=2,3,4$:
			\begin{equation*}
				\begin{split}
					\overset{\sim}{\Upsilon_{j}} &\lesssim k! \sum_{\substack{k=k_{1}+k_{2}+k_3+k_4 \\ 0 \leq k_{1},k_{2},k_{3},k_{4}\leq k}} \frac{2^{k_{1}}}{k_{1}!} \alpha^{k_{2}+k_{3}+k_{4}}\\
					&\lesssim e^{\frac{2}{\alpha}}(k+4)!\alpha^{k}.
				\end{split}
			\end{equation*}
			For the term $	\overset{\approx \approx \approx}{\mathcal{T}_{1,2}}$ we 
			\begin{equation*}
				\begin{split}
					\overset{\approx \approx \approx}{\mathcal{T}_{1,2}} &\lesssim \sum_{\substack{k=k_{1}+k_{2}+k_{3}+k_{4} \\ 0 \leq k_{1},k_{2},k_{3},k_{4} \leq k}} \frac{k! 2^{k_{1}}}{k_{1}!k_{2}!k_{3}!k_{4}!} \,\left\|\langle (\tau,\xi)\rangle^{\theta} \,\mathcal{F}_{t,x}\left(\partial_{x}\chi_{(t_{0},x_{0}),\boldsymbol{\vec{\epsilon}}}\, v_{k_{2}}v_{k_{3}}v_{k_{4}} \right)(\tau,\xi) \right\|_{L^{r'}_{\tau\xi}(\mathbb{R}^{2})} \\
					&\lesssim \sum_{\substack{k = k_1+k_2+k_3+k_4 \\ 0 \leq k_1,k_2,k_3,k_4 \leq k}} \frac{k! 2^{k_1}}{k_1!k_2!k_3!k_4!} \Biggl\{ \\
					&\, \left\| \big( \langle \cdot \rangle^{\theta} |\mathcal{F}_{t,x}(\partial_{x}\chi_{(t_{0},x_{0}),\boldsymbol{\vec{\epsilon}}})| \big) * |\mathcal{F}_{t,x}(\chi_{(t_{0},x_{0}),\boldsymbol{\overset{\sim}{\vec{\epsilon}}}} v_{k_2})| * |\mathcal{F}_{t,x}(\chi_{(t_{0},x_{0}),\boldsymbol{\overset{\sim}{\vec{\epsilon}}}} v_{k_3})| * |\mathcal{F}_{t,x}(\chi_{(t_{0},x_{0}),\boldsymbol{\overset{\sim}{\vec{\epsilon}}}} v_{k_4})| \right\|_{L^{r'}_{\tau\xi}} \\
					&\, + \left\| |\mathcal{F}_{t,x}(\partial_{x}\chi_{(t_{0},x_{0}),\boldsymbol{\vec{\epsilon}}})| * \big( \langle \cdot \rangle^{\theta} |\mathcal{F}_{t,x}(\chi_{(t_{0},x_{0}),\boldsymbol{\overset{\sim}{\vec{\epsilon}}}} v_{k_2})| \big) * |\mathcal{F}_{t,x}(\chi_{(t_{0},x_{0}),\boldsymbol{\overset{\sim}{\vec{\epsilon}}}} v_{k_3})| * |\mathcal{F}_{t,x}(\chi_{(t_{0},x_{0}),\boldsymbol{\overset{\sim}{\vec{\epsilon}}}} v_{k_4})| \right\|_{L^{r'}_{\tau\xi}} \\
					&\, + \left\| |\mathcal{F}_{t,x}(\partial_{x}\chi_{(t_{0},x_{0}),\boldsymbol{\vec{\epsilon}}})| * |\mathcal{F}_{t,x}(\chi_{(t_{0},x_{0}),\boldsymbol{\overset{\sim}{\vec{\epsilon}}}} v_{k_2})| * \big( \langle \cdot \rangle^{\theta} |\mathcal{F}_{t,x}(\chi_{(t_{0},x_{0}),\boldsymbol{\overset{\sim}{\vec{\epsilon}}}} v_{k_3})| \big) * |\mathcal{F}_{t,x}(\chi_{(t_{0},x_{0}),\boldsymbol{\overset{\sim}{\vec{\epsilon}}}} v_{k_4})| \right\|_{L^{r'}_{\tau\xi}} \\
					&\, + \left\| |\mathcal{F}_{t,x}(\partial_{x}\chi_{(t_{0},x_{0}),\boldsymbol{\vec{\epsilon}}})| * |\mathcal{F}_{t,x}(\chi_{(t_{0},x_{0}),\boldsymbol{\overset{\sim}{\vec{\epsilon}}}} v_{k_2})| * |\mathcal{F}_{t,x}(\chi_{(t_{0},x_{0}),\boldsymbol{\overset{\sim}{\vec{\epsilon}}}} v_{k_3})| * \big( \langle \cdot \rangle^{\theta} |\mathcal{F}_{t,x}(\chi_{(t_{0},x_{0}),\boldsymbol{\overset{\sim}{\vec{\epsilon}}}} v_{k_4})| \big) \right\|_{L^{r'}_{\tau\xi}} \Biggr\}\\
					&=	\overset{\approx}{\Upsilon_{1}}+	\overset{\approx}{\Upsilon_{2}}+	\overset{\approx}{\Upsilon_{3}}+	\overset{\approx}{\Upsilon_{4}}.
				\end{split}
			\end{equation*}
			A similar analysis as the employed to handle the term $	\overset{\approx}{\mathcal{T}_{1,1}}$ in {\sc claim } \ref{c5} allows us to obtain that  for $j=1,2,3,4,$ the terms $\overset{\approx}{\Upsilon_{j}}$ satisfy:
			\begin{equation*}
				\begin{split}
					\overset{\approx}{\Upsilon_{j}}&\lesssim  k! \sum_{\substack{k=k_{1}+k_{2}+k_{3}+k_{4} \\ 0 \leq k_{1},k_{2},k_{3},k_{4}\leq k}} \,\frac{2^{k_{1}}}{k_{1}!} \,\alpha^{k_{2}+k_{3}+k_{4}}\\
					&\lesssim e^{\frac{2}{\alpha}}(k+4)!\alpha^{k}.
				\end{split}
			\end{equation*}
			Next, we handle $ \overset{\approx}{\mathcal{T}_{2}}$ which  in virtue of {\sc claim } \ref{c2.1} satisfies
			\begin{equation*}
				\begin{split}
					\overset{\approx}{\mathcal{T}_{2}}&\sim	\left\| \langle (\tau,\xi)\rangle^{\theta-1}\mathcal{F}_{t,x}\left(\chi_{(t_{0},x_{0}),\boldsymbol{\vec{\epsilon}}} v_{k+1} \right)\right\|_{L^{r'}_{\tau\xi}(\mathbb{R}^{2})}\\
					&\lesssim \left\|\chi_{(t_{0},x_{0}),\boldsymbol{\vec{\epsilon}}} v_{k+1} \right\|_{\widehat{L^{r}}_{xt}(\mathbb{R}^{2})}\\
					&\lesssim \alpha^{k+1}(k+1)!
				\end{split}
			\end{equation*}
			We combine the results of {\sc  Claims} \ref{c3} and \ref{c4}  and Lemma \ref{bach1} as applied to the weight $\chi_{(t_{0},x_{0}),\boldsymbol{\overset{\sim}{\vec{\epsilon}}}}$, where we set $\boldsymbol{\overset{\sim}{\vec{\epsilon}}}=(2\epsilon,2\epsilon)$ for some $\epsilon>0$. Indeed, we obtain that for all $k\in\mathbb{N}_{0}:$
			\begin{equation*}
				\begin{split}
					\overset{\approx}{\mathcal{T}_{3}}&\leq \left\| x \chi_{(t_{0},x_{0}),\boldsymbol{\vec{\epsilon}}}\,\chi_{(t_{0},x_{0}),\boldsymbol{\overset{\sim}{\vec{\epsilon}}}}\, v_{k} \right\|_{\widehat{H}^r_{1+\theta}(\mathbb{R}^{2})} +\left\| \partial_{x}\left(x \chi_{(t_{0},x_{0}),\boldsymbol{\vec{\epsilon}}}\right)\,
					\chi_{(t_{0},x_{0}),\boldsymbol{\overset{\sim}{\vec{\epsilon}}}}v_{k} \right\|_{\widehat{H}^r_{\theta}(\mathbb{R}^{2})}\\
					&\lesssim \left\| x \chi_{(t_{0},x_{0}),\boldsymbol{\vec{\epsilon}}}\right\|_{\widehat H^{\,\infty}_{1+\theta}(\mathbb R^{2})}\,\|\chi_{(t_{0},x_{0}),\boldsymbol{\overset{\sim}{\vec{\epsilon}}}}\, v_{k} \|_{\widehat H^{\,r}_{1+\theta}(\mathbb R^{2})}\\
					&\quad+\left\| \partial_{x}\left(x \chi_{(t_{0},x_{0}),\boldsymbol{\vec{\epsilon}}}\right) \right\|_{\widehat H^{\,\infty}_{\theta}(\mathbb R^{2})}\,\|\chi_{(t_{0},x_{0}),\boldsymbol{\overset{\sim}{\vec{\epsilon}}}}\, v_{k} \|_{\widehat H^{\,r}_{\theta}(\mathbb R^{2})}\\
					&\lesssim \alpha^{k}k!.
				\end{split}
			\end{equation*}
			In  virtue of Leibniz rule, Lemma \ref{bach1} and {\sc Claim } \ref{c3} we get
			\begin{equation*}
				\begin{aligned}
					\overset{\approx}{\Xi_{3}} &\lesssim \sum_{j=0}^{3} \binom{3}{j} \left\| P^{j} v_{k} P^{3-j} \big( \chi_{(t_{0},x_{0}),\boldsymbol{\vec{\epsilon}}} \big) \right\|_{\widehat{H}^r_{\theta}(\mathbb{R}^{2})} \\
					&\lesssim \sum_{j=0}^{3} \alpha^{k+j}(k+j)! \left\| P^{3-j} \big( \chi_{(t_{0},x_{0}),\boldsymbol{\vec{\epsilon}}} \big) \right\|_{\widehat{H}^{\infty}_{\theta}(\mathbb{R}^{2})} \\
					&\lesssim_{\epsilon} \alpha^{k+3}(k+3)!, \quad \text{for any } k \in \mathbb{N} \setminus \{0\}.
				\end{aligned}
			\end{equation*}
		\end{claimproof}
        \setcounter{claim}{7}
		\begin{claim}\label{c7}
			Let $x_{0} \in \mathbb{R}$ and $t_{0} \in \mathbb{R}\setminus\{0\}$. Given $\epsilon > 0$ such that $\epsilon <|t_{0}|$, we define the spatial interval $I_{x_{0}} := (x_{0}-\epsilon, x_{0}+\epsilon)$ and the temporal interval $J_{t_0} := (t_{0}-\epsilon, t_{0}+\epsilon)$. For $1 < r \leq 2$, let $r'$ denote the conjugate exponent satisfying $\frac{1}{r}+\frac{1}{r'}=1$. Then, there exists a constant $\alpha_{8} > 0$, independent of $k$ and $l$, such that
			\begin{equation}\label{inter1}
				\sup_{t \in J_{t_0}} \left\| \mathcal{M}^{l} v_{k}(t, \cdot) \right\|_{\widehat{H}^{r}_{1+\theta}(I_{x_{0}})} \lesssim  \alpha^{k+l}_{8} (k+l)!
			\end{equation}
			holds for all $k, l \in \mathbb{N}_{0}$, where $\mathcal{M} := t^{1/3} \partial_{x}$ is the weighted differentiation operator.
		\end{claim}
		\begin{claimproof}
			The argument of the proof follows by applying induction on $l$. In this sense, we study the cases $l=0$ and $l=1$ that can be treated by using a similar analysis we proceed to describe below. 
			
			By Theorems \ref{trace1}-\ref{trace2}, 
			\begin{equation*}
				\begin{split}
					\left\|\mathcal{M}^{l}P^{k}v(t)\right\|_{\widehat{H}^{r}_{1+\theta}(I_{x_{0}})}&=|t|^{\frac{l}{3}}\left\|\partial_{x}^{l}P^{k}v(t)\right\|_{\widehat{H}^{r}_{1+\theta}(I_{x_{0}})}\\
					&\leq c \left(|t_{0}|+\epsilon\right)^{\frac{l}{3}}\,\left\|\partial_{x}^{l}P^{k}v\right\|_{\widehat{H}^{r}_{2+\theta}(I_{t_{0}}\times I_{x_{0}})}\\
					&\lesssim  \left(|t_{0}|+\epsilon\right)^{\frac{l}{3}}\,\left\|P^{k}v\right\|_{\widehat{H}^{r}_{3+\theta}(\mathbb{R}^{2})}\\
					&\lesssim  \left(|t_{0}|+\epsilon\right)^{\frac{l}{3}}\,\alpha^{k} k!\\
					&\lesssim  \left(|t_{0}|+\epsilon\right)^{\frac{l}{3}}\,\alpha^{k+l} (k+l)!\\
					&\leq c \alpha_{8}^{k+l}(k+l)!, \quad \forall \, k\in\mathbb{N}_{0}
				\end{split}
			\end{equation*}
			where we have used {\sc claim } \ref{c6} in  the penultimate inequality and $\alpha_{8}=\max(2,\alpha).$
			
			Since the  estimate \eqref{inter1} is true from the  arguments described above, we now  suppose that \eqref{inter1} is valid up to  $l\in \mathbb{N}$ with $l\geq 2.$
			
			Notice that 
			\begin{equation*}
				\mathcal{M}^{3}P^{k}v=-t\partial_{x}(P+2I)^{k}(v^{3})-\frac{P^{k+1}v}{3}+\frac{x}{3}\partial_{x}v_{k}, \quad k\in \mathbb{N}.
			\end{equation*}
			then
			\begin{equation*}
				\begin{split}
					\left\|\mathcal{M}^{l+1}P^{k}v(t)\right\|_{\widehat{H}^{r}_{1+\theta}(I_{x_{0}})}&=\left\|\mathcal{M}^{l-2}\mathcal{M}^{3}P^{k}v(t)\right\|_{\widehat{H}^{r}_{1+\theta}(I_{x_{0}})}\\
					&\lesssim \left\|t\mathcal{M}^{l-2}\partial_{x}(P+2I)^{k}(v^{3})\,\right\|_{\widehat{H}^{r}_{1+\theta}(I_{x_{0}})}+\left\|\mathcal{M}^{l-2}P^{k+1}v(t)\right\|_{\widehat{H}^{r}_{1+\theta}(I_{x_{0}})}\\
					&\quad +\left\|\mathcal{M}^{l-2}\left(x\partial_{x}v_{k}\right)\right\|_{\widehat{H}^{r}_{1+\theta}(I_{x_{0}})}\\
					&=:\Pi_{1}+\Pi_{2}+\Pi_{3}.
				\end{split}
			\end{equation*}
			First, notice that 
			\begin{equation*}
				\begin{split}
					\mathcal{M}^{l-1} (P+2I)^{k} (v^{3}) &= \sum_{\substack{k_{1}+k_{2}=k \\ 0 \leq k_{1},k_{2} \leq k}} 2^{k_{2}} \Bigg\{ \sum_{\substack{m_{1}+m_{2}+m_{3}=k_{1} \\ 0 \leq m_{1},m_{2},m_{3} \leq k_{1}}} \sum_{\substack{n_{1}+n_{2}+n_{3}=l-1 \\ 0 \leq n_{1},n_{2},n_{3} \leq l-1}} \\
					&\quad \frac{k! (l-1)!}{k_{2}! m_{1}! m_{2}! m_{3}! n_{1}! n_{2}! n_{3}!} \prod_{i=1}^{3} \mathcal{M}^{n_{i}} \left( P^{m_{i}} (v) \right) \Bigg\},
				\end{split}
			\end{equation*}
			so that,
			\begin{equation}\label{t2}
				\begin{split}
					\Pi_{1} &\lesssim (|t_{0}|+\epsilon)^{\frac{2}{3}} \sum_{j=0}^{k} \binom{k}{j} 2^{k-j} \left\{ \sum_{\substack{m_{1}+m_{2}+m_{3}=j \\ 0 \leq m_{1},m_{2},m_{3} \leq j}} \frac{j!}{m_{1}! m_{2}! m_{3}!}\right. \\
					&\left.\quad \times \sum_{\substack{n_{1}+n_{2}+n_{3}=l-1 \\ 0 \leq n_{1},n_{2},n_{3} \leq l-1}} \frac{(l-1)!}{n_{1}! n_{2}! n_{3}!} \left\|\prod_{i=1}^{3} \mathcal{M}^{n_{i}} \left( P^{m_{i}} (v) \right)\right\|_{\widehat{H}^{r}_{1+\theta}(I_{x_{0}})}\right\}.
				\end{split}
			\end{equation}
			To handle the trilinear term, we decouple the threefold product while properly accounting for the domain restriction. For a fixed $t \in J_{t_0}$ and any $\delta > 0$, the definition of the restricted norm as an infimum implies there exist extensions $\zeta_{j, \delta}(\cdot, t) \in \widehat{H}^{r}_{1+\theta}(\mathbb{R})$ for $j=1,2,3$, such that $\zeta_{j, \delta}|_{I_{x_{0}}} = \mathcal{M}^{n_{j}} P^{m_{j}} v(\cdot, t)$ and
			\begin{equation}\label{t1}
				\|\zeta_{j, \delta}(\cdot, t)\|_{\widehat{H}^{r}_{1+\theta}(\mathbb{R})} < \|\mathcal{M}^{n_{j}} P^{m_{j}} v(\cdot, t)\|_{\widehat{H}^{r}_{1+\theta}(I_{x_{0}})} + \delta.
			\end{equation}
			Since the product $\prod_{j=1}^{3} \zeta_{j, \delta}(\cdot, t)$ is a valid global extension of the product $\prod_{j=1}^{3} \mathcal{M}^{n_{j}} P^{m_{j}} v(\cdot, t)$ restricted to $I_{x_0}$, it follows by the definition of the restricted norm that:
			\begin{equation*}
				\left\| \prod_{j=1}^{3} \mathcal{M}^{n_{j}} P^{m_{j}} v(t) \right\|_{\widehat{H}^{r}_{1+\theta}(I_{x_{0}})} \leq \left\| \zeta_{1, \delta} \zeta_{2, \delta} \zeta_{3, \delta} \right\|_{\widehat{H}^{r}_{1+\theta}(\mathbb{R})}.
			\end{equation*}
			Applying the decomposition arguments from \eqref{deco1}-\eqref{deco2} on the real line, we obtain
			\begin{equation*}
				\begin{aligned}
					\left\|\zeta_{1, \delta}\zeta_{2, \delta}\zeta_{3, \delta}\right\|_{\widehat{H}^{r}_{1+\theta}(\mathbb{R})} &\lesssim \sum_{j=1}^3 \mathcal{T}_{j, \delta},
				\end{aligned}
			\end{equation*}
			where each $\mathcal{T}_{j, \delta}$ involves the weight $\langle \cdot \rangle^{1+\theta}$ acting on the $j$-th factor in the convolution. For $\mathcal{T}_{1, \delta}$, invoking Young's inequality and the estimate \eqref{t1} yields:
			\begin{equation*}
				\begin{aligned}
					\mathcal{T}_{1, \delta} &\lesssim \left\|\zeta_{1, \delta}\right\|_{\widehat{H}^{r}_{1+\theta}(\mathbb{R})} \left\|\mathcal{F}_{x}(\zeta_{2, \delta})\right\|_{L^1_\xi} \left\|\mathcal{F}_{x}(\zeta_{3, \delta})\right\|_{L^1_\xi} \\
					&\lesssim \prod_{i=1}^{3} \left\|\zeta_{i, \delta}\right\|_{\widehat{H}^{r}_{1+\theta}(\mathbb{R})} \\
					&\leq c \prod_{i=1}^{3} \left( \left\|\mathcal{M}^{n_{i}} P^{m_{i}} v(\cdot, t)\right\|_{\widehat{H}^{r}_{1+\theta}(I_{x_{0}})} + \delta \right).
				\end{aligned}
			\end{equation*}
			The same bound applies to $\mathcal{T}_{2, \delta}$ and $\mathcal{T}_{3, \delta}$. Substituting these into the product inequality and taking the limit as $\delta \to 0$, we conclude that for each $t \in J_{t_0}$:
			\begin{equation*}
				\left\| \prod_{j=1}^{3} \mathcal{M}^{n_{j}} P^{m_{j}} v(t) \right\|_{\widehat{H}^{r}_{1+\theta}(I_{x_{0}})} \lesssim \prod_{j=1}^{3} \left\|\mathcal{M}^{n_{j}} P^{m_{j}} v(t)\right\|_{\widehat{H}^{r}_{1+\theta}(I_{x_{0}})}.
			\end{equation*}
			Finally, taking the supremum over $t \in J_{t_0}$ completes the proof.
			
			Once we  apply the result above to \eqref{t2}, combined with the inductive hypothesis
			\begin{equation*}
				\begin{split}
					\Pi_{1} 
					&\lesssim (|t_{0}|+\epsilon)^{\frac{2}{3}} \sum_{j=0}^{k} \binom{k}{j} 2^{k-j} \Bigg\{ \sum_{\substack{m_{1}+m_{2}+m_{3}=j \\ 0 \leq m_{1},m_{2},m_{3} \leq j}} \frac{j!}{m_{1}! m_{2}! m_{3}!} \\
					&\qquad \times \sum_{\substack{n_{1}+n_{2}+n_{3}=l-1 \\ 0 \leq n_{1},n_{2},n_{3} \leq l-1}} \frac{(l-1)!}{n_{1}! n_{2}! n_{3}!} \prod_{i=1}^{3}\left\|\mathcal{M}^{n_{i}}\big( P^{m_{i}} (v)\big)\right\|_{\widehat{H}^{r}_{1+\theta}(I_{x_{0}})}\Bigg\}\\
					&\lesssim (|t_{0}|+\epsilon)^{\frac{2}{3}} \sum_{j=0}^{k} \binom{k}{j} 2^{k-j} \Bigg\{ \sum_{\substack{m_{1}+m_{2}+m_{3}=j \\ 0 \leq m_{1},m_{2},m_{3} \leq j}} \frac{j!}{m_{1}! m_{2}! m_{3}!} \\
					&\qquad \times \sum_{\substack{n_{1}+n_{2}+n_{3}=l-1 \\ 0 \leq n_{1},n_{2},n_{3} \leq l-1}} \frac{(l-1)!}{n_{1}! n_{2}! n_{3}!} \alpha^{j+l-1} \,(m_{1}+n_{1})!(m_{2}+n_{2})!(m_{3}+n_{3})! \Bigg\}\\
					&\sim  (|t_{0}|+\epsilon)^{\frac{2}{3}} \sum_{j=0}^{k} \binom{k}{j} 2^{k-j} \alpha^{j+l-1} j! (l-1)!  \sum_{\substack{m_{1}+m_{2}+m_{3}=j, \\ 0\leq m_{1},m_{2},m_{3}\leq j}} \sum_{\substack{n_{1}+n_{2}+n_{3}=l-1 \\  0\leq n_{1},n_{2},n_{3}\leq l-1}} \prod_{i=1}^{3} \binom{m_{i}+n_{i}}{m_i},\\
					&\lesssim  (|t_{0}|+\epsilon)^{\frac{2}{3}}(k+l+1)! \alpha^{k+l+1}_{8},		
				\end{split}
			\end{equation*}
			where we have used that 
			\begin{equation*}
				\sum_{\substack{m_{1}+m_{2}+m_{3}=j, \\ 0\leq m_{1},m_{2},m_{3}\leq j}} \sum_{\substack{n_{1}+n_{2}+n_{3}=l-1 \\  0\leq n_{1},n_{2},n_{3}\leq l-1}} \prod_{i=1}^{3} \binom{m_{i}+n_{i}}{m_i}= \binom{j+l+1}{2} \binom{j+l-1}{j}.
			\end{equation*}
			Now, we estimate the remaining terms. The treatment of $\Pi_{2}$ is straightforward; indeed, the inductive hypothesis directly yields
			\begin{equation*}
				\Pi_{2} \lesssim \alpha^{k+l-1}(k+l-1)!\lesssim (k+l+1)! \alpha^{k+l+1}_{8}.
			\end{equation*}
			Finally, we consider $\Pi_{3}$. We observe that for $l \geq 3$, the operator satisfies the identity
			\begin{equation*}
				\mathcal{M}^{l-2}(x\partial_{x})=x\partial_{x}\mathcal{M}^{l-2} + (l-2)\mathcal{M}^{l-2}.
			\end{equation*}
			This relation facilitates the derivation of the following upper bound:
			\begin{equation*}
				\begin{split}
					\Pi_{3}& \leq \left\|x\partial_{x}\mathcal{M}^{l-2}v_{k}\right\|_{\widehat{H}^{r}_{1+\theta}(I_{x_{0}})}+(l-2)\left\|\mathcal{M}^{l-2}v_{k}\right\|_{\widehat{H}^{r}_{1+\theta}(I_{x_{0}})}\\
					&=\Pi_{3,1}+\Pi_{3,2}.
				\end{split}
			\end{equation*}
			The term $\Pi_{3,2}$ is handled directly by the inductive hypothesis, yielding $$\Pi_{3,2} \lesssim \alpha^{k+l-2}(k+l-2)!.$$ 
			
			However, the estimate for $\Pi_{3,1}$ requires a more detailed localization argument. Let $\eta \in C^{\infty}_{0}(\mathbb{R})$ be a cutoff function such that $\eta \equiv 1$ on $I_{x_{0}}$ and $\text{supp}\,\eta \subseteq 2I_{x_{0}}$.
			
			For any $\delta > 0$ and a fixed $t \in J_{t_{0}}$, by the definition of the restricted norm as an infimum over extensions, there exists a function $\Psi_\delta(\cdot, t) \in \widehat{H}^{r}_{1+\theta}(\mathbb{R})$ such that $\Psi_\delta|_{I_{x_{0}}} = \mathcal{M}^{l-1} v_k|_{I_{x_{0}}}$ and
			\begin{equation}\label{deltadef}
				\|\Psi_\delta(\cdot, t)\|_{\widehat{H}^{r}_{1+\theta}(\mathbb{R})} < \|\mathcal{M}^{l-1}v_{k}(\cdot, t)\|_{\widehat{H}^{r}_{1+\theta}(I_{x_{0}})} + \delta.
			\end{equation}
			Since $\eta \equiv 1$ on $I_{x_{0}}$, the function $x \eta(x) t^{-1/3} \Psi_{\delta}$ is as a global extension for the term $x \partial_x \mathcal{M}^{l-1} v_{k}$ restricted to $I_{x_{0}}$. Thus, it follows that:
			\begin{equation*}
				\begin{aligned}
					\|x\partial_{x}\mathcal{M}^{l-2}v_{k}\|_{\widehat{H}^{r}_{1+\theta}(I_{x_{0}})} &\leq \|x\eta(x) t^{-1/3}\Psi_\delta(\cdot, t)\|_{\widehat{H}^{r}_{1+\theta}(\mathbb{R})} \\
					&\leq |t|^{-1/3} \left\| \langle \xi\rangle^{1+\theta} \left( \mathcal{F}_x(x\eta) * \mathcal{F}_x(\Psi_\delta) \right) \right\|_{L^{r'}_{\xi}(\mathbb{R})} \\
					&\leq (|t_{0}| - \epsilon)^{-1/3} \|\mathcal{F}_x(x\eta)\|_{L^1_\xi(\mathbb{R})} \|\Psi_\delta(\cdot, t)\|_{\widehat{H}^r_{1+\theta}(\mathbb{R})}.
				\end{aligned}
			\end{equation*}
			Substituting the inequality \eqref{deltadef} into the above estimate, we obtain:
			\begin{equation*}
				\|x\partial_{x}\mathcal{M}^{l-2}v_{k}\|_{\widehat{H}^{r}_{1+\theta}(I_{x_{0}})} \leq c_{t_{0}, \epsilon, \eta} \left( \|\mathcal{M}^{l-1}v_{k}(\cdot, t)\|_{\widehat{H}^{r}_{1+\theta}(I_{x_{0}})} + \delta \right).
			\end{equation*}
			Finally, by taking the limit as $\delta \to 0$ and passing to the supremum over $t \in J_{t_0}$, we conclude that
			\begin{equation*}
				\begin{aligned}
					\Pi_{3,1} &\leq C_{t_0, \epsilon, \eta} \|\mathcal{M}^{l-1}v_{k}(\cdot, t)\|_{\widehat{H}^{r}_{1+\theta}(I_{x_{0}})} \\
					&\lesssim \alpha^{k+l-1}(k+l-1)!\lesssim (k+l+1)! \alpha^{k+l+1}_{8},
				\end{aligned}
			\end{equation*}
			where the penultimate inequality follows from the inductive hypothesis established in {\sc Claim} \ref{c7}.
			
			Since this holds for arbitrary $\delta > 0$, we let $\delta \to 0$. Finally, taking the supremum over $t \in J_{t_0}$ and applying the inductive bound from {\sc Claim }\ref{c7}, we conclude the estimate for $\Pi_{3,1}$.
		\end{claimproof}
        \setcounter{claim}{8}
		\begin{claim}\label{c8}
			Let $x_{0} \in \mathbb{R}$, $t_{0} \in \mathbb{R}\setminus\{0\}$, and $0<\epsilon<|t_0|$. Let $I_{x_{0}}$ and $J_{t_0}$ be defined as in Claim \ref{c7}. For $1 < r \leq 2$ with conjugate exponent $r'$, assume there exist positive constants $c$ and $\alpha_{8}$ such that the bound
			\begin{equation*}
				\sup_{t \in J_{t_0}} \left\| \partial_{x}^{l} P^{k} v(t, \cdot) \right\|_{\widehat{H}^{r}_{1+\theta}(I_{x_{0}})} \leq c \alpha_{8}^{k+l} (k+l)!
			\end{equation*}
			holds for all $k, l \in \mathbb{N}_{0}$. Then, there exist positive constants $c$ and $\alpha_{9}$ such that the estimate
			\begin{equation*}
				\sup_{t \in J_{t_0}} \left\| \partial_{t}^{m} \partial_{x}^{l} v(t, \cdot) \right\|_{\widehat{H}^{r}_{1+\theta}(I_{x_{0}})} \leq c \alpha_{9}^{m+l} (m+l)!
			\end{equation*}
			is satisfied for all $m, l \in \mathbb{N}_{0}$.
		\end{claim}
		\begin{claimproof}
			First, we show that for any fixed $t \in J_{t_{0}}$, there exist positive constants $c$, $\alpha_{10}$, and $\alpha_{11}$ such that
			\begin{equation*}
				\left\|\left( x \partial_{x}\right)^{m} \partial_{x}^{l} P^{k} v(t) \right\|_{\widehat{H}^{r}_{1+\theta}(I_{x_{0}})} \leq c \alpha_{10}^{k+l+m} \alpha_{11}^{m} (k+m+l)! \quad \text{for all } k, l, m \geq 0.
			\end{equation*}
			The argument of proof will be as in the previous claim, i,e., induction on the variable $m$.
			
			The case $m=0$ is valid since, by hypothesis 
			\begin{equation*}
				\left\| \partial_{x}^{l} P^{k} v(t, \cdot) \right\|_{\widehat{H}^{r}_{1+\theta}(I_{x_{0}})} \leq c \alpha_{8}^{k+l} (k+l)!\quad \text{for all } k, l\geq 0.
			\end{equation*}
			Now,  for the case $m=1$, we have 
			\begin{equation*}
				\begin{split}
					\left\|\left( x \partial_{x}\right) \partial_{x}^{l} P^{k} v(t) \right\|_{\widehat{H}^{r}_{1+\theta}(I_{x_{0}})} &=\left\| x \partial_{x}^{l+1} P^{k} v(t) \right\|_{\widehat{H}^{r}_{1+\theta}(I_{x_{0}})}
				\end{split}
			\end{equation*}
			However, the estimate for this term requires a more detailed localization argument. Let $\eta \in C^{\infty}_{0}(\mathbb{R})$ be a cutoff function such that $\eta \equiv 1$ on $I_{x_{0}}$ and $\text{supp}\,\eta \subseteq 2I_{x_{0}}$.
			
			For any $\delta > 0$ and a fixed $t \in J_{t_{0}}$, by the definition of the restricted norm as an infimum over extensions, there exists a function $\Psi_\delta(\cdot, t) \in \widehat{H}^{r}_{1+\theta}(\mathbb{R})$ such that $\Psi_\delta|_{I_{x_{0}}} = \partial_{x}^{l+1} P^{k} v(t)|_{I_{x_{0}}}$ and
			\begin{equation}\label{deltadef2}
				\left\|\Psi_\delta(\cdot, t)\right\|_{\widehat{H}^{r}_{1+\theta}(\mathbb{R})} < \left\|\partial_{x}^{l+1} P^{k} v(t)\right\|_{\widehat{H}^{r}_{1+\theta}(I_{x_{0}})} + \delta.
			\end{equation}
			Since $\eta \equiv 1$ on $I_{x_{0}}$, the function $x \eta(x) \Psi_{\delta}$ serves as a global extension for the term $x \partial_{x}^{l+1} P^{k} v(t)$ restricted to $I_{x_{0}}$. Thus, using Young's convolution inequality, it follows that:
			\begin{equation*}
				\begin{aligned}
					\left\|x \partial_{x}^{l+1} P^{k} v(t)\right\|_{\widehat{H}^{r}_{1+\theta}(I_{x_{0}})} &\leq \left\|x\eta(x) \Psi_\delta(\cdot, t)\right\|_{\widehat{H}^{r}_{1+\theta}(\mathbb{R})} \\
					&\leq \left\| \langle \xi\rangle^{1+\theta} \left( \mathcal{F}_x(x\eta) * \mathcal{F}_x(\Psi_\delta) \right) \right\|_{L^{r'}_{\xi}(\mathbb{R})} \\
					&\leq \left\|\mathcal{F}_x(x\eta)\right\|_{L^{1}_\xi(\mathbb{R})} \left\|\Psi_\delta(\cdot, t)\right\|_{\widehat{H}^r_{1+\theta}(\mathbb{R})}.
				\end{aligned}
			\end{equation*}
			Substituting the inequality \eqref{deltadef2} into the above estimate, we obtain:
			\begin{equation*}
				\left\|x \partial_{x}^{l+1} P^{k} v(t)\right\|_{\widehat{H}^{r}_{1+\theta}(I_{x_{0}})} \leq c_{\eta} \left( \left\|\partial_{x}^{l+1} P^{k} v(t)\right\|_{\widehat{H}^{r}_{1+\theta}(I_{x_{0}})} + \delta \right),
			\end{equation*}
			where $c_{\eta} := \|\mathcal{F}_x(x\eta)\|_{L^1_\xi(\mathbb{R})}$ is a constant depending only on the choice of the cutoff function.
			
			Finally, by taking the limit as $\delta \to 0$ and passing to the supremum over $t \in J_{t_0}$, we conclude that
			\begin{equation*}
				\begin{aligned}
					\sup_{t \in J_{t_0}} \left\|x \partial_{x}^{l+1} P^{k} v(t)\right\|_{\widehat{H}^{r}_{1+\theta}(I_{x_{0}})} &\leq c_{\eta} \sup_{t \in J_{t_0}}\left \|\partial_{x}^{l+1} P^{k} v(t)\right\|_{\widehat{H}^{r}_{1+\theta}(I_{x_{0}})} \\
					&\leq c_{\eta} c \alpha_{8}^{k+l+1}(k+l+1)!,
				\end{aligned}
			\end{equation*}
			where the last inequality follows from applying the assumed hypothesis.
			Assume as our inductive hypothesis that the result holds up to $m$. More precisely,
			\begin{equation*}
				\left\| (x \partial_{x})^{j} \partial_{x}^{l} P^{k} v(t) \right\|_{\widehat{H}^{r}_{1+\theta}(I_{x_{0}})} \leq c \alpha_{10}^{k+l+j} \alpha_{11}^{j} (k+j+l)! \quad \text{for all } k, l \in \mathbb{N}_0 \text{ and } 0 \leq j \leq m.
			\end{equation*}
			Applying an argument analogous to the base case, we expand $(x\partial_x + I)^m$ and invoke our cutoff bound to obtain
			\begin{equation*}
				\begin{aligned}
					\left\| (x \partial_{x})^{m+1} \partial_{x}^{l} P^{k} v(t) \right\|_{\widehat{H}^{r}_{1+\theta}(I_{x_{0}})} 
					&= \left\| x \left( x\partial_{x} + I \right)^{m} \partial_{x}^{l+1} P^{k} v(t) \right\|_{\widehat{H}^{r}_{1+\theta}(I_{x_{0}})}\\
					&\leq \sum_{m_{1}=0}^{m} \binom{m}{m_1} \left\| x (x\partial_{x})^{m_{1}} \partial_{x}^{l+1} P^{k} v(t) \right\|_{\widehat{H}^{r}_{1+\theta}(I_{x_{0}})}\\
					&\leq c_{\eta} \sum_{m_{1}=0}^{m} \binom{m}{m_1} \left\| (x\partial_{x})^{m_{1}} \partial_{x}^{l+1} P^{k} v(t) \right\|_{\widehat{H}^{r}_{1+\theta}(I_{x_{0}})}\\
					&\leq c_{\eta}  \sum_{m_{1}=0}^{m} \binom{m}{m_1} \alpha_{10}^{k+l+m_1+1} \alpha_{11}^{m_1} (k+m_1+l+1)!\\
					&\leq  c \alpha_{10}^{k+l+m+1} \alpha_{11}^{m+1} (k+m+l+1)! \left[ \frac{c_{\eta}}{\alpha_{11}} \sum_{j=0}^{m} \frac{1}{j!} \left(\frac{1}{\alpha_{10}\alpha_{11}}\right)^{j} \right] \\
					&\leq c \alpha_{10}^{k+l+m+1} \alpha_{11}^{m+1} (k+m+l+1)! \left[ \frac{c_{\eta}}{\alpha_{11}} \exp\left(\frac{1}{\alpha_{10}\alpha_{11}}\right) \right].
				\end{aligned}
			\end{equation*}
			By choosing the constant $\alpha_{11} > 0$ sufficiently large such that $\frac{c_{\eta}}{\alpha_{11}} \exp\left(\frac{1}{\alpha_{10}\alpha_{11}}\right) \leq 1$, we recover the required bound for $m+1$. This closes the induction.
			
			Next, we show that there exists positive constants  $c$ and $\alpha_{12}$ such that:
			\begin{equation*}
				\left\|(t\partial_{t})^{m}\partial_{x}^{l} v(t, \cdot) \right\|_{\widehat{H}^{r}_{1+\theta}(I_{x_{0}})}\leq c\alpha_{11}^{l+m}(l+m)!,\quad l,m=0,1,2,\cdots,.
			\end{equation*} 
			Since $t\partial_{t} = \frac{1}{3}(P - x\partial_{x})$, it is clear that
			\begin{equation*}
				\begin{aligned}
					\left\|(t\partial_{t})^{m}\partial_{x}^{l} v(t, \cdot) \right\|_{\widehat{H}^{r}_{1+\theta}(I_{x_{0}})} 
					&\leq \frac{1}{3^{m}} \sum_{\substack{m_1 + m_2 = m \\ 0 \leq m_1, m_2 \leq m}} \frac{m!}{m_1! m_2!} \left\| (x\partial_{x})^{m_1}\partial_{x}^{l}(P-I)^{m_{2}} v(t, \cdot) \right\|_{\widehat{H}^{r}_{1+\theta}(I_{x_{0}})} \\
					&\leq \frac{1}{3^{m}} \sum_{\substack{m_1 + m_2 = m \\ 0 \leq m_1, m_2 \leq m}} \sum_{j=0}^{m_2} \frac{m!}{m_1! j! (m_2 - j)!} \left\| (x\partial_{x})^{m_1} \partial_{x}^{l} P^{j} v(t, \cdot) \right\|_{\widehat{H}^{r}_{1+\theta}(I_{x_{0}})} \\
					&\leq \frac{c}{3^{m}} \sum_{\substack{m_1 + m_2 = m \\ 0 \leq m_1, m_2 \leq m}} \sum_{j=0}^{m_2} \frac{m!}{m_1! j! (m_2 - j)!} \alpha_{10}^{j+l+m_1} \alpha_{11}^{m_1} (j+m_1+l)!
				\end{aligned}
			\end{equation*}
			By the monotonicity of the factorial and the powers of $\alpha$, we observe that for all $j \leq m_2$ and $m_1 + m_2 = m$, we have
			\begin{equation*}
				\alpha_{10}^{j+l+m_1} \alpha_{11}^{m_1} (j+m_1+l)! \leq \alpha_{10}^{m+l} \alpha_{11}^{m} (m+l)!.
			\end{equation*}
			Substituting this into the sum and extracting the constants, we obtain
			\begin{equation*}
				\begin{aligned}
					\left\|(t\partial_{t})^{m}\partial_{x}^{l} v(t, \cdot) \right\|_{\widehat{H}^{r}_{1+\theta}(I_{x_{0}})} 
					&\leq \frac{c}{3^{m}} \alpha_{10}^{m+l} \alpha_{11}^{m} (m+l)! \left\{ \sum_{m_1 + m_2 = m} \binom{m}{m_1} \sum_{j=0}^{m_2} \binom{m_2}{j} \right\} \\
					&= \frac{c}{3^{m}} \alpha_{10}^{m+l} \alpha_{11}^{m} (m+l)! \left\{ \sum_{m_1 + m_2 = m} \binom{m}{m_1} 2^{m_2} \right\} \\
					&= \frac{c}{3^{m}} \alpha_{10}^{m+l} \alpha_{11}^{m} (m+l)!  3^{m}  \\
					&= c \alpha_{10}^{m+l} \alpha_{11}^{m} (m+l)!.
				\end{aligned}
			\end{equation*}
			This confirms the result for the index $m$, effectively canceling the $3^{-m}$ factor originating from the operator $t\partial_t$.
			
			To remove the explicit $t$ dependence and derive estimates for the pure time derivatives $\partial_{t}^{m}$, we consider the interval $J_{t_{0}} = (t_{0} - \epsilon, t_{0} + \epsilon)$, where we have $|t_{0}| - \epsilon > 0$. We assume the inductive hypothesis:
			\begin{equation*}
				\left\| \partial_{t}^{m} \partial_{x}^{l} P^{k} v(t) \right\|_{\widehat{H}^{r}_{1+\theta}(I_{x_{0}})} \leq c \alpha_{10}^{k+l+m} \alpha_{11}^{m} (k+l+m)!
			\end{equation*}
			holds up to $m$. To prove the case $m+1$, we use the identity $t\partial_{t}^{m+1} = \partial_{t}^{m}(t\partial_{t}) - m\partial_{t}^{m}$. Applying this to $\partial_{x}^{l} P^{k} v$ and taking the norm, since $|t|^{-1} \leq (|t_{0}| - \epsilon)^{-1}$, we obtain:
			\begin{equation*}
				\begin{aligned}
					\left\| \partial_{t}^{m+1} \partial_{x}^{l} P^{k} v(t) \right\|_{\widehat{H}^{r}_{1+\theta}(I_{x_{0}})} &\leq \frac{1}{|t_{0}| - \epsilon} \left\| t\partial_{t}^{m+1} \partial_{x}^{l} P^{k} v(t) \right\|_{\widehat{H}^{r}_{1+\theta}(I_{x_{0}})} \\
					&\leq \frac{1}{|t_{0}| - \epsilon} \left\{ \left\| \partial_{t}^{m} (t\partial_{t}) \partial_{x}^{l} P^{k} v(t) \right\|_{\widehat{H}^{r}_{1+\theta}(I_{x_{0}})} + m \left\| \partial_{t}^{m} \partial_{x}^{l} P^{k} v(t) \right\|_{\widehat{H}^{r}_{1+\theta}(I_{x_{0}})} \right\}.
				\end{aligned}
			\end{equation*}
			Substituting the operator identity $t\partial_{t} = \frac{1}{3}(P - x\partial_{x})$ directly into the first term, it acts on the entire spatial derivative block:
			\begin{equation*}
				(t\partial_{t}) \left( \partial_{x}^{l} P^{k} v \right) = \frac{1}{3} P \left( \partial_{x}^{l} P^{k} v \right) - \frac{1}{3} x\partial_{x} \left( \partial_{x}^{l} P^{k} v \right) = \frac{1}{3} \partial_{x}^{l} P^{k+1} v - \frac{1}{3} x \partial_{x}^{l+1} P^{k} v.
			\end{equation*}
			Using the spatial localization  $|x| \leq c_{\eta}$ on $I_{x_{0}}$, the norm of this first term expands as:
			\begin{equation*}
				\begin{split}
					\left\| \partial_{t}^{m} \left( \frac{1}{3}\partial_{x}^{l} P^{k+1} - \frac{1}{3}x \partial_{x}^{l+1} P^{k} \right) v(t) \right\|_{\widehat{H}^{r}_{1+\theta}(I_{x_{0}})} &\leq \frac{1}{3} \left\| \partial_{t}^{m} \partial_{x}^{l} P^{k+1} v(t) \right\|_{\widehat{H}^{r}_{1+\theta}(I_{x_{0}})}\\
					&\quad  + \frac{c_{\eta}}{3} \left\| \partial_{t}^{m} \partial_{x}^{l+1} P^{k} v(t) \right\|_{\widehat{H}^{r}_{1+\theta}(I_{x_{0}})}.
				\end{split}
			\end{equation*}
			Now, applying the inductive hypothesis to all three terms (the two from $t\partial_{t}$ and the $+m$ term), we find:
			\begin{equation*}
				\begin{aligned}
					\left\| \partial_{t}^{m+1} \partial_{x}^{l} P^{k} v(t) \right\|_{\widehat{H}^{r}_{1+\theta}(I_{x_{0}})} &\leq \frac{c \alpha_{11}^{m} \alpha_{10}^{k+l+m+1} (k+l+m+1)!}{|t_{0}| - \epsilon} \left\{ \frac{1}{3} + \frac{c_{\eta}}{3} + \frac{m}{\alpha_{10}(k+l+m+1)} \right\} \\
					&\leq c \alpha_{10}^{k+l+m+1} \alpha_{11}^{m+1} (k+l+m+1)! \left\{\frac{1 + c_{\eta} + 3\alpha_{10}^{-1}}{3(|t_{0}| - \epsilon)\alpha_{11}} \right\}.
				\end{aligned}
			\end{equation*}
			Assuming $\alpha_{10} \geq 1$, we can bound $3\alpha_{10}^{-1} \leq 3$. By choosing the constant $\alpha_{11}$ large enough such that:
			\begin{equation*}
				\alpha_{11} \geq \frac{4 + c_{\eta}}{3(|t_{0}| - \epsilon)},
			\end{equation*}
			the term in the brackets becomes $\leq 1$. This concludes the induction and removes the $t$ dependence from the estimate.
				\end{claimproof}
		\end{proof}
	With the analyticity of the flow map for the modified Korteweg-de Vries equation (\textit{mKdV}) established, we now extend our considerations to the cubic KdV equation. This transition is motivated by the fact that the cubic interaction, while structurally more complex, shares a similar dispersive soul with the \textit{mKdV} through the underlying linear operator $\partial_{t} + \partial_{x}^{3}$.
	
	The following proof adapts the previously established \emph{Claim-based framework} to handle the higher-order multi-linear interactions. In particular, we demonstrate that the sequence space techniques and the sharp estimates in the $X^{r}_{s,b}$ framework are robust enough to capture the analytic smoothing effects for the cubic non-linearity. By establishing a correspondence between the estimates derived in \Cref{apendiceA} and the requirements of the cubic flow, we provide a unified treatment of these rough Sobolev regimes.
	
	We begin by establishing the necessary multi-linear bounds for the cubic term, following a series of claims that parallel the methodology utilized in the modified case.
	To demonstrate the geometric precision and versatility of the continuous scaling flow framework, we present two distinct families of initial data. These examples serve as canonical mathematical models, showing how the time-parameter complexification seamlessly absorbs local spatial singularities and global frequency behaviors.
    \section{Proof of Theorem \ref{main2}}\label{proofthmB}

	\begin{proof}
	The proof is divided into a sequence of claims. We begin by collecting several auxiliary results for the local well-posedness theory of the generalized Korteweg-de Vries equation ($\mathrm{gKdV}$-3), which will be used in the proof of Claim~1.

Our first goal is to establish the local well-posedness of the following system, corresponding to the case $m=4$ in the general formulation \eqref{SYS1}:
\begin{equation}\label{System1}
\begin{cases}
\partial_t v_k+\partial_x^3v_k+\Pi_k^{(4)}(v_0,\ldots,v_k)=0,
& (x,t)\in\mathbb{R}\times\mathbb{R},\\
v_k(0,x)=\phi_k(x)=(x\partial_x)^k\phi_0(x),
& k=0,1,2,\ldots.
\end{cases}
\end{equation}
Here $v_0:=v$, and the nonlinear operator $\Pi_k^{(4)}(v_0,\ldots,v_k)$ is defined by \eqref{Pikm}.

 In the proof of the local well-posedness result, we employ the associated Bourgain space $X^{s,b}$ and its localized version $X_{T}^{s,b}$, both of which correspond precisely to the case $r=2$ in the notation of the spaces $X^{r}_{s,b}$ introduced in Section 1.2. The same observation applies to the spaces $\mathcal{A}_{A_0}(X^{s,b})$. The proof of the next claim relies on the following two lemmas:
	\begin{lema}(\cite{KPV2})\label{LemaB1}
		Let $s \in \mathbb{R}$, $b', b \in \left(\frac{1}{2}, \frac{7}{8}\right)$ with $b < b'$ and $\delta \in (0, 1)$. Then, for $v \in X^{s,b'-1}$, we have
		\begin{equation*}
			\|\psi_{\delta}v\|_{X^{s,b-1}} \leq c \, \delta^{\frac{b'-b}{8(1-b)} }\, \|v\|_{X^{s,b'-1}}.
		\end{equation*}
	\end{lema}
	\begin{lema}(\cite{Gru2005})\label{LemaB}
		For $- \frac{1}{6} < s \leq 0$ and $- \frac{1}{2} < b'-1 < s - \frac{1}{3}$ and
		$b > \frac{1}{2}$, we have
		\begin{equation}\label{EstimateB}
			\left\|\partial_x (u_1 u_2 u_3 u_4) \, \right\|_{X^{s,b'-1}} 
			\;\leq\; 
			c \, \prod_{i=1}^4 \| u_i \|_{X^{s,b}}.  
		\end{equation}
		Additionally, for $s \geq 0$,  $- \frac{1}{2} < b'-1 < - \frac{1}{3}$ and  $b > \frac{1}{2}$ the estimate \eqref{EstimateB} also holds true. 
	\end{lema}
	The well-posedness result for system \eqref{System1} is the following. 
	
\setcounter{claim}{0}
\begin{claim}\label{PropA}
Let $-1/6<s \leq 0$ and $b>\tfrac12$. Suppose $\phi_k\in H^s(\mathbb R)$, $k=0,1,\dots$, satisfy
\[
\|\vec\phi\|_{\A_{A_0}}:=\sum_{k=0}^\infty\frac{A_0^k}{k!}\,\|\phi_k\|_{H^s}<\infty .
\]
 Then: 
\begin{enumerate}
		\item \underline{Existence and Uniqueness:} There exist $T=T(\|\vec\phi\|_{\A_{A_0}})>0$ and a unique solution $\mathbf v(t)=(v_k(t))_{k\ge0}$ of system \eqref{System1} with $v_k\in C([-T,T];H^s)\cap X_T^{s,b}$ and $\sum_{k}\frac{A_0^k}{k!}\|v_k\|_{X_T^{s,b}}<\infty$.
		\item \underline{Uniform Bound:} The solution $\mathbf{v}$ satisfies the analytical estimate:
		\begin{equation*}
			\|\mathbf{v}\|_{\mathcal{A}_{A_0}(X^{s,b}_T)} \leq c \|\boldsymbol{\phi}\|_{\mathcal{A}_{A_0}\left(H^{s}(\mathbb{R})\right)}.
		\end{equation*}
		
 \item \underline{Lipschitz continuity:} If $\boldsymbol{\phi},\tilde{\boldsymbol{\phi}}$ both satisfy $\|\boldsymbol{\phi}\|_{\A_{A_0}},\|\tilde{\boldsymbol{\phi}}\|_{\A_{A_0}}\le R'$ and $\mathbf v,\widetilde{\mathbf v}$ are the corresponding solutions on the common interval $[-T(R'),T(R')]$, then
\[
\|\mathbf v-\widetilde{\mathbf v}\|_{\A_{A_0}(X_T^{s,b})}+\|\mathbf v-\widetilde{\mathbf v}\|_{C([-T,T];\A_{A_0}(H^s))}\le C(T)\,\|\boldsymbol{\phi}-\tilde{\boldsymbol{\phi}} \|_{\A_{A_0}(H^s)};
\]
\end{enumerate}
\end{claim}
\begin{claimproof}
The proof is identical to that of Claim 1 in Theorem A, relying on the linear estimate in \cite[Lemma 3.2]{KPV2} together with Lemmas \ref{LemaB1} and \ref{LemaB}. We therefore omit the details. 
\end{claimproof}
\begin{corolario}\label{CoroA}
Let $-1/6<s \leq 0$, $b>1/2$. Suppose $\phi_k:=(x\partial_x)^k\phi_0\in H^s(\mathbb R)$, $k=0,1,\dots$, satisfies $\|\vec\phi\|_{\A_{A_0}}<\infty$. Then there exist $T=T(\|\vec\phi\|_{\A_{A_0}})>0$ and a unique solution $v$ of the cubic gKdV  equation with $v\in C([-T,T];H^s)\cap X_T^{s,b}$ and $\sum_k\frac{A_0^k}{k!}\|P^kv\|_{X_T^{s,b}}<\infty$; moreover $\phi_0 \mapsto v(t)$ is Lipschitz continuous in the following sense:
\[
\|P^kv-P^k\tilde v\|_{X_T^{s,b}}+\|P^kv-P^k\tilde v\|_{C([-T,T];H^s)}\le\frac{k!}{A_0^k}\,C(T)\sum_{m=0}^\infty\frac{A_0^m}{m!}\|(x\partial_x)^m(\phi_0-\tilde{\phi_0})\|_{H^s}.
\]
\end{corolario}
\begin{proof}
Apply Proposition~\ref{PropA} to $\phi_k := (x\partial_x)^k\phi_0$. Since estimate~(3) is a series in $k$ with non-negative terms, we may isolate the $k$-th term and divide both sides by $A_0^k/k!$, yielding the desired inequality.
\end{proof}
\begin{obs}
By virtue of Corollary \ref{CoroA}, the components $v_k$ of the infinite system \eqref{System1} can be rigorously identified with $P^k v$, where $v$ is the solution to gKdv-3. Although this relation is classical for smooth solutions, its extension to the low-regularity setting of Proposition \ref{PropA} follows from a standard regularization argument: approximating the initial data by smooth functions yields smooth solutions for which $v_{k,n} = P^k v_n$ holds, and the limit as $n \to \infty$ preserves this identity in the weak sense due to the Lipschitz continuous dependence.
\end{obs}
The next result will be used repeatedly in the following claims.
\begin{lema}\label{LemaC}
		Let $u$ belong to the Bourgain space $X^{0,b}$ associated to the KdV/Airy group. Then there exists a constant $c=c(b)$ such that
		\begin{equation*}
			\|u\|_{L^{4}_{xt}} \leq c \|u\|_{X^{0,b}}, \;\;\; b > 1/3,
		\end{equation*}
		and 
		\begin{equation*}
			\|u\|_{L^{8}_{xt}} \leq c \|u\|_{X^{0,b}}, \;\;\; b > 1/2. 
		\end{equation*}  
		If the estimates are stated on a finite time interval $[0,T]$, write $X^{0,b}_T$ and allow $c=c(b,T)$.
	\end{lema}
	\begin{proof}
		See \cite{Gru2005}. 
	\end{proof}
\setcounter{claim}{1}
		\begin{claim}\label{ct2.1}
			Let $\vec{\boldsymbol{\epsilon}} = (\epsilon, \epsilon) \in (0, \infty)^{2}$ and let $\chi_{(t_{0},x_{0}),\boldsymbol{\vec{\epsilon}}}$ be the cutoff function defined in \eqref{cutoff1}. There exist constants $c > 0,$  such that the following estimate holds for all $k \in \mathbb{N}_0$:
			\begin{equation*}
				\left\| \chi_{(t_{0},x_{0}),\boldsymbol{\vec{\epsilon}}} \,v_{k} \right\|_{L^{2}_{xt}(\mathbb{R}^2)}  \leq c \alpha^{k}_{1} k!.
			\end{equation*}
		\end{claim}
		\begin{claimproof}
			From Lemma \ref{LemaC}, we obtain 
			\begin{equation*}
				\left\| \chi_{(t_{0},x_{0}),\boldsymbol{\vec{\epsilon}}} \,v_{k} \right\|_{L^{2}_{xt}(\mathbb{R}^2)}  \lesssim \|  \chi_{(t_{0},x_{0}),\boldsymbol{\vec{\epsilon}}} \|_{L^{4}_{xt}(\mathbb{R}^2)} \|v_k\|_{L^{4}_{xt}(\mathbb{R}^2)} \lesssim \|v_k\|_{X^{0,b}}.
			\end{equation*}
			After combing Lemma \ref{LemaC} with Proposition \ref{PropA} it  yields
			\begin{equation*}
				\left\| \chi_{(t_{0},x_{0}),\boldsymbol{\vec{\epsilon}}} \,v_{k} \right\|_{L^{2}_{xt}(\mathbb{R}^2)} \lesssim\alpha_{1}^{k} k!,\quad \forall\,\, k\in \mathbb{N}_{0}.
			\end{equation*}
		\end{claimproof}
		\begin{claim}\label{ct2.2}
			Let $\vec{\boldsymbol{\epsilon}} = (\epsilon, \epsilon) \in (0, \infty)^{2}$ and let $\chi_{(t_{0},x_{0}),\boldsymbol{\vec{\epsilon}}}$ be the cutoff function defined in \eqref{cutoff1}. There exist constants $c > 0,$  such that the following estimate holds for all $k \in \mathbb{N}_0$:
			\begin{equation*}
				\left\| \chi_{(t_{0},x_{0}),\boldsymbol{\vec{\epsilon}}} \,v_{k} \right\|_{H^{1/2}_{xt}(\mathbb{R}^2)} \leq c \alpha^{k}_{1} k!.
			\end{equation*}
		\end{claim}
		\begin{claimproof}
			Denote $v_{k}:= P^{k} v$. Taking $\mu = 1/2$ in Lemma \ref{LemaA}, we have that
			\begin{equation*}
				\begin{split}
					\left\|\langle \nabla_{t,x} \rangle^{3}\left(  \chi_{(t_{0},x_{0}),\boldsymbol{\vec{\epsilon}}} v_k\right) \right\|_{H^{- 5/2}_{xt}(\mathbb{R}^{2})} & = \left\|\langle \nabla_{t,x} \rangle^{1/2} \left( \chi_{(t_{0},x_{0}),\boldsymbol{\vec{\epsilon}}}v_k\right) \right\|_{L^{2}_{xt}(\mathbb{R}^2)} \\
					& \lesssim \left\|  \chi_{(t_{0},x_{0}),\boldsymbol{\vec{\epsilon}}} v_k \right\|_{H^{- 5/2}_{xt}(\mathbb{R}^{2})} + \left\| t \partial_{x}^{3} ( \chi_{(t_{0},x_{0}),\boldsymbol{\vec{\epsilon}}} v_k) \right\|_{H^{- 5/2}_{xt}(\mathbb{R}^{2})} \\
					&\quad + \left\| P^{3}( \chi_{(t_{0},x_{0}),\boldsymbol{\vec{\epsilon}}} v_k)\right \|_{H^{-5/2}_{xt}(\mathbb{R}^{2})}.
				\end{split}
			\end{equation*}
			Notice that  by {\sc claim } \ref{ct2.1} it follows that  for all $k\in\mathbb{N}_{0}$:
			\begin{equation*}
				\left\|  \chi_{(t_{0},x_{0}),\boldsymbol{\vec{\epsilon}}} v_k \right\|_{H^{-5/2}_{xt}(\mathbb{R}^{2})}
				\lesssim
				\left\|  \chi_{(t_{0},x_{0}),\boldsymbol{\vec{\epsilon}}} v_k \right\|_{L^{2}_{xt}(\mathbb{R}^2)}
				\lesssim
				A_{0}^{k}\, k!.
			\end{equation*}
			Additionally, by the Leibniz rule it follows that 
			\begin{equation*}
				\begin{aligned}
					\left\| P^{3}\left( \chi_{(t_{0},x_{0}),\boldsymbol{\vec{\epsilon}}} v_k\right) \right\|_{H^{-5/2}_{xt}(\mathbb{R}^{2})}
					&\lesssim
					\sum_{j = 0}^{3} \binom{3}{j}\,\left\| P^{3-j}  \chi_{(t_{0},x_{0}),\boldsymbol{\vec{\epsilon}}} \,P^{j} v_k \right\|_{L^{2}_{xt}(\mathbb{R}^2)}\\
					&\lesssim
					\sum_{j = 0}^{3} \binom{3}{j}\,\left\| P^{3 - j}  \chi_{(t_{0},x_{0}),\boldsymbol{\vec{\epsilon}}} \right\|_{L^{\infty}_{xt}(\mathbb{R}^2)} \,\|v_{j +k}\|_{L^{2}_{xt}} \\
					&\lesssim A_{0}^{k +3} (k +3)!. 
				\end{aligned}
			\end{equation*}
			Additionally,
			\begin{equation}\label{eqt2.1}
				\begin{aligned}
					\left\| t\,\partial_{x}^{3} \left(\chi_{(t_{0},x_{0}),\boldsymbol{\vec{\epsilon}}}  v_k\right) \right\|_{H^{-5/2}_{xt}(\mathbb{R}^{2})}
					&\lesssim
					\left\|\chi_{(t_{0},x_{0}),\boldsymbol{\vec{\epsilon}}} \, t\,\partial_{x}^3 v_k\right\|_{H^{-5/2}_{xt}(\mathbb{R}^{2})}+ \left\|\partial_{x}^{2}\left(t\,\partial_x \chi_{(t_{0},x_{0}),\boldsymbol{\vec{\epsilon}}}  v_k\right)\right\|_{H^{-5/2}_{xt}(\mathbb{R}^{2})}\\
					&\quad+\left \|\partial_{x}\left(t\,\partial_{x}^2 \chi_{(t_{0},x_{0}),\boldsymbol{\vec{\epsilon}}} v_k\right)\right\|_{H^{-5/2}_{xt}(\mathbb{R}^{2})}+ \left\| t\,\left(\partial_x^{3} \chi_{(t_{0},x_{0}),\boldsymbol{\vec{\epsilon}}} \right)\,v_k \right\|_{H^{-5/2}_{xt}(\mathbb{R}^{2})} \\
					&= \mathcal{I}_{3,1} + \mathcal{I}_{3,2} + \mathcal{I}_{3,3} + \mathcal{I}_{3,4}. 
				\end{aligned}
			\end{equation}
			For estimate the term $\mathcal{I}_{3,1}$ we use the identity
			\begin{equation*}
				\chi_{(t_{0},x_{0}),\boldsymbol{\vec{\epsilon}}} t \partial_{x}^3 v_k = - \frac{\chi_{(t_{0},x_{0}),\boldsymbol{\vec{\epsilon}}}}{3} P v_k + \frac{\chi_{(t_{0},x_{0}),\boldsymbol{\vec{\epsilon}}}}{3}  x \partial_x v_k - t \chi_{(t_{0},x_{0}),\boldsymbol{\vec{\epsilon}}} \Pi_{k}^{(4)}(v).
			\end{equation*}
			Thus, for all $k \in \mathbb{N}_{0}$:
			$$
			\begin{aligned}
				\left\| \chi_{(t_{0},x_{0}), \vec{\epsilon}} P v_k \right\|_{H^{-5/2}_{xt}(\mathbb{R}^{2})}
				&\le \left\| \chi_{(t_{0},x_{0}), \vec{\epsilon}} P v_k \right\|_{L^{2}_{xt}(\mathbb{R}^2)} \\
				&= \left\| \chi_{(t_{0},x_{0}), \vec{\epsilon}} v_{k+1} \right\|_{L^{2}_{xt}(\mathbb{R}^2)} \\
				&\lesssim  A_{0}^{k} (k+1)!
			\end{aligned}
			$$
			and 
			$$
			\begin{aligned}
				\left\| \chi_{(t_{0},x_{0}), \vec{\epsilon}} x \partial_x v_k \right\|_{H^{-5/2}_{xt}(\mathbb{R}^{2})}
				&\le \left\| \left(\chi_{(t_{0},x_{0}), \vec{\epsilon}} x\right) v_k \right\|_{H^{-3/2}_{xt}(\mathbb{R}^{2})} + \left\| v_k \partial_{x}(x \chi_{(t_{0},x_{0}), \vec{\epsilon}}) \right\|_{H^{-5/2}_{xt}(\mathbb{R}^{2})} \\[4pt]
				&\le \left\| \left(\chi_{(t_{0},x_{0}), \vec{\epsilon}} x\right) v_k \right\|_{L^{1}_{xt}(\mathbb{R}^{2})} + \left\| v_k \partial_{x}(x\chi_{(t_{0},x_{0}), \vec{\epsilon}} ) \right\|_{L^{1}_{xt}(\mathbb{R}^{2})} \\[4pt]
				&\le \left\| \chi_{(t_{0},x_{0}), \vec{\epsilon}} x \right\|_{L^{2}_{xt}(\mathbb{R}^2)} \left\| v_k \right\|_{L^{2}_{xt}(\mathbb{R}^2)}\\
				&\quad  + \left\| \partial_{x}\left(x \chi_{(t_{0},x_{0}), \vec{\epsilon}}\right)\right\|_{L^{\infty}_{xt}(\mathbb{R}^2)} \left\| \chi_{(t_{0},x_{0}), \tilde{\vec{\epsilon}}} v_k \right\|_{L^{2}_{xt}(\mathbb{R}^2)} \\[4pt]
				&\lesssim \left\| \chi_{(t_{0},x_{0}), \tilde{\vec{\epsilon}}} v_k \right\|_{L^{2}_{xt}(\mathbb{R}^2)} \\[4pt]
				&\lesssim A_{0}^{k} k!,
			\end{aligned}
			$$
			where we employ the cut-off function $\chi_{(t_{0},x_{0}), \tilde{\vec{\epsilon}}}$ as defined in \eqref{cutoff1}, specializing to the case $\tilde{\vec{\epsilon}} = (2\epsilon, 2\epsilon)$. The penultimate estimate follows by applying {\sc Claim} \ref{ct2.1} to $\chi_{(t_{0},x_{0}), \tilde{\vec{\epsilon}}} v_k.$
			
			Next,
			\begin{equation*}
				\left\|t \chi_{(t_{0},x_{0}), \vec{\epsilon}} \Pi_{k}^{(4)}(v)\right\|_{H^{- 5/2}_{xt}(\mathbb{R}^2)} \lesssim   \sum_{\substack{k=k_{1}+k_{2}+k_{3}+k_{4} +k_{5}\\ 0 \leq k_{1},k_{2},k_{3},k_{4},k_{5} \leq k}} \frac{2^{k_{1}}k!}{k_{1}!k_{2}!k_{3}!k_{4}!k_{5}!} \,\left\| t  \chi_{(t_{0},x_{0}), \vec{\epsilon}} \partial_{x}(v_{k_2} v_{k_3} v_{k_4} v_{k_5})\right\|_{H^{-5/2}_{xt}(\mathbb{R}^{2})},
			\end{equation*}
			and by H\"older's inequality we get 
			\begin{align*}
				&	\left\| t  \chi_{(t_{0},x_{0}), \vec{\epsilon}} \partial_{x}(v_{k_2} v_{k_3} v_{k_4} v_{k_5})\right\|_{H^{-5/2}_{xt}(\mathbb{R}^{2})} \\
				&\leq \left\| v_{k_2} v_{k_3} v_{k_4} v_{k_5}\left(t \chi_{(t_{0},x_{0}), \vec{\epsilon}}\right)\right\|_{H^{-3/2}_{xt}(\mathbb{R}^{2})} 
				+ \left\| \partial_{x}\left(t  \chi_{(t_{0},x_{0}), \vec{\epsilon}}\right)\,v_{k_2} v_{k_3} v_{k_4} v_{k_5}\right\|_{H^{-5/2}_{xt}(\mathbb{R}^{2})} \\
				&\lesssim \left\| v_{k_2} v_{k_3} v_{k_4} v_{k_5} \left(t \chi_{(t_{0},x_{0}), \vec{\epsilon}}\right)\right\|_{L^{1}_{xt}(\mathbb{R}^2)} 
				+\left \| \partial_{x}\left(t\chi_{(t_{0},x_{0}), \vec{\epsilon}} \right) v_{k_2} v_{k_3} v_{k_4} v_{k_5} \right\|_{L^{1}_{xt}(\mathbb{R}^2)} \\
				&\lesssim \left\| \prod_{i =2}^{5} v_{k_i} \left(t \chi_{(t_{0},x_{0}), \vec{\epsilon}}\right)\right\|_{L^{1}_{xt}(\mathbb{R}^2)}
				+ \left\| \prod_{i =2}^{5}  v_{k_i} \partial_{x}\left(t \chi_{(t_{0},x_{0}), \vec{\epsilon}}\right)\right\|_{L^{1}_{xt}(\mathbb{R}^2)} \\
				&\lesssim  \prod_{i =2}^{5} \left\| v_{k_i}\right\|_{L^{4}_{xt}(\mathbb{R}^2)}  \left\|t \chi_{(t_{0},x_{0}), \vec{\epsilon}}\right\|_{L^{\infty}_{xt}(\mathbb{R}^2)}
				+ \prod_{i =2}^{5} \left\|v_{k_i}\right\|_{L^{4}_{xt}}  \left\|t \partial_x \chi_{(t_{0},x_{0}), \vec{\epsilon}}\right\|_{L^{\infty}_{xt}(\mathbb{R}^2)} \\
				&\sim   \prod_{i =2}^{5} \left\| v_{k_i}\right\|_{L^{4}_{xt}(\mathbb{R}^2)}.
			\end{align*}
			Thus,
			\begin{equation*}
				\left\|t \chi_{(t_{0},x_{0}), \vec{\epsilon}} \Pi_{k}^{(4)}(v)\right\|_{H^{- 5/2}_{xt}(\mathbb{R}^2)} \lesssim \sum_{\substack{k=k_{1}+k_{2}+k_{3}+k_{4} +k_{5}\\ 0 \leq k_{1},k_{2},k_{3},k_{4},k_{5} \leq k}} 2^{k_1}  \frac{k!}{k_1! k_2! k_3! k_4! k_5!} \prod_{i =2}^{5} \left\|  v_{k_i}\right\|_{L^{4}_{xt}(\mathbb{R}^2)},
			\end{equation*}
			from which we get, after applying  Lemma \ref{LemaC} that 
			\begin{align*}
				\left\|t \,  \chi_{(t_{0},x_{0}), \vec{\epsilon}} \, \Pi_{k}^{(4)}(v)\right\|_{H^{-5/2}_{xt}(\mathbb{R}^2)} 
				&\lesssim \sum_{\substack{k=k_{1}+k_{2}+k_{3}+k_{4} +k_{5}\\ 0 \leq k_{1},k_{2},k_{3},k_{4},k_{5} \leq k}}  2^{k_1} \frac{k!}{k_1! k_2! k_3! k_4! k_5!} \prod_{\ell=2}^{5} A_{\ell}^{k_{\ell}} k_{\ell}! \\
				&=\sum_{\substack{k=k_{1}+k_{2}+k_{3}+k_{4} +k_{5}\\ 0 \leq k_{1},k_{2},k_{3},k_{4},k_{5} \leq k}}  2^{k_1} \frac{k!}{k_1!} A_{2}^{k_2} A_{3}^{k_3} A_{4}^{k_4} A_{5}^{k_5} \\
				&\lesssim k! \sum_{\substack{k=k_{1}+k_{2}+k_{3}+k_{4} +k_{5}\\ 0 \leq k_{1},k_{2},k_{3},k_{4},k_{5} \leq k}} \frac{2^{k_1}}{k_1!} A_{6}^{k_2 + k_3 + k_4 + k_5} \\
				&\lesssim k! \sum_{\substack{k=k_{1}+k_{2}+k_{3}+k_{4} +k_{5}\\ 0 \leq k_{1},k_{2},k_{3},k_{4},k_{5} \leq k}} \frac{2^{k_1}}{k_1!} A_{6}^{k - k_1} \\
				&\lesssim A_{6}^{k} (k + 4)!.
			\end{align*}
			For the remaining terms in \eqref{eqt2.1} we obtain
			\begin{align*}
				&\mathcal{I}_{3,2} + \mathcal{I}_{3,3} + \mathcal{I}_{3,4}\\
				&\;\lesssim\; \left\|t\,\partial_x \chi_{(t_{0},x_{0}), \vec{\epsilon}}  v_k\right\|_{H^{-1/2}_{xt}(\mathbb{R}^{2})} 
				+\left\|t\,\partial_{x}^2 \chi_{(t_{0},x_{0}), \vec{\epsilon}}  v_k\right\|_{H^{-3/2}_{xt}(\mathbb{R}^{2})} \\
				&\quad 
				+ \left\| t\,\left(\partial_x^{3} \chi_{(t_{0},x_{0}), \vec{\epsilon}}\right)\,v_k \right\|_{H^{-5/2}_{xt}(\mathbb{R}^{2})}  \\[4pt]
				&\;\lesssim\; \left(\left\|\partial_x \chi_{(t_{0},x_{0}), \vec{\epsilon}} \right\|_{L^{\infty}_{xt}(\mathbb{R}^2)}
				+\left\|\partial_{x}^2 \chi_{(t_{0},x_{0}), \vec{\epsilon}} \right\|_{L^{\infty}_{xt}(\mathbb{R}^2)}
				+ \left\| \partial_x^{3}  \chi_{(t_{0},x_{0}), \vec{\epsilon}}\right\|_{L^{\infty}_{xt}(\mathbb{R}^2)}\right)
				\left\|\chi_{(t_{0},x_{0}), \tilde{\vec{\epsilon}}}  v_k \right\|_{L^{2}_{xt}(\mathbb{R}^2)} \\[4pt]
				&\;\lesssim\; A_{7}^{k}\, k!.
			\end{align*}
		\end{claimproof}
		\begin{claim}\label{ct2.2.1}
			Let $\vec{\boldsymbol{\epsilon}} = (\epsilon, \epsilon) \in (0, \infty)^{2}$ and let $\chi_{(t_{0},x_{0}),\boldsymbol{\vec{\epsilon}}}$ be the cutoff function defined in \eqref{cutoff1}. There exist constants $c > 0,$  such that the following estimate holds for all $k \in \mathbb{N}_0$:
			\begin{equation*}
				\left\|\chi_{(t_{0},x_{0}),\boldsymbol{\vec{\epsilon}}}P^{k} v \right\|_{H^{3/2}_{xt}(\mathbb{R}^2)} \leq c A_{2}^{k} \; k! , \; \; k = 0,1,2, \ldots 
			\end{equation*}
		\end{claim}
		\begin{claimproof} 
			After  taking $\mu =3/2$ in Lemma \ref{LemaA} we get  
			\begin{equation*}
				\begin{split}
					\left\|\langle \nabla_{t,x}\rangle^{3/2} \,\left(\chi_{(t_{0},x_{0}),\boldsymbol{\vec{\epsilon}}} v_{k}\right)\right\|_{L^{2}_{xt}(\mathbb{R}^2)}
					&\lesssim
					\left\|\chi_{(t_{0},x_{0}),\boldsymbol{\vec{\epsilon}}} v_k\right\|_{H^{-3/2}_{xt}(\mathbb{R}^2)}
					\;+\; \left\|t \,\partial_x^3\left(\chi_{(t_{0},x_{0}),\boldsymbol{\vec{\epsilon}}} v_{k}\right)\right\|_{H^{-3/2}_{xt}(\mathbb{R}^2)} \\
					&\quad + \left\| P^{3}\left(\chi_{(t_{0},x_{0}),\boldsymbol{\vec{\epsilon}}} v_{k}\right) \right\|_{H^{-3/2}_{xt}(\mathbb{R}^2)}\\
					&=\Pi_{1}+\Pi_{2}+\Pi_{3}.
				\end{split}
			\end{equation*}
			In the case of $\Pi_{1}$ it follows that  for all $k\in\mathbb{N}_{0}:$
			\begin{equation*}
				\Pi_{1}
				\lesssim
				\left\|\chi_{(t_{0},x_{0}),\boldsymbol{\vec{\epsilon}}} v_{k}\right\|_{L^{2}_{xt}(\mathbb{R}^2)} 
				\;\leq\; A_{1}^{k} k ! ,
			\end{equation*}
			where we have used {\sc claim } \ref{ct2.1}.
			
			On the other hand, applying the Leibniz rule in combination with {\sc Claim} \ref{ct2.1} yields 
			\begin{align*}
				\Pi_{3}
				&\;\lesssim\; 
				\sum_{j=0}^3 \binom{3}{j} 
				\left\| P^{j+k} v \, P^{3-j} \chi_{(t_{0},x_{0}),\boldsymbol{\vec{\epsilon}}}\right\|_{H^{-3/2}_{xt}(\mathbb{R}^2)} \\[2mm]
				&\;\lesssim\; 
				\sum_{j=0}^3 \binom{3}{j} 
				\left\|P^{3-j} \chi_{(t_{0},x_{0}),\boldsymbol{\vec{\epsilon}}}\right\|_{L^{\infty}_{xt}(\mathbb{R}^2)} 
				\;\left\|\chi_{(t_{0},x_{0}), \tilde{\vec{\epsilon}}} P^{k+j} v\right\|_{L^{2}_{xt}(\mathbb{R}^2)} \\[1mm]
				&\;\lesssim\; 
				\sum_{j=0}^3 \binom{3}{j} 
				\left\|P^{3-j}\chi_{(t_{0},x_{0}),\boldsymbol{\vec{\epsilon}}} \right\|_{L^{\infty}_{xt}(\mathbb{R}^2)} 
				\; A_{1}^{k +j} (k + j)! \\[1mm]
				&\;\lesssim\; A_{1}^{k +3} (k + 3)!,
			\end{align*}
			where we employed the cut-off function $\chi_{(t_{0},x_{0}), \tilde{\vec{\epsilon}}}$ defined in \eqref{cutoff1}, specialized to the case $\tilde{\vec{\epsilon}} = (2\epsilon, 2\epsilon)$. The penultimate estimate then follows by applying {\sc Claim} \ref{ct2.1} to the product $\chi_{(t_{0},x_{0}), \tilde{\vec{\epsilon}}} v_k.$

			On the other hand, the identity 
			\begin{equation}\label{IdentityA}
				\begin{split}
					\partial_{x}^{3}\left(\chi_{(t_{0},x_{0}),\boldsymbol{\vec{\epsilon}}} f\right) &= \chi_{(t_{0},x_{0}),\boldsymbol{\vec{\epsilon}}}\, \partial_{x}^{3} f + 3 \left(\partial_{x} \chi_{(t_{0},x_{0}),\boldsymbol{\vec{\epsilon}}}\right) \partial_{x}^{2} f + 3 \left(\partial_{x}^{2} \chi_{(t_{0},x_{0}),\boldsymbol{\vec{\epsilon}}}\right) \partial_{x} f\\
					&\quad  + \left(\partial_{x}^{3} \chi_{(t_{0},x_{0}),\boldsymbol{\vec{\epsilon}}}\right) f,
				\end{split}
			\end{equation}
			allows us to obtain that 
			\begin{equation*}
				\begin{split}
					\Pi_{2} 
					&\lesssim \left\| \chi_{(t_{0},x_{0}),\boldsymbol{\vec{\epsilon}}} t \partial_{x}^{3}v_k\right\|_{H^{-3/2}_{xt}(\mathbb{R}^2)} 
					+ \left\| \partial_{x}^{2}\left(t \partial_{x}\chi_{(t_{0},x_{0}),\boldsymbol{\vec{\epsilon}}}  v_k\right) \right\|_{H^{-3/2}_{xt}(\mathbb{R}^2)}
					+ \left\| \partial_{x} \left(t \partial_{x}^{2}  \chi_{(t_{0},x_{0}),\boldsymbol{\vec{\epsilon}}}  v_k\right) \right\|_{H^{-3/2}_{xt}(\mathbb{R}^2)}\\
					&\quad 
					+ \left\| t \left(\partial_{x}^{3}\chi_{(t_{0},x_{0}),\boldsymbol{\vec{\epsilon}}}  \right) v_k \right\|_{H^{-3/2}_{xt}(\mathbb{R}^2)}\\
					&=\Pi_{2,1}+\Pi_{2,2}+\Pi_{2,3}+\Pi_{2,4}.
				\end{split}
			\end{equation*}
			Notice that by {\sc claims } \ref{ct2.1} and \ref{ct2.2}  combined with Theorem \ref{kp1}, we obtain 
			\begin{equation*}
				\begin{split}
					\Pi_{2,2} &\sim \left\| t \partial_{x}\chi_{(t_{0},x_{0}),\boldsymbol{\vec{\epsilon}}}  v_k \right\|_{H^{1/2}_{xt}(\mathbb{R}^2)} \\
					&= \left\| t \partial_x \chi_{(t_{0},x_{0}),\boldsymbol{\vec{\epsilon}}} \left(\chi_{(t_{0},x_{0}), \tilde{\vec{\epsilon}}} v_k\right)\right\|_{H^{1/2}_{xt}(\mathbb{R}^2)} \\
					& \lesssim\left\|t \partial_x \chi_{(t_{0},x_{0}),\boldsymbol{\vec{\epsilon}}} \right\|_{L^{\infty}_{xt}(\mathbb{R}^{2})} \left\| \chi_{(t_{0},x_{0}), \tilde{\vec{\epsilon}}} v_k\right\|_{H^{1/2}_{xt}(\mathbb{R}^2)}+\left\|t \partial_x \chi_{(t_{0},x_{0}),\boldsymbol{\vec{\epsilon}}} \right\|_{H^{1/2}_{xt}(\mathbb{R}^{2})}\left\| \chi_{(t_{0},x_{0}), \tilde{\vec{\epsilon}}} v_k\right\|_{L^{2}_{xt}(\mathbb{R}^2)} \\
					&\lesssim A_{1}^{k} k!.
				\end{split}
			\end{equation*}
			Additionally, the terms $\Pi_{2,3}$ and $\Pi_{2,4}$ satisfy:
			\begin{equation*}
				\begin{split}
					\Pi_{2,3}&\sim  \left\| t \, \partial_x^2 \chi_{(t_{0},x_{0}),\boldsymbol{\vec{\epsilon}}}\, v_{k}\right\|_{H^{-1/2}_{xt}(\mathbb{R}^2)}\\
					&\lesssim \left\|\chi_{(t_{0},x_{0}), \tilde{\vec{\epsilon}}}  v_{k} \right\|_{L^{2}_{xt}(\mathbb{R}^2)} \, \left\| t \, \partial_x^2  \chi_{(t_{0},x_{0}),\boldsymbol{\vec{\epsilon}}} \right\|_{L^{\infty}_{xt}(\mathbb{R}^{2})}\\
					&\lesssim A_1^k \, k!,
				\end{split}
			\end{equation*}
			and 
			\begin{equation*}
				\begin{split}
					\Pi_{2,4} &=  \left\| t \partial_{x}^{3} \chi_{(t_{0},x_{0}),\boldsymbol{\vec{\epsilon}}} \, v_k \right\|_{H^{-3/2}_{xt}(\mathbb{R}^2)}\\
					& \lesssim \left\| t \partial_{x}^{3} \chi_{(t_{0},x_{0}),\boldsymbol{\vec{\epsilon}}} \right\|_{L^{\infty}_{xt}(\mathbb{R}^{2})} \left\|\chi_{(t_{0},x_{0}), \tilde{\vec{\epsilon}}} v_k\right\|_{L^{2}_{xt}(\mathbb{R}^2)} \\
					& \lesssim A_1^k \, k!
				\end{split}
			\end{equation*}
			To estimate the term $\Pi_{2,1}$, we notice that 
			\begin{equation}\label{IdentityI}
				\chi_{(t_{0},x_{0}),\boldsymbol{\vec{\epsilon}}} t  \partial_x^3 v_{k} = -\tfrac{\chi_{(t_{0},x_{0}),\boldsymbol{\vec{\epsilon}}}}{3}  \, P v_{k} + \tfrac{\chi_{(t_{0},x_{0}),\boldsymbol{\vec{\epsilon}}}}{3}  x  \partial_x v_k- t \chi_{(t_{0},x_{0}),\boldsymbol{\vec{\epsilon}}} \, \Pi_{k}^{(4)}(v).   
			\end{equation}
			Therefore
			\begin{equation*}
				\begin{split}
					\Pi_{2,1}& = \left\| t \chi_{(t_{0},x_{0}),\boldsymbol{\vec{\epsilon}}} \, \partial_x^3 v_{k} \right\|_{H^{-3/2}(\mathbb{R}^2)}\\
					&\lesssim \| \chi_{(t_{0},x_{0}),\boldsymbol{\vec{\epsilon}}} \,  \, v_{k +1}\|_{H^{-3/2}_{xt}(\mathbb{R}^{2})}
					+ \left\| \chi_{(t_{0},x_{0}),\boldsymbol{\vec{\epsilon}}} x \, \partial_x v_k \right\|_{H^{-3/2}_{xt}(\mathbb{R}^{2})}
					+ \left\| t \chi_{(t_{0},x_{0}),\boldsymbol{\vec{\epsilon}}} \, \Pi_{k}^{(4)}(v)\right\|_{H^{-3/2}_{xt}(\mathbb{R}^{2})}\\
					&=\Lambda_{1}+\Lambda_{2}+\Lambda_{3}.
				\end{split}
			\end{equation*}
			In the first place, by {\sc claim } \ref{ct2.1}
			\begin{equation*}
				\begin{split}
					\Lambda_{1} &\lesssim \left\| \chi_{(t_{0},x_{0}),\boldsymbol{\vec{\epsilon}}} \, v_{k +1}\right\|_{H^{-3/2}_{xt}(\mathbb{R}^{2})} \\
					&\lesssim \left\| \chi_{(t_{0},x_{0}),\boldsymbol{\vec{\epsilon}}}\, v_{k +1}\right\|_{L^{2}_{xt}(\mathbb{R}^2)} \\
					&\lesssim A_1^{k} (k +1)!.
				\end{split}
			\end{equation*}
			In the second place, combining H\"older's inequality and {\sc claim } \ref{ct2.1}
			\begin{equation*}
				\begin{split}
					\Lambda_{2}&= \left\| \chi_{(t_{0},x_{0}),\boldsymbol{\vec{\epsilon}}} \,x \, \partial_x v_k \right\|_{H^{-3/2}_{xt}(\mathbb{R}^{2})} \\
					&\lesssim \left\| \partial_{x}\left(\chi_{(t_{0},x_{0}),\boldsymbol{\vec{\epsilon}}} x v_k\right)\right\|_{H^{-3/2}_{xt}(\mathbb{R}^{2})}
					+ \left\| \partial_{x}\left(\chi_{(t_{0},x_{0}),\boldsymbol{\vec{\epsilon}}} x\right) \chi_{(t_{0},x_{0}), \tilde{\vec{\epsilon}}}  v_k \right\|_{H^{-3/2}_{xt}(\mathbb{R}^{2})} \\
					&\lesssim \left\| \chi_{(t_{0},x_{0}),\boldsymbol{\vec{\epsilon}}} x v_k\right\|_{H^{-1/2}_{xt}(\mathbb{R}^{2})}
					+ \left\| \partial_{x}\left(\chi_{(t_{0},x_{0}),\boldsymbol{\vec{\epsilon}}} x\right)\right\|_{L^{\infty}_{xt}(\mathbb{R}^{2})}
					\left\|\chi_{(t_{0},x_{0}), \tilde{\vec{\epsilon}}}  v_{k}\right\|_{L^{2}_{xt}(\mathbb{R}^{2})} \\
					&\lesssim \left\|\chi_{(t_{0},x_{0}),\boldsymbol{\vec{\epsilon}}} x\right\|_{L^{\infty}_{xt}(\mathbb{R}^{2})} 
					\left\|\chi_{(t_{0},x_{0}), \tilde{\vec{\epsilon}}}  v_k\right\|_{L^{2}_{xt}(\mathbb{R}^{2})}
					+ \left\| \partial_{x}\left(\chi_{(t_{0},x_{0}),\boldsymbol{\vec{\epsilon}}} x\right)\right\|_{L^{\infty}_{xt}(\mathbb{R}^{2})}
					\left\|\chi_{(t_{0},x_{0}), \tilde{\vec{\epsilon}}}  v_{k}\right\|_{L^{2}_{xt}(\mathbb{R}^{2})} \\
					&\lesssim A_1^{k} k!.
				\end{split}
			\end{equation*}
			In the third place,
			\begin{equation*}
				\begin{split}
					\Lambda_{3}&\lesssim  \sum_{\substack{k=k_{1}+k_{2}+k_{3}+k_{4} +k_{5}\\ 0 \leq k_{1},k_{2},k_{3},k_{4},k_{5} \leq k}}\frac{2^{k_1} k!}{k_1! k_2! k_3! k_4! k_5!} 
					\; \left\| t \chi_{(t_{0},x_{0}),\boldsymbol{\vec{\epsilon}}} \, \partial_x (v_{k_2} v_{k_3} v_{k_4} v_{k_5}) \right\|_{H^{-3/2}_{xt}(\mathbb{R}^{2})}.
				\end{split}
			\end{equation*}
			Observe that by Lemma \ref{LemaC}
			\begin{equation*}
				\begin{split}
					&\left\| t \chi_{(t_{0},x_{0}),\boldsymbol{\vec{\epsilon}}} \, \partial_x \bigl(v_{k_2} v_{k_3} v_{k_4} v_{k_5}\bigr) \right\|_{H^{-3/2}_{xt}(\mathbb{R}^{2})}\\
					&\;\lesssim\; 
					\left\| \partial_x \left(t \chi_{(t_{0},x_{0}),\boldsymbol{\vec{\epsilon}}} \, v_{k_2} v_{k_3} v_{k_4} v_{k_5}\right) \right\|_{H^{-3/2}_{xt}(\mathbb{R}^{2})}
					+ \left\| \partial_{x}\left(t  \chi_{(t_{0},x_{0}),\boldsymbol{\vec{\epsilon}}}\right) \, v_{k_2} v_{k_3} v_{k_4} v_{k_5} \right\|_{H^{-3/2}_{xt}(\mathbb{R}^{2})}\\
					&\lesssim 
					\left\| t \chi_{(t_{0},x_{0}),\boldsymbol{\vec{\epsilon}}} \, \prod_{\ell =2}^{5}  v_{k_{\ell}} \right\|_{H^{-1/2}_{xt}(\mathbb{R}^{2})}
					+ \left\| \partial_x \left(t \chi_{(t_{0},x_{0}),\boldsymbol{\vec{\epsilon}}}\right) \,  \prod_{\ell =2}^{5}  v_{k_{\ell}}\right\|_{L^{1}_{xt}(\mathbb{R}^{2})}\\
					&\lesssim  
					\left\| t \chi_{(t_{0},x_{0}),\boldsymbol{\vec{\epsilon}}} \,  \prod_{\ell =2}^{5}  v_{k_{\ell}} \right\|_{L^{2}_{xt}(\mathbb{R}^{2})}
					+ \left\| \partial_x \left(t \chi_{(t_{0},x_{0}),\boldsymbol{\vec{\epsilon}}}\right) \right\|_{	L^{\infty}_{xt}(\mathbb{R}^{2})}  
					\left\|  \prod_{\ell =2}^{5}  v_{k_{\ell}} \right\|_{L^{1}_{xt}(\mathbb{R}^{2})}\\
					&\lesssim 
					\prod_{\ell =2}^{5} \bigl\|  v_{k_{\ell}} \bigr\|_{L^{8}_{xt}(\mathbb{R}^{2})}
					+  \prod_{\ell =2}^{5} \bigl\|  v_{k_{\ell}} \bigr\|_{L^{4}_{xt}(\mathbb{R}^{2})}.
				\end{split}
			\end{equation*}
			Thus,
			\begin{equation*}
				\begin{split}
					\Lambda_{3}& \lesssim k!   \sum_{\substack{k=k_{1}+k_{2}+k_{3}+k_{4} +k_{5}\\ 0 \leq k_{1},k_{2},k_{3},k_{4},k_{5} \leq k}}   \frac{2^{k_1}}{k_1! } A_{7}^{k- k_1} \\
					&\lesssim A_{7}^{k} (k +4)!.
				\end{split}
			\end{equation*}
		\end{claimproof}
		\begin{claim}\label{ct3.2}
			Let $\vec{\boldsymbol{\epsilon}} = (\epsilon, \epsilon) \in (0, \infty)^{2}$ and let $\chi_{(t_{0},x_{0}),\boldsymbol{\vec{\epsilon}}}$ be the cutoff function defined in \eqref{cutoff1}. There exist constants $c > 0,$  such that the following estimate holds for all $k \in \mathbb{N}_0$:
			\begin{equation*}
				\left\|\chi_{(t_{0},x_{0}),\boldsymbol{\vec{\epsilon}}}P^{k} v \right\|_{H^{2}_{xt}(\mathbb{R}^2)} \leq c A_{2}^{k} \; k! , \; \; k = 0,1,2, \ldots 
			\end{equation*}
		\end{claim}
		\begin{claimproof} 
			After taking $\mu = 2 $ in Lemma \ref{LemaA} we get  
			\begin{equation*}
				\begin{split}
					\left\|\langle \nabla_{t,x}\rangle^{2} \,\left( \chi_{(t_{0},x_{0}),\boldsymbol{\vec{\epsilon}}}v_{k}\right)\right\|_{L^{2}_{xt}(\mathbb{R}^2)}
					&\lesssim\;
					\left\|\chi_{(t_{0},x_{0}),\boldsymbol{\vec{\epsilon}}} v_k\right\|_{H^{-1}_{xt}(\mathbb{R}^2)}
					+ \left\|t \,\partial_x^3\left(\chi_{(t_{0},x_{0}),\boldsymbol{\vec{\epsilon}}} v_{k}\right)\right\|_{H^{-1}_{xt}(\mathbb{R}^2)}\\
					&\quad + \left\| P^{3} \left(\chi_{(t_{0},x_{0}),\boldsymbol{\vec{\epsilon}}} v_{k}\right) \right\|_{H^{-1}_{xt}(\mathbb{R}^2)}\\
					&=\Pi_{1}+\Pi_{2}+\Pi_{3}.
				\end{split}
			\end{equation*}
			In the case of $\Pi_{1}$ and $\Pi_{2}$, {\sc claim } \ref{ct2.1} implies that  for all $k\in \mathbb{N}_{0}:$
			\begin{equation*}
				\Pi_{1}
				\;\lesssim\; 
				\left\|\chi_{(t_{0},x_{0}),\boldsymbol{\vec{\epsilon}}}v_{k}\right\|_{L^{2}_{xt}(\mathbb{R}^2)}
				\;\leq\; A_{1}^{k} k ! ,
			\end{equation*}
			and 
			\begin{equation*} 
				\begin{split}
					\left\|P^3\left(\chi_{(t_{0},x_{0}),\boldsymbol{\vec{\epsilon}}} v_{k}\right)\right\|_{H^{-
							3/2}_{xt}(\mathbb{R}^2)} 
					&\lesssim
					\sum_{j=0}^3 \binom{3}{j} 
					\left\| P^{j+k} v \, P^{3-j} \chi_{(t_{0},x_{0}),\boldsymbol{\vec{\epsilon}}}\right\|_{H^{-3/2}_{xt}(\mathbb{R}^2)} \\
					&\;\lesssim\; 
					\sum_{j=0}^3 \binom{3}{j} 
					\left\|P^{3-j} \chi_{(t_{0},x_{0}),\boldsymbol{\vec{\epsilon}}}\right\|_{L^{\infty}_{xt}(\mathbb{R}^2)} 
					\left\|\chi_{(t_{0},x_{0}), \tilde{\vec{\epsilon}}}  P^{k+j} v\right\|_{L^{2}_{xt}(\mathbb{R}^2)} \\
					&\lesssim
					\sum_{j=0}^3 \binom{3}{j} 
					\left\|P^{3-j} \chi_{(t_{0},x_{0}),\boldsymbol{\vec{\epsilon}}}\right\|_{L^{\infty}_{xt}(\mathbb{R}^2)}
					A_{1}^{k +j} (k + j)! \\
					&\lesssim A_{1}^{k +3} (k + 3)!.
				\end{split}
			\end{equation*}
			Using the  identity \eqref{IdentityA} we get
			\begin{equation*}
				\begin{split}
					\left\| t \partial_{x}^{3}\left(\chi_{(t_{0},x_{0}),\boldsymbol{\vec{\epsilon}}} v_k\right) \right\|_{H^{-
							1}_{xt}(\mathbb{R}^2)} 
					&\lesssim \left\| \chi_{(t_{0},x_{0}),\boldsymbol{\vec{\epsilon}}}t \partial_{x}^{3}v_k \right\|_{H^{-
							1}_{xt}(\mathbb{R}^2)}
					+ \left\| \partial_{x}^{2}\left(t \partial_{x}\chi_{(t_{0},x_{0}),\boldsymbol{\vec{\epsilon}}}  v_k\right) \right\|_{H^{-
							1}_{xt}(\mathbb{R}^2)}\\
					&\quad
					+ \left\| \partial_{x} \left(t \partial_{x}^{2} \chi_{(t_{0},x_{0}),\boldsymbol{\vec{\epsilon}}}  v_k\right) \right\|_{H^{-
							1}_{xt}(\mathbb{R}^2)}
					+ \left\| t \left(\partial_{x}^{3} \chi_{(t_{0},x_{0}),\boldsymbol{\vec{\epsilon}}}\right) v_k \right\|_{H^{-
							1}_{xt}(\mathbb{R}^2)}\\
					&\lesssim \left\| \chi_{(t_{0},x_{0}),\boldsymbol{\vec{\epsilon}}}t \partial_{x}^{3}v_k \right\|_{H^{-
							1}_{xt}(\mathbb{R}^2)} 
					+ \left\| t \partial_{x}\chi_{(t_{0},x_{0}),\boldsymbol{\vec{\epsilon}}}  v_k \right\|_{H^{
							1}_{xt}(\mathbb{R}^2)} \\
					&	+ \left\| t \partial_{x}^{2} \chi_{(t_{0},x_{0}),\boldsymbol{\vec{\epsilon}}}  v_k \right\|_{L^{
							2}_{xt}(\mathbb{R}^2)}  
					+ \left\| t \left(\partial_{x}^{3} \chi_{(t_{0},x_{0}),\boldsymbol{\vec{\epsilon}}}\right) v_k \right\|_{H^{-
							1}_{xt}(\mathbb{R}^2)}\\
					&= \Lambda_{1}+\Lambda_{2}+\Lambda_{3}+\Lambda_{4}.
				\end{split}
			\end{equation*}
			The terms $\Lambda_{2}, \Lambda_{3}$, and $\Lambda_{4}$ are straightforward to handle. 
			In this sense, we use $\chi_{(t_{0},x_{0}), \tilde{\vec{\epsilon}}}$ to handle them, 
			combining it with Theorem \ref{kp1} and {\sc claims} \ref{ct2.1}-\ref{ct2.2} in order to obtain, 
			that for all $k\in \mathbb{N}_{0}:$
			\begin{equation*}
				\begin{split}
					\Lambda_{2}
					&\lesssim \left\| t \partial_x \chi_{(t_{0},x_{0}),\boldsymbol{\vec{\epsilon}}} \right\|_{H^{3/2}_{xt}(\mathbb{R}^2)} \left\| \chi_{(t_{0},x_{0}), \tilde{\vec{\epsilon}}} v_k \right\|_{H^{3/2}_{xt}(\mathbb{R}^2)} \\
					&\lesssim A_{1}^{k} k!,
				\end{split}
			\end{equation*}
			and
			\begin{equation*}
				\begin{split}
					\Lambda_{3} 
					&\lesssim \left\| \chi_{(t_{0},x_{0}), \tilde{\vec{\epsilon}}} v_{k} \right\|_{L^{2}_{xt}(\mathbb{R}^2)} \left\| t \partial_x^2 \chi_{(t_{0},x_{0}),\boldsymbol{\vec{\epsilon}}}\right\|_{L^{\infty}_{xt}(\mathbb{R}^2)} \\
					&\lesssim A_1^k k!,
				\end{split}
			\end{equation*}
			and 
			\begin{equation*}
				\begin{split}
					\Lambda_{4} 
					&\lesssim \left\| t \partial_{x}^{3} \chi_{(t_{0},x_{0}),\boldsymbol{\vec{\epsilon}}}\right\|_{L^{\infty}_{xt}(\mathbb{R}^2)} \left\| \chi_{(t_{0},x_{0}), \tilde{\vec{\epsilon}}}v_k\right\|_{L^{2}_{xt}(\mathbb{R}^2)} \\
					&\lesssim A_1^k k!.
				\end{split}
			\end{equation*}
			The term $\Lambda_{1}$ is more involved, and it requires a more delicate analysis.  In this sense, we apply \eqref{IdentityI} to  obtain
			\begin{equation*}
				\begin{split}
					\Lambda_{1}&
					\lesssim \left\| \chi_{(t_{0},x_{0}),\boldsymbol{\vec{\epsilon}}} \, v_{k +1}\right\|_{H^{
							-1}_{xt}(\mathbb{R}^2)} 
					+ \left\| \chi_{(t_{0},x_{0}),\boldsymbol{\vec{\epsilon}}} x \, \partial_x v_k \right\|_{H^{
							-1}_{xt}(\mathbb{R}^2)} 
					+ \left\| t \chi_{(t_{0},x_{0}),\boldsymbol{\vec{\epsilon}}} \, \Pi_{k}^{(4)}(v)\right\|_{H^{
							-1}_{xt}(\mathbb{R}^2)} \\
					&=\Lambda_{1,1}+\Lambda_{1,2}+\Lambda_{1,3}.
				\end{split}
			\end{equation*}
			Now,  after combining {\sc claims } \ref{ct2.1}  we obtain the following bound
			\begin{equation*}
				\begin{split}
					\Lambda_{1,1}&\lesssim \left\| \chi_{(t_{0},x_{0}),\boldsymbol{\vec{\epsilon}}}\, v_{k +1}\right\|_{H^{
							-1}_{xt}(\mathbb{R}^2)} \\
					&\lesssim \left\| \chi_{(t_{0},x_{0}),\boldsymbol{\vec{\epsilon}}} \, v_{k +1}\right\|_{L^{
							2}_{xt}(\mathbb{R}^2)} \\
					& \lesssim A_1^{k} (k +1)!,\quad \forall \,k\in \mathbb{N}_{0}.
				\end{split}
			\end{equation*}
			Also, by {\sc claim } \ref{ct2.1} 
			\begin{equation*}
				\begin{split}
					\Lambda_{1,2}&= \left\| \chi_{(t_{0},x_{0}),\boldsymbol{\vec{\epsilon}}} x \, \partial_x v_k \right\|_{H^{
							-1}_{xt}(\mathbb{R}^2)} \\
					&\lesssim \left\| \partial_{x}\left( \chi_{(t_{0},x_{0}),\boldsymbol{\vec{\epsilon}}} x v_k\right)\right\|_{H^{
							-1}_{xt}(\mathbb{R}^2)}
					+ \left\| \partial_{x}\left( \chi_{(t_{0},x_{0}),\boldsymbol{\vec{\epsilon}}} x\right) \chi_{(t_{0},x_{0}), \tilde{\vec{\epsilon}}} v_k \right\|_{H^{
							-1}_{xt}(\mathbb{R}^2)} \\
					&\lesssim \left\| \chi_{(t_{0},x_{0}),\boldsymbol{\vec{\epsilon}}} x v_k\right\|_{L^{
							2}_{xt}(\mathbb{R}^2)} 
					+ \left\| \partial_{x}\left( \chi_{(t_{0},x_{0}),\boldsymbol{\vec{\epsilon}}} x\right)\right\|_{L^{
							\infty}_{xt}(\mathbb{R}^2)} 
					\left\|\chi_{(t_{0},x_{0}), \tilde{\vec{\epsilon}}} v_{k}\right\|_{L^{
							2}_{xt}(\mathbb{R}^2)} \\
					&\lesssim \left\| \chi_{(t_{0},x_{0}),\boldsymbol{\vec{\epsilon}}} x\right\|_{L^{
							\infty}_{xt}(\mathbb{R}^2)}
					\left\|\chi_{(t_{0},x_{0}), \tilde{\vec{\epsilon}}} v_k\right\|_{L^{
							2}_{xt}(\mathbb{R}^2)}
					+ \left\| \partial_{x}\left( \chi_{(t_{0},x_{0}),\boldsymbol{\vec{\epsilon}}} x\right)\right\|_{L^{
							\infty}_{xt}(\mathbb{R}^2)} 
					\left\|\chi_{(t_{0},x_{0}), \tilde{\vec{\epsilon}}} v_{k}\right\|_{L^{
							2}_{xt}(\mathbb{R}^2)}\\
					&\lesssim A_1^{k} k!\quad \forall \,k\in \mathbb{N}_{0}.
				\end{split}
			\end{equation*}
			Additionally, 
			\begin{equation*}
				\Lambda_{1,3} \lesssim  \sum_{\substack{k=k_{1}+k_{2}+k_{3}+k_{4} +k_{5}\\ 0 \leq k_{1},k_{2},k_{3},k_{4},k_{5} \leq k}}   \frac{2^{k_1} k!}{k_1! k_2! k_3! k_4! k_5!} 
				\; \left\| t \chi_{(t_{0},x_{0}),\boldsymbol{\vec{\epsilon}}} \, \partial_x (v_{k_2} v_{k_3} v_{k_4} v_{k_5}) \right\|_{H^{
						-1}_{xt}(\mathbb{R}^2)},
			\end{equation*}
			from where it is straightforward to obtain 
			\begin{equation*}
				\begin{split}
					&\left\| t \chi_{(t_{0},x_{0}),\boldsymbol{\vec{\epsilon}}} \, \partial_x \bigl(v_{k_2} v_{k_3} v_{k_4} v_{k_5}\bigr) \right\|_{H^{
							-1}_{xt}(\mathbb{R}^2)}\\
					&\;\lesssim\; 
					\left\| \partial_x \bigl(t \chi_{(t_{0},x_{0}),\boldsymbol{\vec{\epsilon}}} \, v_{k_2} v_{k_3} v_{k_4} v_{k_5}\bigr) \right\|_{H^{
							-1}_{xt}(\mathbb{R}^2)}
					+ \left\| \partial_{x}\left(t \chi_{(t_{0},x_{0}),\boldsymbol{\vec{\epsilon}}}\right) \, v_{k_2} v_{k_3} v_{k_4} v_{k_5} \right\|_{H^{
							-1}_{xt}(\mathbb{R}^2)}\\
					&\lesssim 
					\left\| \left(t\chi_{(t_{0},x_{0}),\boldsymbol{\vec{\epsilon}}}\right) \, \prod_{\ell =2}^{5} v_{k_{\ell}} \right\|_{L^{
							2}_{xt}(\mathbb{R}^2)}
					+ \left\| \partial_x \left(t \chi_{(t_{0},x_{0}),\boldsymbol{\vec{\epsilon}}} \right) \,  \prod_{\ell =2}^{5}  v_{k_{\ell}} \right\|_{L^{
							2}_{xt}(\mathbb{R}^2)} \\
					&\lesssim  
					\Big(\left\|t\chi_{(t_{0},x_{0}),\boldsymbol{\vec{\epsilon}}}\right\|_{L^{
							\infty}_{xt}(\mathbb{R}^2)}
					+ \left\| \partial_x \left(t \chi_{(t_{0},x_{0}),\boldsymbol{\vec{\epsilon}}}\right) \right\|_{L^{
							\infty}_{xt}(\mathbb{R}^2)} \Big) \left\| \,  \prod_{\ell =2}^{5} v_{k_{\ell}} \right\|_{L^{
							2}_{xt}(\mathbb{R}^2)} \\
					&\lesssim \left(\left\|t\chi_{(t_{0},x_{0}),\boldsymbol{\vec{\epsilon}}}\right\|_{L^{
							\infty}_{xt}(\mathbb{R}^2)}  + \left\| \partial_x \left(t \chi_{(t_{0},x_{0}),\boldsymbol{\vec{\epsilon}}}\right) \right\|_{L^{
							\infty}_{xt}(\mathbb{R}^2)}  \right)  \prod_{\ell =2}^{5} \bigl\|  v_{k_{\ell}} \bigr\|_{L^{
							8}_{xt}(\mathbb{R}^2)},
				\end{split}
			\end{equation*}
			where we have used 
			Lemma \ref{LemaC} .
			
			Hence
			\begin{equation*}
				\begin{split}
					\Lambda_{1,3} &\lesssim k!  \sum_{\substack{k=k_{1}+k_{2}+k_{3}+k_{4} +k_{5}\\ 0 \leq k_{1},k_{2},k_{3},k_{4},k_{5} \leq k}}   \frac{2^{k_1}}{k_1! } A_{8}^{k- k_1}\\
					& \lesssim A_{8}^{k} (k +4)!,\quad \forall k\in \mathbb{N}_{0}.
				\end{split}
			\end{equation*}
		\end{claimproof}
		\begin{claim}\label{ct3.3}
			Let $\vec{\boldsymbol{\epsilon}} = (\epsilon, \epsilon) \in (0, \infty)^{2}$ and let $\chi_{(t_{0},x_{0}),\boldsymbol{\vec{\epsilon}}}$ be the cutoff function defined in \eqref{cutoff1}. There exist constants $c > 0,$  such that the following estimate holds for all $k \in \mathbb{N}_0$:
			\begin{equation*}
				\left\|\chi_{(t_{0},x_{0}),\boldsymbol{\vec{\epsilon}}}P^{k} v \right\|_{H^{7/2}_{xt}(\mathbb{R}^2)} \leq c A_{2}^{k} \; k! , \; \; k = 0,1,2, \ldots 
			\end{equation*}
		\end{claim}
		\begin{claimproof} 
			By Lemma \ref{LemaA}, it is clear that 
			\begin{equation*}
				\begin{split}
					\left\| \langle \nabla_{t,x} \rangle^{7/2} \left(\chi_{(t_{0},x_{0}),\boldsymbol{\vec{\epsilon}}} v_k\right) \right\|_{L^{
							2}_{xt}(\mathbb{R}^2)}
					&\lesssim
					\left\|\chi_{(t_{0},x_{0}),\boldsymbol{\vec{\epsilon}}} v_k\right\|_{H^{1/2
						}_{xt}(\mathbb{R}^2)}
					+ \left\|t \,\partial_{x}^{3} \left(\chi_{(t_{0},x_{0}),\boldsymbol{\vec{\epsilon}}} v_k\right)\right\|_{H^{1/2
						}_{xt}(\mathbb{R}^2)}\\
					&\quad 
					+ \left\|P^3\left(\chi_{(t_{0},x_{0}),\boldsymbol{\vec{\epsilon}}} v_k\right)\right\|_{H^{1/2
						}_{xt}(\mathbb{R}^2)}.\\
					&= 	\overset{\sim}{\Xi_{1}} +	\overset{\sim}{\Xi_{2}} +	\overset{\sim}{\Xi_{3}} .
				\end{split}
			\end{equation*}
			Notice that by {\sc claim } \ref{ct2.2} we have 
			\begin{equation*}
				\overset{\sim}{\Xi_{1}}	\leq c A_{1}^{k} k!, \quad \forall \,k\in \mathbb{N}_{0}.
			\end{equation*}
			Additionally, by combining Claim \ref{ct3.2}, the Leibniz rule, Theorem \ref{kp1}, and interpolation, we obtain
			\begin{equation*}
				\begin{split}
					\overset{\sim}{\Xi_{3}} 
					&\lesssim \sum_{j=0}^{3} \binom{3}{j} 
					\left\| \chi_{(t_{0},x_{0}), \tilde{\vec{\epsilon}}} P^{j+k} v \, P^{3-j} \chi_{(t_{0},x_{0}),\boldsymbol{\vec{\epsilon}}} \right\|_{H^{1/2}_{xt}(\mathbb{R}^2)}\\
					&\lesssim \sum_{j=0}^{3} \binom{3}{j} 
					\left\| \chi_{(t_{0},x_{0}), \tilde{\vec{\epsilon}}} v_{j+k} \right\|_{H^{1/2}_{xt}(\mathbb{R}^2)} \, \left\| P^{3-j} \chi_{(t_{0},x_{0}),\boldsymbol{\vec{\epsilon}}} \right\|_{H^{1/2}_{xt}(\mathbb{R}^2)} \\
					&\lesssim  A_2^k \, (k+3)!, \quad \forall \,k\in \mathbb{N}_{0}.
				\end{split}
			\end{equation*}
			By the Leibniz rule and the triangular inequality 
			\begin{equation*}
				\overset{\sim}{\Xi_{2}}
				\lesssim
				\sum_{\ell=0}^3 \left\|t\, \partial_x^{3-\ell} v_k \, \partial_x^\ell \chi_{(t_{0},x_{0}),\boldsymbol{\vec{\epsilon}}}\right\|_{H^{1/2}_{xt}(\mathbb{R}^2)} .
			\end{equation*}
			In order to estimate the last expression, we consider the terms separately. If $\ell = 0$,  we get after make use of identity \eqref{IdentityI}
			\begin{equation*}
				\begin{split}
					\left\| t \chi_{(t_{0},x_{0}),\boldsymbol{\vec{\epsilon}}}  \partial_{x}^{3} v_k \right\|_{H^{1/2}_{xt}(\mathbb{R}^2)} &\lesssim \left\| \chi_{(t_{0},x_{0}),\boldsymbol{\vec{\epsilon}}} \; P v_k\right\|_{H^{1/2}_{xt}(\mathbb{R}^2)} 
					+ \left\|\chi_{(t_{0},x_{0}),\boldsymbol{\vec{\epsilon}}} x \partial_x v_k\right\|_{H^{1/2}_{xt}(\mathbb{R}^2)}\\
					&\quad  + \left\| t\chi_{(t_{0},x_{0}),\boldsymbol{\vec{\epsilon}}} \Pi_{k}^{(4)}(v)\right\|_{H^{1/2}_{xt}(\mathbb{R}^2)} .
				\end{split}
			\end{equation*}
			The first two terms in the above inequality are easy to estimate. Indeed, we have by {\sc claim } \ref{ct2.2}
			\begin{equation*}
				\begin{split}
					\left\| \chi_{(t_{0},x_{0}),\boldsymbol{\vec{\epsilon}}} \,P v_k\right\|_{H^{1/2}_{xt}(\mathbb{R}^2)}  &= \| \chi_{(t_{0},x_{0}),\boldsymbol{\vec{\epsilon}}} \; v_{k +1}\|_{H^{1/2}_{xt}(\mathbb{R}^2)} \\
					&\leq c A_{1}^{k} (k+1)!, \; \; k = 0, 1,2,\ldots
				\end{split}
			\end{equation*}
			and  by  combining {\sc claim } \ref{ct2.2.1} and Theorem \ref{kp1}
			\begin{equation*}
				\begin{split}
					\left\|\chi_{(t_{0},x_{0}), \tilde{\vec{\epsilon}}}\chi_{(t_{0},x_{0}),\boldsymbol{\vec{\epsilon}}} \,x \partial_x v_k\right\|_{H^{1/2}_{xt}(\mathbb{R}^2)}&  = \| \chi_{(t_{0},x_{0}),\boldsymbol{\vec{\epsilon}}} x \partial_{x}\left( v_k\right) \|_{H^{1/2}_{xt}(\mathbb{R}^2)}  \\
					&\lesssim  \left\| \chi_{(t_{0},x_{0}),\boldsymbol{\vec{\epsilon}}}x \right\|_{H^{1/2}_{xt}(\mathbb{R}^2)} \left\| \chi_{(t_{0},x_{0}), \tilde{\vec{\epsilon}}}  v_{k} \right\|_{H^{3/2}_{xt}(\mathbb{R}^2)} \\
					&\leq c A_{2}^{k} \; k! , \; \; k = 0,1,2, \ldots
				\end{split}
			\end{equation*}
			The last term can be estimated as follows 
			\begin{equation*}
				\begin{split}
					&\left\|t\, \chi_{(t_{0},x_{0}),\boldsymbol{\vec{\epsilon}}}\, \Pi_{k}^{(4)}(v)\right\|_{H^{1/2}_{xt}(\mathbb{R}^2)}\\
					&\lesssim
					\sum_{\substack{k=k_{1}+k_{2}+k_{3}+k_{4} +k_{5}\\ 0 \leq k_{1},k_{2},k_{3},k_{4},k_{5} \leq k}}  \frac{2^{k_1} k!}{k_1! k_2! k_3! k_4 ! k_5 !} 
					\left\|t\, \chi_{(t_{0},x_{0}),\boldsymbol{\vec{\epsilon}}}\, \partial_x \left( v_{k_2} v_{k_3} v_{k_4} v_{k_5}\right)\right\|_{H^{1/2}_{xt}(\mathbb{R}^2)}\\
					&\leq c\sum_{\substack{k=k_{1}+k_{2}+k_{3}+k_{4} +k_{5}\\ 0 \leq k_{1},k_{2},k_{3},k_{4},k_{5} \leq k}}  \frac{2^{k_1} k!}{k_1! k_2! k_3! k_4 ! k_5 !} 
					\left\|\partial_x \left(t\, \chi_{(t_{0},x_{0}),\boldsymbol{\vec{\epsilon}}}\,  v_{k_2} v_{k_3} v_{k_4} v_{k_5}\right)\right\|_{H^{1/2}_{xt}(\mathbb{R}^2)}\\
					&\quad +\sum_{\substack{k=k_{1}+k_{2}+k_{3}+k_{4} +k_{5}\\ 0 \leq k_{1},k_{2},k_{3},k_{4},k_{5} \leq k}}  \frac{2^{k_1} k!}{k_1! k_2! k_3! k_4 ! k_5 !} 
					\left\|\partial_x \left(t\, \chi_{(t_{0},x_{0}),\boldsymbol{\vec{\epsilon}}}\right)\,  v_{k_2} v_{k_3} v_{k_4} v_{k_5}\right\|_{H^{1/2}_{xt}(\mathbb{R}^2)}.
				\end{split}
			\end{equation*}
			Nevertheless, applying Theorem \ref{kp1}, interpolation, and Sobolev embedding allows us to bound each term in the summation as follows:
			\begin{equation*}
				\begin{split}
					\left\|\partial_x \left(t \chi_{(t_{0},x_{0}),\boldsymbol{\vec{\epsilon}}} \, v_{k_2} v_{k_3} v_{k_4} v_{k_5}\right)\right\|_{H^{1/2}_{xt}(\mathbb{R}^2)} 
					&\;\lesssim\;
					\left\| t \chi_{(t_{0},x_{0}),\boldsymbol{\vec{\epsilon}}} \, v_{k_2} v_{k_3} v_{k_4} v_{k_5}\right\|_{H^{3/2}_{xt}(\mathbb{R}^2)}  \\
					&\lesssim \left\| t \chi_{(t_{0},x_{0}),\boldsymbol{\vec{\epsilon}}} \right\|_{H^{3/2}(\mathbb{R}^2)} 
					\left\| \prod_{j=2}^{5} \chi_{(t_{0},x_{0}), \tilde{\vec{\epsilon}}}v_{k_{j}}\right\|_{H^{3/2}_{xt}(\mathbb{R}^2)}  \\
					&\lesssim \left\| t \chi_{(t_{0},x_{0}),\boldsymbol{\vec{\epsilon}}}  \right\|_{H^{3/2}_{xt}(\mathbb{R}^2)} \;
					\prod_{j=2}^{5} \left\| \chi_{(t_{0},x_{0}), \tilde{\vec{\epsilon}}}v_{k_{j}} \right\|_{H^{3/2}_{xt}(\mathbb{R}^2)} \\
					&\lesssim \prod_{j=2}^{5} A_{j}^{k_{j}} k_{j}! ,
				\end{split}
			\end{equation*}
			where we have used {\sc claim}  \ref{ct2.2.1} in the last inequality adapted to $ \chi_{(t_{0},x_{0}), \tilde{\vec{\epsilon}}}$.
			
			\emph{Mutatis mutandis}, the term $\left\| t \partial_x \chi_{(t_{0},x_{0}),\boldsymbol{\vec{\epsilon}}} \left(\prod_{j = 2}^{5} v_{k_{j}} \right)\right\|_{H^{3/2}_{xt}(\mathbb{R}^2)}$ can be estimated in an analogous manner.  Indeed,
			\begin{equation*}
				\left\| t \partial_x \chi_{(t_{0},x_{0}),\boldsymbol{\vec{\epsilon}}} \left(\prod_{j = 2}^{5} v_{k_{j}} \right)\right\|_{H^{3/2}_{xt}(\mathbb{R}^2)}\lesssim \prod_{j=2}^{5} A_{j}^{k_{j}} k_{j}!.
			\end{equation*}
			Finally, we gather the estimates above to obtain 
			\begin{align*}
				\left\|t \chi_{(t_{0},x_{0}),\boldsymbol{\vec{\epsilon}}} \,\Pi_{k}^{(4)}(v)\right\|_{H^{1/2}_{xt}(\mathbb{R}^2)}
				&\lesssim\sum_{\substack{k=k_{1}+k_{2}+k_{3}+k_{4} +k_{5}\\ 0 \leq k_{1},k_{2},k_{3},k_{4},k_{5} \leq k}}  
				\frac{2^{k_1} \, k!}{k_1!\,k_2!\,k_3!\,k_4!\,k_5!}\,
				A_6^{\,k-k_1}\, k_2! \, k_3! \, k_4! \, k_5! \\
				&= \sum_{\substack{k=k_{1}+k_{2}+k_{3}+k_{4} +k_{5}\\ 0 \leq k_{1},k_{2},k_{3},k_{4},k_{5} \leq k}}  \frac{2^{k_1} \, k!}{k_1!}\, A_6^{\,k-k_1} \\
				&= k!\, A_6^k \sum_{\substack{k=k_{1}+k_{2}+k_{3}+k_{4} +k_{5}\\ 0 \leq k_{1},k_{2},k_{3},k_{4},k_{5} \leq k}} \frac{2^{k_1}}{k_1 !\, A_6^{\,k_1}} \\
				&\lesssim\; A_7^k (k+4)!.
			\end{align*}
			Hence, for all $k\in \mathbb{N}_{0}$,
			\begin{equation}\label{KeyEstimateII}
				\left\| t \chi_{(t_{0},x_{0}),\boldsymbol{\vec{\epsilon}}} \partial_{x}^{3} v_k \right\|_{H^{1/2}_{xt}(\mathbb{R}^2)}\lesssim A_{8}^{k} k!.   
			\end{equation}
			The remaining expression is 
			\begin{equation*}
				\begin{split}
					\sum_{\ell=1}^3 \left\|t\, \partial_x^{3-\ell} v_k \,\partial_x^\ell \chi_{(t_{0},x_{0}),\boldsymbol{\vec{\epsilon}}} \right\|_{H^{1/2}_{xt}(\mathbb{R}^2)}
					&= \; \left\|t\,\partial_x^{2}v_k\,\partial_x \chi_{(t_{0},x_{0}),\boldsymbol{\vec{\epsilon}}} \right\|_{H^{1/2}_{xt}(\mathbb{R}^2)}
					+ \left\|t\,\partial_x v_k\,\partial_x^{2}\chi_{(t_{0},x_{0}),\boldsymbol{\vec{\epsilon}}} \right\|_{H^{1/2}_{xt}(\mathbb{R}^2)}\\
					&\quad 
					+ \left\|t\,v_k\,\partial_x^{3}\chi_{(t_{0},x_{0}),\boldsymbol{\vec{\epsilon}}} \right\|_{H^{1/2}_{xt}(\mathbb{R}^2)}.\\
				\end{split}
			\end{equation*}
			Each of these terms can be estimated individually. We begin with the first term, $\big\|t\,\partial_x^{2}v_k\,\partial_x \chi_{(t_{0},x_{0}),\boldsymbol{\vec{\epsilon}}} \big\|_{H^{1/2}_{xt}(\mathbb{R}^2)}$, which satisfies
			\begin{equation*}
				\begin{split}
					\left\|t\,\partial_x^{2}v_k\,\partial_x \chi_{(t_{0},x_{0}),\boldsymbol{\vec{\epsilon}}}\right\|_{H^{1/2}_{xt}(\mathbb{R}^2)} 
					&\lesssim \left\|t\,\partial_x^{2}v_k\,\partial_x \chi_{(t_{0},x_{0}),\boldsymbol{\vec{\epsilon}}}\right\|_{L^{
							2}_{xt}(\mathbb{R}^2)} + \left\| t\, \partial_{x}\left(\partial_x^{2}v_k\,\partial_x \chi_{(t_{0},x_{0}),\boldsymbol{\vec{\epsilon}}}\right)\right\|_{L^{
							2}_{xt}(\mathbb{R}^2)} .
				\end{split}
			\end{equation*}
			For the first component on the right-hand side, we have
			\begin{equation*}
				\left\|t\,\partial_x^{2}v_k\,\partial_x \chi_{(t_{0},x_{0}),\boldsymbol{\vec{\epsilon}}}\right\|_{L^{
						2}_{xt}(\mathbb{R}^2)}\lesssim \left\|t \partial_x \chi_{(t_{0},x_{0}),\boldsymbol{\vec{\epsilon}}} \right\|_{L^{
						\infty}_{xt}(\mathbb{R}^2)} \,\left\|  \, \partial_{x}^2\left(\chi_{(t_{0},x_{0}), \tilde{\vec{\epsilon}}} v_k\right)\right\|_{L^{
						2}_{xt}(\mathbb{R}^2)} \lesssim \left\| \chi_{(t_{0},x_{0}), \tilde{\vec{\epsilon}}} v_k\right\|_{H^{
						2}_{xt}(\mathbb{R}^2)},  
			\end{equation*}
			where we have used the fact that $\chi_{(t_{0},x_{0}), \tilde{\vec{\epsilon}}}\equiv 1$ on $\operatorname{supp}_{t,x}\left(t\,\partial_x \chi_{(t_{0},x_{0}),\boldsymbol{\vec{\epsilon}}}\right)$.
			
			Consequently, Claim~\ref{ct3.2} combined with the bound \eqref{KeyEstimateII} adapted to $\chi_{(t_{0},x_{0}), \tilde{\vec{\epsilon}}}$ yields
			\begin{equation*}
				\begin{split}
					\left\| t\, \partial_{x}\left(\partial_x^{2}v_k\,\partial_x \chi_{(t_{0},x_{0}),\boldsymbol{\vec{\epsilon}}}\right)\right\|_{L^{
							2}_{xt}(\mathbb{R}^2)}& \lesssim \left\| \left(t \partial_x \chi_{(t_{0},x_{0}),\boldsymbol{\vec{\epsilon}}}\right) \partial_{x}^{3} v_k \right\|_{L^{
							2}_{xt}(\mathbb{R}^2)} + \left\| t \partial_{x}^2 \chi_{(t_{0},x_{0}),\boldsymbol{\vec{\epsilon}}} \partial_{x}^2 v_k \right\|_{L^{
							2}_{xt}(\mathbb{R}^2)}\\
					&\lesssim \left\|  \partial_x \chi_{(t_{0},x_{0}),\boldsymbol{\vec{\epsilon}}}\|_{L^{
							\infty}_{xt}(\mathbb{R}^2)} \,\|t\chi_{(t_{0},x_{0}), \tilde{\vec{\epsilon}}} \partial_{x}^{3} v_k \right\|_{L^{2}(\mathbb{R}^2)} \\
					&\quad + \left\| t \partial_{x}^2 \chi_{(t_{0},x_{0}),\boldsymbol{\vec{\epsilon}}} \chi_{(t_{0},x_{0}), \tilde{\vec{\epsilon}}} \partial_{x}^2 v_k \right\|_{L^{2}(\mathbb{R}^2)}\\
					&\lesssim A_{8}^{k}k!,\quad \forall\, k\in \mathbb{N}_{0}.
				\end{split}
			\end{equation*}
		\end{claimproof}
		\begin{claim}\label{ct3.3-4}
			Let $t_0 \neq 0$, $x_0 \in \mathbb{R}$, and let $\varepsilon$ satisfy $0 < \varepsilon < |t_0|$. Define the intervals $I_{t_0} = (t_0 - \varepsilon, t_0 + \varepsilon)$ and $I_{x_0} = (x_0 - \varepsilon, x_0 + \varepsilon)$, and let $\vec{\varepsilon} = (\varepsilon, \varepsilon)$. 
			
			Let $\chi_{(t_{0},x_{0}),\vec{\varepsilon}}$ be the cutoff function defined in \eqref{cutoff1}. Suppose there exists a constant $c > 0$ such that the estimate
			\begin{equation}\label{hipo1}
				\left\|\chi_{(t_{0},x_{0}),\vec{\varepsilon}} \, P^{k} v \right\|_{H^{7/2}_{xt}(\mathbb{R}^2)} \leq c A_{2}^{k} \, k!
			\end{equation}
			holds for all $k \in \mathbb{N}_0$. Then we have
			\begin{equation*}
				\sup_{t \in I_{t_0}} \left\| (t^{1/3} \partial_{x})^{l} P^k v \right\|_{H^{1}_{x}(I_{x_0})} \le c A_{4}^{k + l} \, (k+l)!
			\end{equation*}
			for all $k, l \in \mathbb{N}_0$.
		\end{claim}
		\begin{claimproof} 
			Introduce the operator abbreviation $\mathcal{M} := t^{1/3} \partial_x$. For any fixed $t \in I_{t_0}$, note that
			\begin{equation*}
				\mathcal{M}^{3} P^{k} v	= - \frac{1}{3} P^{k +1}v + \frac{1}{3} x \partial_{x} P^{k} v -  t \partial_{x}(P + 2)^{k} (v^4). 
			\end{equation*}
			We will show that for some positive constants $c$ and $A_4$, the following inequality holds:
			\begin{equation}\label{EqX}
				\left\| \mathcal{M}^l P^k v \right\|_{H^1_x\left(I_{x_0}\right)} \le c A_4^{k+l} (k+l)! , 
				\quad k, l = 0, 1, 2, \dots
			\end{equation}
			We establish this inequality by induction with respect to $l$. In the cases $l = 0,1,2$, it is easy to show that \eqref{EqX} follows directly from the assumption \eqref{hipo1}. In fact, by the trace theorem,
			\begin{equation}\label{4.29}
				\begin{split}
					\left\|  \mathcal{M}^l P^k v(t, \cdot) \right\|_{H^1_x(I_{x_0})} 
					&\le \left\| t^{l/3} \partial_x^l P^k v(t, \cdot) \right\|_{H^1_x(I_{x_0})}  \\
					&\le (|t_0| + \varepsilon)^{l/3} \left\| \partial_x^l P^k v \right\|_{H^{1}_{xt}(I_{t_0} \times I_{x_0})}  \\
					&\le (|t_0| + \varepsilon)^{l/3} \left\| P^k v \right\|_{H^{7/2}(I_{t_0} \times I_{x_0})} \\
					&\le (|t_0| + \varepsilon)^{l/3} \left\| \chi_{(t_{0},x_{0}),\vec{\varepsilon}}  P^k v \right\|_{H^{7/2}_{xt}(\mathbb{R}^2)}\\
					&\le c A_2^{k} k! \\
					&\le c A_2^{k+l} (k + l) ! ,
				\end{split}
			\end{equation}
			where we take $c \sim  (|t_0|+\varepsilon)^{l/3} $ and $A_4 = \max\{1, A_2\}$.
			
			Assume that \eqref{EqX} is valid up to $l$ with  $l \ge 2$. Since
			\begin{equation*}
				\mathcal{M}^3 P^k v = -\frac{1}{3} P^{k+1} v +\frac{1}{3} x \partial_x P^k v - t \partial_x (P + 2)^{k} (v^4) ,
			\end{equation*}
			we have
			\begin{equation*}\label{4.30}
				\begin{split}
					\left\| \mathcal{M}^{l+1} P^k v \right\|_{H^1_x(I_{x_0})} 
					&\le \left\| \mathcal{M}^{l-2} \mathcal{M}^3 P^k v \right\|_{H^1_x(I_{x_0})} \\
					&\le \frac{1}{3} \left\| \mathcal{M}^{l-2} P^{k+1} v \right\|_{H^1_x(I_{x_0})} 
					+ \left\| \mathcal{M}^{l-2} x \partial_x P^k v \right\|_{H^1_x(I_{x_0})}  \\
					&\quad + \left\| t \mathcal{M}^{l-2}  \partial_x (P + 2)^{k} (v^4)\right\|_{H^1_x(I_{x_0})}\\
					&= \mathcal{I}_1 + \mathcal{I}_2 + \mathcal{I}_3 .
				\end{split}
			\end{equation*}
			By the induction hypothesis, it is clear that 
			\begin{equation*}\label{4.31}
				\mathcal{I}_1 \le \frac{1}{3} c A_4^{k+l+1} (k + l + 1)! \quad \text{for } l,k \ge 0.
			\end{equation*}
			Since 
			\begin{equation*}
				\mathcal{M}^{l-2} (x \partial_x) = x \partial_x \mathcal{M}^{l-2} + (l-2) \mathcal{M}^{l-2} \; \mbox{for } \; \; l = 3,4,\dots,
			\end{equation*}
			then
			\begin{equation*}\label{4.32}
				\begin{split}
					\mathcal{I}_2 &\le \left\| x \partial_x \mathcal{M}^{l-2} P^k v \right\|_{H^1_x(I_{x_0})} + (l-2) \left\| \mathcal{M}^{l-2} P^k v \right\|_{H^1_x(I_{x_0})}  \\
					&\le \left\| x t^{-1/3} \mathcal{M}^{l-1} P^k v \right\|_{H^1_x(I_{x_0})} + (l-2) \left\| \mathcal{M}^{l-2} P^k v \right\|_{H^1_x(I_{x_0})}  \\
					&\le c |t_0|^{-1/3} (|x_0| + \varepsilon + 1) \left\| \mathcal{M}^{l-1} P^k v \right\|_{H^1_x(I_{x_0})} \\
					&\quad + (l-2) \left\| \mathcal{M}^{l-2} P^k v \right\|_{H^1_x(I_{x_0})} \\
					&\le  c A_4^{k+l+1} (k + l - 1)!, 
				\end{split}
			\end{equation*}
			where $c \sim   \Big( |t_0|^{-\frac{1}{3}} (|x_0| + \epsilon +1) + 1 \Big)$. 
			
			Now, we  estimate $ \mathcal{I}_{3}$ which satisfies
			\begin{equation*}
				\begin{split}
					\mathcal{I}_3 &= \biggl\|  t  \mathcal{M}^{l-2} \, \partial_x (P + 2)^k (v^4)  
					\biggr\|_{H_x^{1}(I_{x_0})}\\
					&\lesssim 
					\sum_{\substack{k=k_{1}+k_{2}+k_{3}+k_{4} +k_{5}\\ 0 \leq k_{1},k_{2},k_{3},k_{4},k_{5} \leq k}} 
					2^{k_1} \frac{k!}{k_{1}!k_{2}!k_{3}!k_{4}!k_{5}!}
					\biggl\| t^{\tfrac{2}{3}} \mathcal{M}^{l-1}\left(v_{k_2} v_{k_3} v_{k_4} v_{k_5}\right) 
					\biggr\|_{H_x^{1}(I_{x_0})}.
				\end{split}
			\end{equation*}
			Applying the inductive hypothesis yields the bound on $\mathcal{I}_{3}$:
			\begin{equation*}
				\begin{split}
					\mathcal{I}_{3}
					&\lesssim (|t_0| + \varepsilon)^{\tfrac{2}{3}} 
					\sum_{\substack{k=k_{1}+\dots+k_{5} \\ 0 \leq k_{j} \leq k}} 
					\frac{2^{k_1} \, k!}{k_{1}! \, k_{2}! \, k_{3}! \, k_{4}! \, k_{5}!}
					\left( 
					\sum_{\substack{l-1=l_{1}+\dots+l_{4} \\ 0 \leq l_{j} \leq l-1}} 
					\frac{(l-1)!}{l_{1}! \, l_{2}! \, l_{3}! \, l_{4}!}
					\Big\| \prod_{j=1}^{4} M^{l_j} v_{k_{j+1}} \Big\|_{H_x^1(I_{x_0})} 
					\right) \\
					&\lesssim (|t_0| + \varepsilon)^{\tfrac{2}{3}}
					\sum_{\substack{k=k_{1}+\dots+k_{5} \\ 0 \leq k_{j} \leq k}}  
					\sum_{\substack{l-1=l_{1}+\dots+l_{4} \\ 0 \leq l_{j} \leq l-1}} 
					\frac{(l_1 + k_2)!}{l_1! \, k_2!}
					\frac{(l_2 + k_3)!}{l_2! \, k_3!}
					\frac{(l_3 + k_4)!}{l_3! \, k_4!}
					\frac{(l_4 + k_5)!}{l_4! \, k_5!}
					\frac{2^{k_1} \, k!}{k_1!} (l - 1)! \, A_{2}^{\,l - 1 + k - k_1} \\
					&\lesssim (|t_0| + \varepsilon)^{\tfrac{2}{3}}
					\sum_{\substack{k=k_{1}+\dots+k_{5} \\ 0 \leq k_{j} \leq k}}  
					\frac{A_{2}^{\,l - 1 + k - k_1} 2^{k_1} \, k! \, (l - 1)!}{k_1! \, k_2! \, k_3! \, k_4! \, k_5!}
					\sum_{\substack{l-1=l_{1}+\dots+l_{4} \\ 0 \leq l_{j} \leq l-1}} 
					\frac{(l - 1 + k - k_1)!}{l_1! \, l_2! \, l_3! \, l_4!}
					\frac{\prod_{j = 1}^{4} (l_j + k_{j+1})!}{(l - 1 + k - k_1)!}.
				\end{split}
			\end{equation*}
			Regrouping the combinations and extending the inner summation range over the full partition of the shifted index then simplifies the expression directly to:
			\begin{equation*}
				\begin{split}
					\mathcal{I}_3 
					&\lesssim \left(|t_0| + \varepsilon\right)^{\tfrac{2}{3}}
					\sum_{\substack{k=k_{1}+\dots+k_{5} \\ 0 \leq k_{j} \leq k}}  
					\frac{A_{2}^{\,l - 1 + k - k_1} 2^{k_1} \, k! \, (l - 1)!}{k_1! \, k_2! \, k_3! \, k_4! \, k_5!}
					\left(\sum_{\substack{l-1+k-k_{1}=l_{1}+\dots+l_{4} \\ 0 \leq l_{j} \leq l-1+k-k_{1}}} 
					\frac{(l - 1 + k - k_1)!}{l_1! \, l_2! \, l_3! \, l_4!}\right).
				\end{split}
			\end{equation*}
			Nevertheless, by the multinomial theorem,
			\begin{equation*}
				\sum_{\substack{l-1=l_{1}+\dots+l_{4} \\ 0 \leq l_{j} \leq l-1}} 
				\frac{(l - 1 + k - k_1)!}{l_1! \, l_2! \, l_3! \, l_4!}
				= 4^{\,l - 1 + k - k_1}.
			\end{equation*}
			Thus,
			\begin{equation*}
				\begin{split}
					\mathcal{I}_3 
					&\lesssim \left(|t_0| + \varepsilon\right)^{\tfrac{2}{3}}
					\big(\max(4, A_{2})\big)^{\,l - 1 + k}\left(
					\sum_{\substack{k=k_{1}+k_{2}+k_{3}+k_{4} +k_{5}\\ 0 \leq k_{1},k_{2},k_{3},k_{4},k_{5} \leq k}} 
					\frac{
						\left(\tfrac{2}{A_{2}}\right)^{k_1} 
						k! \, (l - 1)!
					}{
						k_1! \, k_2! \, k_3! \, k_4! \, k_5!
					}
					\right)\\
					&\lesssim (|t_0| + \varepsilon)^{\tfrac{2}{3}}
					\big(\max(4, A_{2})\big)^{\,l - 1 + k}
					(l - 1)!\left(
					\sum_{\substack{k=k_{1}+k_{2}+k_{3}+k_{4} +k_{5}\\ 0 \leq k_{1},k_{2},k_{3},k_{4},k_{5} \leq k}}
					\frac{k!}{k_{1}! \, k_{2}! \, k_{3}! \, k_{4}! \, k_{5}!}\right)
					\\
					&\lesssim (|t_0| + \varepsilon)^{\tfrac{2}{3}}
					\big(\max(4, A_{2})\big)^{\,l - 1 + k}
					5^k \,
					(l - 1)!,
				\end{split}
			\end{equation*}
			which finishes the proof of the claim.
		\end{claimproof}
		\begin{claim}\label{ct3.3-5}
			Let $x_{0} \in \mathbb{R}$, $t_{0} \in \mathbb{R}\setminus\{0\}$, and $0 < \varepsilon < |t_0|$. Let $I_{x_{0}} = (x_0 - \varepsilon, x_0 + \varepsilon)$ and $J_{t_0} = (t_0 - \varepsilon, t_0 + \varepsilon)$. Assume there exist positive constants $c$ and $A_4$ such that the bound
			\begin{equation*}
				\sup_{t \in J_{t_0}} \left\| \partial_x^l P^k v(t, \cdot) \right\|_{H^1(I_{x_{0}})} \le c A_4^{k+l} (k+l)!
			\end{equation*}
			holds for all $k, l \in \mathbb{N}_{0}$. Then, there exist positive constants $\widetilde{c}$ and $A_5$ such that the estimate
			\begin{equation*}
				\sup_{t \in J_{t_0}} \left\| \partial_t^m \partial_x^l v(t, \cdot) \right\|_{H^1(I_{x_{0}})} \le \widetilde{c} A_5^{m+l} (m+l)!
			\end{equation*}
			is satisfied for all $m, l \in \mathbb{N}_{0}$, where the constants $\widetilde{c}$ and $A_5$ depend only on $c$, $A_4$, $\varepsilon$, and $(t_0,x_0)$.
		\end{claim}
		\begin{claimproof}
			Fix $t \in J_{t_0}$. 
			Notice that the case $m=0$ follows directly from the hypothesis, so there's nothing to prove. In the case $m=1$ we have 
			\begin{equation*}
				\begin{split}
					\left\|(x \partial_{x}) \partial_x^{l} P^{k} v \right\|_{H_{x}^{1}(I_{x_0})} &\leq (1+\epsilon+|x_{0}|)\left\|\partial_x^{l+1} P^{k} v \right\|_{H_{x}^{1}(I_{x_0})}\\
					&\leq c(1+\epsilon+|x_{0}|)A_{4}^{k+l+1}(k+l+1)!
				\end{split}
			\end{equation*}
			First, we show that for some positive constants $C_6, A_6$, and $B_6$ 
			\begin{equation}\label{Key-Estimate}
				\left\|(x \partial_{x})^{m} \partial_x^{l} P^{k} v \right\|_{H_{x}^{1}(I_{x_0})} \leq C_6 A_{6}^{k + m + l} B_{6}^{m} (k + m + l)!    \quad \mbox{for}\quad	k, l, m = 0 , 1, 2, \dots.
			\end{equation}
			We will use an  induction argument on $m$, assuming \eqref{Key-Estimate} valid up to $m$:
			\begin{align*}
				\left\|(x \partial_x)^{m+1} \partial_{x}^{l} P^{k} v \right\|_{H^{1}_{x}(I_{x_0})} &= \left\| (x \partial_x) (x \partial_x)^{m} \partial_x^{l} P^{k} v \right\|_{H_{x}^{1}(I_{x_{0}})}\\
				&\leq (|x_0| + \epsilon + 1) \left\|(x \partial_x + 1)^{m} \partial_{x}^{l + 1} P^{k} v \right\|_{H^{1}_{x}(I_{x_0})} \\
				&\leq C(|x_0|, \epsilon) \sum_{m_1 = 0}^{m} \frac{m!}{m_1 ! (m - m_1)!} \left\|(x \partial_x)^{m_1} \partial_{x}^{l +1} P^{k} v \right\|_{H^{1}_{x}(I_{x_0})} \\
				&\leq  C(|x_0|, \epsilon) \sum_{m_1 = 0}^{m} \frac{m!}{m_1 ! (m - m_1)!} C_6 A_6^{k + m_1 + l +1} B_{6}^{m} (k + m_1 + l + 1)! \\
				&\leq C_6 A_{6}^{k + m + l + 1} B_6^{m} (k + m + l + 1)! \sum_{m_1 = 0 }^{m} \frac{(A_6 B_6)^{- (m - m_1)}}{(m - m_1)!} \frac{m!}{m_1 !} \frac{(k + m_1 + l + 1)!}{(k + m + l + 1)!}\\
				&\leq e^{- A_6 B_6} C_6 A_{6}^{k + m + l +1} B_6^{m} (k + m + l + 1)!, 
			\end{align*}
			where we take $B_6$ so large that 
			\begin{equation*}
				B_6 \geq \max ((|x_0| + \epsilon + 1) e^{- A_6 B_6} , 1). 
			\end{equation*}
			Next, we show that for some positive constants $C_7, A_7$ such that 
			\begin{equation*}
				\left\|(t \partial_{t})^{m} \partial_{x}^{l} v \right\|_{H^{1}(I_{x_0})} \leq C_7 A_{7}^{l + m} (l + m)!, 
			\end{equation*}
			for $(l, m = 0, 1, 2, \ldots)$. Note that $t \partial_{t} = \frac{1}{\textbf{3}} (P - x \partial_x)$, then 
			\begin{align*}
				\left\|(t \partial_{t})^{m} \partial_{x}^{l} v \right\|_{H^{1}(I_{x_0})} &\leq 3^{-m} \sum_{m_1 + m_2 = m} \frac{m!}{m_1 ! m_2 !} \left\|(x \partial_x)^{m_1} P^{m_2} \partial_{x}^{l} v \right\|_{H^{1}_{x}(I_{x_0})} \\
				&= 3^{-m} \sum_{m_1 + m_2 = m} \frac{m!}{m_1 ! m_2 !} \left\|(x \partial_x)^{m_1} \partial_{x}^{l} (P - l)^{m_2}v \right\|_{H^{1}_{x}(I_{x_0})} .
			\end{align*}
			To justify the last equality, it is enough to note that (by induction on $m_2$) we have 
			\begin{equation*}
				P^{m_2 +1}(\partial_x^{l} v) = P P^{m_2}(\partial_x^{l} v) = P(\partial_x^{l} (P - l)^{m_2} v)
			\end{equation*}
			\begin{equation*}
				= (P\partial_x^{l}) (P- l)^{m_2} v = \partial_{x}^{l} (P - l)^{m_2 + 1} v. 
			\end{equation*}
			Therefore, by the induction hypothesis, we have with $B_7 = A_6 B_6 (\geq 1)$ that 
			\begin{align*}
				\left\|(t \partial_{t})^{m} \partial_{x}^{l} v \right\|_{H^{1}(I_{x_0})} &\leq  3^{-m} \sum_{m_1 + m_2 + m_3 = m} \frac{m!}{m_1 ! m_2 ! m_3!} l^{m_3} \left\|(x \partial_x)^{m_1} \partial_{x}^{l} P^{m_2} v\right\|_{H^{1}_{x}(I_{x_0})}\\
				&\leq 3^{-m} \sum_{m_1 + m_2 + m_3 = m} \frac{m!}{m_1 ! m_2 ! m_3!} l^{m_3} C_6 B_{7}^{m_1 + m_2 + l} (m_1 + m_2 + l)! \\
				&= 3^{-m} C_6 B_{7}^{m + l} (m + l)!  \sum_{m_1 + m_2 + m_3 = m} B_{7}^{-m_3}  \frac{m!}{m_1 ! m_2 ! m_3!} \underbrace{l^{m_3}  \frac{(m_1 + m_2 + l)!}{(m + l)!} }_{\leq 1 }\\
				&\leq  C_6 B_{7}^{m + l} (l + m)!  \; \; (\mbox{by using the multinomial theorem}).
			\end{align*}
			Finally, we give a process to remove $t$ in the time derivative in \eqref{Key-Estimate}. To this end, we show that for some positive constants $C_{8}, A_{8}$ and $B_{8}$ we have 
			\begin{equation}\label{MainEquation}
				\left\|(t \partial_t)^{j} \partial_{t}^{m} \partial_{x}^{l} v\right\|_{H^{1}_{x}(I_{x_0})} \leq C_8 A_{8}^{j + m + l} B_{8}^{m} (j + m + l)!,    
			\end{equation}
			for $j, l, m = 0, 1, 2, \ldots$. We proceed by using the induction argument on $m$. It follows that 
			\begin{align*}
				&\left\|(t \partial_t)^{j} \partial_{t}^{m+1} \partial_{x} v \right\|_{H_{x}^{1}(I_{x_0})}\\
				 &= \left\|\partial_{t} (t \partial_t - 1)^{j} \partial_{t}^{m} \partial_{x}^{l} v \right\|_{H^{1}_{x}(I_{x_0})} \\
				&= |t|^{-1} \left\|t \partial_{t} (t \partial_t - 1)^{j} \partial_{t}^{m} \partial_{x}^{l} v \right\|_{H^{1}_{x}(I_{x_0})} \\
				&\leq 2 |t_0|^{-1} \sum_{j_1 = 0}^{j} \frac{j!}{j_1 ! (j - j_1)!} \left\|(t \partial_t)^{j_1 + 1} \partial_{t}^{m} \partial_{x}^{l} v \right\|_{H^{1}_{x}(I_{x_0})}\\
				&\leq 2 |t_0|^{-1}  \sum_{j_1 = 0}^{j} \frac{j!}{j_1 ! (j - j_1)!} C_{8} A_{8}^{j_1 + m + l + 1} B_{8}^{m} (j_1 + m + l + 1)! \; \\
				&=2 C_8 |t_0|^{-1}  A_{8}^{j + m + l + 1} B_{8}^{m}  (j + m + l +1)!   \sum_{j_1 = 0}^{j} \frac{A_{8}^{- (j - j_1)}}{(j - j_1)!} \underbrace{\frac{j!}{j_1 !} \frac{(j_1 + m + l +1)!}{(j + m + l +1)!}}_{\leq 1}\\
				&\leq 2 C_{8}  |t_0|^{-1}  e^{- A_{8}} A_{8}^{j + m + l + 1} B_{8}^{m}  (j + m + l +1)!  \\
				&\leq  C_{8}  A_{8}^{j + m + l + 1} B_{8}^{m + 1}  (j + m + l +1)!,    
			\end{align*}
			where we take $B_8 \geq  2 |t_0|^{-1} e^{- A_{8}}$. Finally, we choose $j = 0$ in \eqref{MainEquation} and take $C_5 = C_8$ and $A_{5} = A_8 B_8$ for obtaining the conclusion of proposition. 
		\end{claimproof}
		This completes the proof of analyticity. 
		
		We are now ready to prove the analyticity of solutions in both space and time variables. 
		\begin{prop}
			Suppose that there exist positive constants $C_4$ and $A_4$ such that
			\begin{equation*}
				\sup_{t_0 - \varepsilon <t < t_0 + \varepsilon} 
				\| \partial_x^l P^k v \|_{H^1(x_0 - \varepsilon, x_0 + \varepsilon)} 
				\le C_4 A_4^{k+l} (k+l)! , \quad k,l = 0,1,2,\dots
			\end{equation*}
			Then we have
			\begin{equation*}
				\sup_{t_0 - \varepsilon < t <  t_0 + \varepsilon} 
				\| \partial_t^m \partial_x^l v \|_{H^1(x_0 - \varepsilon, x_0 + \varepsilon)} 
				\le C_5 A_5^{m+l} (m+l)! , \quad l,m = 0,1,2,\dots
			\end{equation*}
			where the constants $C_5$ and $A_5$ only depend on $C_4$, $A_4$, $\varepsilon$ and $(t_0,x_0)$  with $0 < \varepsilon < \frac{|t_0|}{2}$.
		\end{prop}
		\begin{claimproof}
			Fix $t \in J_{t_0} = (t_0 - \epsilon, t_0 + \epsilon)$. First, we show that for some positive constants $C_6, A_6$, and $B_6$ 
			\begin{equation}\label{Key-Estimate}
				\|(x \partial_{x})^{m} \partial_x^{l} P^{k} v \|_{H_{x}^{1}(I_{x_0})} \leq C_6 A_{6}^{k + m + l} B_{6}^{m} (k + m + l)!    
			\end{equation}
			for $k, l, m = 0 , 1, 2, \ldots$. We use the induction argument on $m$, assuming \eqref{Key-Estimate} valid up to $m$:
			\begin{align*}
				&\|(x \partial_x)^{m+1} \partial_{x}^{l} P^{k} v \|_{H^{1}_{x}(I_{x_0})} \\
				&=  \| (x \partial_x) (x \partial_x)^{m} \partial_x^{l} P^{k} v \|_{H_{x}^{1}(I_{x_{0}})}\\
				&\leq (|x_0| + \epsilon + 1) \|(x \partial_x + 1)^{m} \partial_{x}^{l + 1} P^{k} v \|_{H^{1}_{x}(I_{x_0})} \\
				&\leq C(|x_0|, \epsilon) \sum_{m_1 = 0}^{m} \frac{m!}{m_1 ! (m - m_1)!} \|(x \partial_x)^{m_1} \partial_{x}^{l +1} P^{k} v \|_{H^{1}_{x}(I_{x_0})} \\
				&\leq  C(|x_0|, \epsilon) \sum_{m_1 = 0}^{m} \frac{m!}{m_1 ! (m - m_1)!} C_6 A_6^{k + m_1 + l +1} B_{6}^{m} (k + m_1 + l + 1)! \\
				&\leq C_6 A_{6}^{k + m + l + 1} B_6^{m} (k + m + l + 1)! \sum_{m_1 = 0 }^{m} \frac{(A_6 B_6)^{- (m - m_1)}}{(m - m_1)!} \frac{m!}{m_1 !} \frac{(k + m_1 + l + 1)!}{(k + m + l + 1)!}\\
				&\leq e^{- A_6 B_6} C_6 A_{6}^{k + m + l +1} B_6^{m} (k + m + l + 1)!, 
			\end{align*}
			where we take $B_6$ so large that 
			\begin{equation*}
				B_6 \geq \max ((|x_0| + \epsilon + 1) e^{- A_6 B_6} , 1). 
			\end{equation*}
			Next, we show that for some positive constants $C_7, A_7$ such that 
			\begin{equation*}
				\left\|(t \partial_{t})^{m} \partial_{x}^{l} v \right\|_{H^{1}(I_{x_0})} \leq C_7 A_{7}^{l + m} (l + m)!, 
			\end{equation*}
			for $(l, m = 0, 1, 2, \ldots)$. Note that $t \partial_{t} = \frac{1}{\textbf{3}} (P - x \partial_x)$, then 
			\begin{align*}
				\|(t \partial_{t})^{m} \partial_{x}^{l} v \|_{H^{1}(I_{x_0})} &\leq 3^{-m} \sum_{m_1 + m_2 = m} \frac{m!}{m_1 ! m_2 !} \left\|(x \partial_x)^{m_1} P^{m_2} \partial_{x}^{l} v \right\|_{H^{1}_{x}(I_{x_0})} \\
				&= 3^{-m} \sum_{m_1 + m_2 = m} \frac{m!}{m_1 ! m_2 !} \left\|(x \partial_x)^{m_1} \partial_{x}^{l} (P - l)^{m_2}v \right\|_{H^{1}_{x}(I_{x_0})} .
			\end{align*}
			To justify the last equality, it is enough to note that (by induction on $m_2$) we have 
			\begin{equation*}
				P^{m_2 +1}(\partial_x^{l} v) = P P^{m_2}(\partial_x^{l} v) = P(\partial_x^{l} (P - l)^{m_2} v)
			\end{equation*}
			\begin{equation*}
				= (P\partial_x^{l}) (P- l)^{m_2} v = \partial_{x}^{l} (P - l)^{m_2 + 1} v. 
			\end{equation*}
			Therefore, by the induction hypothesis, we have with $B_7 = A_6 B_6 (\geq 1)$ that 
			\begin{align*}
				\|(t \partial_{t})^{m} \partial_{x}^{l} v \|_{H^{1}(I_{x_0})} &\leq  3^{-m} \sum_{m_1 + m_2 + m_3 = m} \frac{m!}{m_1 ! m_2 ! m_3!} l^{m_3} \|(x \partial_x)^{m_1} \partial_{x}^{l} P^{m_2} v\|_{H^{1}_{x}(I_{x_0})}\\
				&\leq 3^{-m} \sum_{m_1 + m_2 + m_3 = m} \frac{m!}{m_1 ! m_2 ! m_3!} l^{m_3} C_6 B_{7}^{m_1 + m_2 + l} (m_1 + m_2 + l)! \\
				&= 3^{-m} C_6 B_{7}^{m + l} (m + l)!  \sum_{m_1 + m_2 + m_3 = m} B_{7}^{-m_3}  \frac{m!}{m_1 ! m_2 ! m_3!} \underbrace{l^{m_3}  \frac{(m_1 + m_2 + l)!}{(m + l)!} }_{\leq 1 }\\
				&\leq  C_6 B_{7}^{m + l} (l + m)!  \; \; (\mbox{by using the multinomial theorem}).
			\end{align*}
			Finally, we give a process to remove $t$ in the time derivative in \eqref{Key-Estimate}. To this end, we show that for some positive constants $C_{8}, A_{8}$ and $B_{8}$ we have 
			\begin{equation}\label{MainEquation}
				\|(t \partial_t)^{j} \partial_{t}^{m} \partial_{x}^{l} v\|_{H^{1}_{x}(I_{x_0})} \leq C_8 A_{8}^{j + m + l} B_{8}^{m} (j + m + l)!,    
			\end{equation}
			for $j, l, m = 0, 1, 2, \ldots$. We proceed by using the induction argument on $m$. It follows that 
			\begin{align*}
				\|(t \partial_t)^{j} \partial_{t}^{m+1} \partial_{x} v \|_{H_{x}^{1}(I_{x_0})} &= \|\partial_{t} (t \partial_t - 1)^{j} \partial_{t}^{m} \partial_{x}^{l} v \|_{H^{1}_{x}(I_{x_0})} \\
				&= |t|^{-1} \|t \partial_{t} (t \partial_t - 1)^{j} \partial_{t}^{m} \partial_{x}^{l} v \|_{H^{1}_{x}(I_{x_0})} \\
				&\leq 2 |t_0|^{-1} \sum_{j_1 = 0}^{j} \frac{j!}{j_1 ! (j - j_1)!} \|(t \partial_t)^{j_1 + 1} \partial_{t}^{m} \partial_{x}^{l} v \|_{H^{1}_{x}(I_{x_0})}\\
				&\leq 2 |t_0|^{-1} \sum_{j_1 = 0}^{j} \frac{j!}{j_1 ! (j - j_1)!} C_{8} A_{8}^{j_1 + m + l + 1} B_{8}^{m} (j_1 + m + l + 1)! \; \\
				&= 2 C_8  |t_0|^{-1} A_{8}^{j + m + l + 1} B_{8}^{m}  (j + m + l +1)!   \sum_{j_1 = 0}^{j} \frac{A_{8}^{- (j - j_1)}}{(j - j_1)!} \underbrace{\frac{j!}{j_1 !} \frac{(j_1 + m + l +1)!}{(j + m + l +1)!}}_{\leq 1}\\
				&\leq  2 C_{8} |t_0|^{-1} e^{- A_{8}} A_{8}^{j + m + l + 1} B_{8}^{m}  (j + m + l +1)!  \\
				&\leq  C_{8}  A_{8}^{j + m + l + 1} B_{8}^{m + 1}  (j + m + l +1)!,    
			\end{align*}
	where we take $B_8 \geq  2 |t_0|^{-1} e^{- A_{8}}$. Finally, we choose $j = 0$ in \eqref{MainEquation} and take $C_5 = C_8$ and $A_{5} = A_8 B_8$ for obtaining the conclusion of proposition. 
	\end{claimproof}
	This completes the proof of analyticity. 
\end{proof}
	\appendix
	\section{Appendix: Technical Estimates}\label{apendiceA}
	\setcounter{teorema}{0}
	\renewcommand{\theteorema}{\thesection.\arabic{teorema}}
	\setcounter{claim}{0}
	\renewcommand{\theclaim}{\thesection.\arabic{claim}}
	In this appendix, we collect several technical estimates and auxiliary lemmas that are essential for the rigorous derivation of the main results. While these estimates are primarily technical in nature, they provide the necessary bounds for the multi-linear interactions and the convergence of the sequence spaces defined in the preceding sections. Their inclusion here ensures a self-contained presentation of the analytical framework while maintaining the clarity of the primary arguments in the main body of the work.
	
	Their inclusion here ensures a self-contained presentation of the analytical framework while maintaining the clarity of the primary arguments in the main body of the work.
	
	Since our main results are local in nature,  the analyticity of solutions is established in a neighborhood of a fixed point $(t_{0},x_{0})$ with $t_{0}\neq 0$. We first recall the precise notions of restriction and extension of distributions, together with the associated localized Fourier-Lebesgue-Sobolev norms.
	\setcounter{teorema}{0}
	\renewcommand{\theteorema}{\Alph{teorema}}
	\begin{defini}
		Let $\Omega \subset \mathbb{R}^n$ be an open set. Given a global distribution $f \in \mathcal{D}'(\mathbb{R}^n)$, its restriction to $\Omega$, denoted by $f|_\Omega \in \mathcal{D}'(\Omega)$, is defined by its action on test functions $\phi \in C^\infty_0(\Omega)$ as
		\begin{equation*}
			\langle f|_\Omega, \phi \rangle := \langle f, \widetilde{\phi} \rangle,
		\end{equation*}
		where $\widetilde{\phi} \in C^\infty_0(\mathbb{R}^n)$ denotes the extension of $\phi$ by zero outside of $\Omega$.
	\end{defini}
	\begin{defini}
		Let $\Omega \subset \mathbb{R}^n$ and $u \in \mathcal{D}'(\Omega)$. A global distribution $g \in \mathcal{D}'(\mathbb{R}^n)$ is said to be an extension of $u$ if $g|_\Omega = u$ in the sense of distributions.
	\end{defini}
	With this notion of extension at hand, we introduce localized versions of the $\widehat{H}^{r}_{s}$ norms on space-time intervals, which will will be the natural setting for the trace estimates established below.
	\begin{defini}\label{restridefin}
		Let $J_{t_0}, I_{x_0} \subset \mathbb{R}$ be open intervals containing $t_0$ and $x_0$, respectively, and let $U = J_{t_0} \times I_{x_0}$ be the associated space-time domain. The localized spatial norm and localized space-time norm are defined, respectively, as
		\begin{equation*}
			\|u\|_{\widehat{H}^{r}_{s}(I_{x_0})} := \inf \left\{ \|g\|_{\widehat{H}^{r}_{s}(\mathbb{R})} : g \in \widehat{H}^{r}_{s}(\mathbb{R}) \text{ and } g|_{I_{x_0}} = u \right\}
		\end{equation*}
		and
		\begin{equation*}
			\|v\|_{\widehat{H}^{r}_{s}(U)} := \inf \left\{ \|\widetilde{v}\|_{\widehat{H}^{r}_{s}(\mathbb{R}^2)} : \widetilde{v} \in \widehat{H}^{r}_{s}(\mathbb{R}^2) \text{ and } \widetilde{v}|_U = v \right\}.
		\end{equation*}
	\end{defini}
	Next, we establish a global trace bound for the Fourier-Lebesgue-Sobolev spaces $\widehat{H}^{r}_{s}(\mathbb{R}^2)$; this bound is what allows us to control the regularity of a solution at a fixed time slice, a step needed to close the continuity argument in the proof of Theorem \ref{main1}.
	\begin{teorema}\label{trace1}
		Let $r,s$ be real numbers such that $1<r\leq 2$  and $r'$ is its harmonic conjugate.	Let $\mu>\frac{1}{r}$ and $\mu\geq \frac{1}{r}+s.$ For any $t\in \mathbb{R}$ there exists a positive constant  such that for any $v\in \widehat{H}^{r}_{s}(\mathbb{R}^{2}):$
		\begin{equation*}
			\|v(t,\cdot )\|_{ \widehat{H}^{r}_{s}(\mathbb{R})}\leq c\|v\|_{ \widehat{H}^{r}_{s}(\mathbb{R}^{2})}.
		\end{equation*}
		\begin{proof}
			We  start by assuming that $v\in \mathcal{S}(\mathbb{R}^{2}).$ In this sense,  H\"older's inequality  yields
			\begin{equation*}
				\begin{split}
					\left|\mathcal{F}_{x}(v(t,\cdot))(\xi)\right|\leq  \left\|\frac{1}{\langle(\cdot,\xi)\rangle^{\mu }}\right\|_{L^{r}_{\tau}}\,\left(\int_{\mathbb{R}} \langle(\tau,\xi)\rangle^{\mu  r'}\,\left|\mathcal{F}_{t,x}(v)(\tau,\xi)\right|^{r'}\, \mathrm{d}\tau\right)^{\frac{1}{r'}}.
				\end{split}
			\end{equation*}
			Now, we perform a change of variables in the first integral. More precisely, we make $\tau=\langle \xi\rangle \eta$, so we get from the first integral the following expression  
			\begin{equation*}
				\left\|\frac{1}{\langle(\cdot,\xi)\rangle^{\mu }}\right\|_{L^{r}_{\tau}}\sim_{\mu,r}\frac{1}{\langle \xi\rangle^{\mu -\frac{1}{r}}},
			\end{equation*}
			whenever $\mu>\frac{1}{r}.$
			
			Now, 
			\begin{equation*}
				\begin{split}
					\langle \xi\rangle^{s}\,\left|\mathcal{F}_{x}(v(t,\cdot))(\xi)\right|\lesssim\frac{1}{\langle \xi\rangle^{\mu -\frac{1}{r}-s}}\,\left(\int_{\mathbb{R}} \langle(\tau,\xi)\rangle^{\mu  r'}\,\left|\mathcal{F}_{t,x}(v)(\tau,\xi)\right|^{r'}\, \mathrm{d}\tau\right)^{\frac{1}{r'}},
				\end{split}
			\end{equation*}
			from where we obtain for $\mu\geq s+\frac{1}{r}$ that 
			\begin{equation*}
				\left\|\langle \xi\rangle^{s}\,\mathcal{F}_{x}(v(t,\cdot))(\xi)\right\|_{L^{r'}_{\xi}}\lesssim \left\|\int_{\mathbb{R}} \langle(\tau,\xi)\rangle^{\mu  r'}\,\left|\mathcal{F}_{t,x}(v)(\tau,\xi)\right|^{r'}\, \mathrm{d}\tau\right\|_{L^{r'}_{\xi}}\sim \left\|v\right\|_{\widehat{H}^{r}_{s}(\mathbb{R}^{2})},
			\end{equation*}
			which is the required bound.
		\end{proof}
	\end{teorema}
	Now, we provide a localized version of the trace theorem restricted to intervals; it is precisely this version, rather than its global counterpart, that is invoked in the proof of Theorem \ref{main1} in Section \ref{sectA}.
	\begin{teorema}\label{trace2}
		Let $r, s \in \mathbb{R}$ satisfy $1 < r \leq 2$, and let $\mu \in \mathbb{R}$ such that
		\begin{equation*}
			\mu > \frac{1}{r} \quad \text{and} \quad \mu \geq s + \frac{1}{r}.
		\end{equation*}
		For fixed $t_0, x_0 \in \mathbb{R}$, let $J_{t_0}$ and $I_{x_0}$ be symmetric open intervals centered at $t_0$ and $x_0$, respectively. Then, for any $t \in J_{t_0}$, there exists a constant $c > 0$ such that the inequality
		\begin{equation*}
			\|v(t, \cdot)\|_{\widehat{H}^{r}_{s}(I_{x_{0}})} \leq c \|v\|_{\widehat{H}^{r}_{\mu}(J_{t_{0}} \times I_{x_{0}})}
		\end{equation*}
		holds for all $v \in \widehat{H}^{r}_{\mu}(J_{t_0} \times I_{x_0})$.
	\end{teorema}
	\begin{proof}
		Let $v\in \widehat{H}^{r}_{\mu}(J_{t_0} \times I_{x_0}).$  By definition \ref{restridefin}, for all $\epsilon>0,$ there exists $ \widetilde{v}\in \widehat{H}^{r}_{s}(\mathbb{R}^2),$ such that    $ \widetilde{v}|_{J_{t_0} \times I_{x_0}} = v $ and $\|\widetilde{v}\|_{\widehat{H}^{r}_{s}(\mathbb{R}^{2})} <\|\widetilde{v}\|_{\widehat{H}^{r}_{s}(J_{t_0} \times I_{x_0})} +\epsilon.$
		
		Additionally, notice that  for all $t\in J_{t_{0}}$, the restriction of the global extension  at time $t$ detonated by  $\widetilde{v}(t,\cdot)$ agrees with $v(t,\cdot)$ for  all $x\in I_{x_{0}}.$ Thus, $\widetilde{v}(t,\cdot)$ is a  global extension of  the localized trace at $v(t,\cdot)$, which implies that 
		\begin{equation*}
			\|v(t,\cdot)\|_{\widehat{H}^{r}_{\mu}( I_{x_0})}\leq \|\widetilde{v}\|_{\widehat{H}^{r}_{s}(\mathbb{R}^{2})}.
		\end{equation*}
		Now, Theorem \ref{trace1} implies that 
		\begin{equation*}
			\|\widetilde{v}(t,\cdot)\|_{\widehat{H}^{r}_{s}(\mathbb{R})}\lesssim \|\widetilde{v}\|_{\widehat{H}^{r}_{s}(\mathbb{R}^{2})}.
		\end{equation*}
		Finally, we combine the inequalities obtained above  to obtain 
		\begin{equation*}
			\|v(t, \cdot)\|_{\widehat{H}^{r}_{s}(I_{x_{0}})} \leq c\left(\|\widetilde{v}\|_{\widehat{H}^{r}_{s}(J_{t_0} \times I_{x_0})} +\epsilon\right),\quad \forall \,\, \epsilon>0,
		\end{equation*}
		which finishes the proof.
	\end{proof}
	Beyond the trace estimates above, the proof of Lemma~\ref{interpo1} below also requires controlling how multiplication by a fixed smooth cutoff interacts with the $\widehat{H}^{r}_{s}$ norms. This is summarized in the following lemma.
	\begin{lema}\label{bach1}
		Let $m\in C_{0}^{\infty}(\mathbb R^{2})$ and $p\in(1,\infty)$. Then multiplication by $m$ is bounded
		on $\widehat H^{\,p}_{s}(\mathbb R^{2})$ for every $s\in\mathbb R$, i.e. there exists a positive constant $c=c(s),$ such that 
		\begin{equation*}
			\|mf\|_{\widehat H^{\,p}_{s}(\mathbb R^{2})}\leq c \|m\|_{\widehat H^{\,\infty}_{|s|}(\mathbb R^{2})}\|f\|_{\widehat H^{\,p}_{s}(\mathbb R^{2})}.
		\end{equation*}
	\end{lema}
	\begin{proof}
		Since $m\in C_{0}^{\infty}(\mathbb{R}^{2})$, we have
		$\mathcal{F}(m)\in L^{1}(\mathbb{R}^{2})$. Moreover, it is straightforward to verify that
		\begin{equation*}
			\langle(\tau,\xi)\rangle^{s}
			\lesssim_{s}
			\langle(\tau-\tau',\xi-\xi')\rangle^{|s|}\,\langle(\tau',\xi')\rangle^{s},\quad \forall \tau,\tau',\xi,\xi'\in \mathbb{R},
		\end{equation*}
		hence			
		\begin{equation*}
			\langle(\tau,\xi) \rangle^{s}\,\left|\mathcal{F}(mf)(\tau,\xi)\right|
			\lesssim_{s}
			\left(\big(\langle\cdot, \cdot \rangle^{|s|}|\mathcal{F}(m)|\big)*
			\big(\langle\cdot, \cdot\rangle^{s}|\mathcal{ F}(f)|\big)\right)(\tau,\xi),\quad \forall (\tau,\xi)\in \mathbb{R}^{2}.
		\end{equation*}
		Thus, Young's inequality  yields 
		\begin{equation*}
			\begin{split}
				\|mf\|_{\widehat H^{\,p}_{s}}&\lesssim_{s}\left\|\big(\langle\cdot, \cdot\rangle^{|s|}\mathcal{F}( m)\big)*
				\big(\langle\cdot, \cdot\rangle^{s}\mathcal{F}(f)\big)\right\|_{L^{p'}_{\tau\xi}(\mathbb R^{2})}\\
				&\lesssim\|\langle\cdot, \cdot\rangle^{|s|}\mathcal{F}(m)\|_{ L^{1}_{\tau\xi}(\mathbb R^{2})}\,\|\langle\cdot, \cdot\rangle^{s}\mathcal{F}(f)\|_{L^{p'}_{\tau\xi}(\mathbb R^{2})},
			\end{split}
		\end{equation*}
		which finishes the proof.
	\end{proof}
	We now turn to the key technical estimate of this appendix. Roughly speaking, it shows that full $\widehat{H}^{r}_{s}$ regularity of a localized piece $g$ can be recovered once one controls the action of the vector field $t\partial_{x}^{3}$ and of the dilation generator $P = 3t\partial_{t}+x\partial_{x}$ on $g$; this is the mechanism through which analytic smoothing manifests itself in our Fourier-Lebesgue setting.
	\begin{lema}\label{interpo1}
		Let $s, r \in \mathbb{R}$ and $(t_0, x_0) \in \mathbb{R}^2$ with $t_0 \neq 0$. There exists $\epsilon_0 > 0$ such that for any $\epsilon \in (0, \epsilon_0]$ and any $g \in \mathcal{S}'(\mathbb{R}^2)$ satisfying $\supp(g) \subseteq B_{2\epsilon}\left((t_0, x_0)\right)$, if 
		\begin{equation*}
			g, \,\, t\partial_x^3 g, \,\, P^3 g \in \widehat{H}^r_{s-3}(\mathbb{R}^2),
		\end{equation*}
		then 
		\begin{equation*}
			\|g\|_{\widehat{H}^r_s} \lesssim_{t_0, x_0} \|g\|_{\widehat{H}^r_{s-3}} + \left\|t\partial_x^3 g\right\|_{\widehat{H}^r_{s-3}} + \left\|P^3 g\right\|_{\widehat{H}^r_{s-3}}.
		\end{equation*}
	\end{lema}
	\begin{proof}
		We start by noticing that for any point $(t_{0},x_{0})$  with $t_{0}\neq 0$, the following inequality holds
		\begin{equation}\label{ine1}
			\langle (\tau, \xi)\rangle ^{3}\lesssim_{t_{0},x_{0}} 1+ |t_{0}\xi^{3}|+|3t_{0}\tau+x_{0}\xi|.
		\end{equation}
		As a consequence of \eqref{ine1} we have 
		\begin{equation*}
			\begin{split}
				\|g\|_{\widehat{H}^r_s}&\lesssim_{t_{0},x_{0}} \left\|\left\langle (\tau,\xi)\right\rangle^{s-3}\mathcal{F}_{t,x}g(\tau,\xi)\right\|_{L^{r'}_{\tau\xi}}+\left\|\left\langle (\tau,\xi)\right\rangle^{s-3}|t_{0}\xi^{3}|\mathcal{F}_{t,x}g(\tau,\xi)\right\|_{L^{r'}_{\tau\xi}}\\
				&\quad +\left\|\left\langle (\tau,\xi)\right\rangle^{s-3}|3t_{0}\tau+x_{0}\xi|^{3}\mathcal{F}_{t,x}g(\tau,\xi)\right\|_{L^{r'}_{\tau\xi}}\\
				&=\Pi_{1}+\Pi_{2}+\Pi_{3}.
			\end{split}
		\end{equation*}
		Notice that for $\Pi_{1}$  it is straightforward  since 
		\begin{equation*}
			\begin{split}
				\Pi_{1}&=c(t_{0},x_{0}) \left\|\left\langle (\tau,\xi)\right\rangle^{s-3}\mathcal{F}_{t,x}g(\tau,\xi)\right\|_{L^{r'}_{\tau\xi}}\\
				&=c(t_{0},x_{0})\|g\|_{\widehat{H}^{r}_{s-3}}.
			\end{split}
		\end{equation*}
		To handle $ \Pi_{2}$ we   consider for  $ \chi_{(t_{0}, x_{0}), \widetilde{\vec{\boldsymbol{\epsilon}}}}$, with  $\widetilde{\vec{\boldsymbol{\epsilon}}}=(2\epsilon, 2\epsilon).$
		Notice that this particular choice implies that $\chi_{(t_{0}, x_{0}), \widetilde{\vec{\boldsymbol{\epsilon}}}} g=g,$ so that
		\begin{equation*}
			\mathcal{F}_{t,x}(t_{0}\partial_{x}^{3}g)(\tau, \xi)=\mathcal{F}_{t,x}(t\partial_{x}^{3}g)(\tau,\xi)+\mathcal{F}_{t,x}((t_{0}-t)\partial_{x}
			^{3}\left(\chi_{(t_{0}, x_{0}), \widetilde{\vec{\boldsymbol{\epsilon}}}} g\right)(\tau, \xi),
		\end{equation*}
		which implies after  replacing  into $\Pi_{2}$, the following equivalent expression 
		\begin{equation*}
			\begin{split}
				\Pi_{2}&=c(t_{0},x_{0}) \left\|\left\langle (\tau,\xi)\right\rangle^{s-3}|t_{0}\xi^{3}|\mathcal{F}_{t,x}g(\tau,\xi)\right\|_{L^{r'}_{\tau\xi}}\\
				&\leq c(t_{0},x_{0})\left\{\left\|\left\langle (\tau,\xi)\right\rangle^{s-3}\mathcal{F}_{t,x}(t\partial_{x}^{3}g)(\tau,\xi)\right\|_{L^{r'}_{\tau\xi}}+\left\|\left\langle (\tau,\xi)\right\rangle^{s-3}\mathcal{F}_{t,x}((t_{0}-t)\partial_{x}^{3}\left(\chi_{(t_{0}, x_{0}), \widetilde{\vec{\boldsymbol{\epsilon}}}}g\right)(\tau, \xi)\right\|_{L^{r'}_{\tau\xi}}\right\}\\
				&= c(t_{0},x_{0})\left\{ \|t\partial_{x}^{3}g\|_{\widehat{H}^{r}_{s-3}}+c\left\|\left\langle (\tau,\xi)\right\rangle^{s-3}\mathcal{F}_{t,x}((t_{0}-t)\partial_{x}^{3}\left(\chi_{(t_{0}, x_{0}), \widetilde{\vec{\boldsymbol{\epsilon}}}} g\right)(\tau, \xi)\right\|_{L^{r'}_{\tau\xi}}\right\}\\
				&\leq c(t_{0},x_{0})\left\{ \|t\partial_{x}^{3}g\|_{\widehat{H}^{r}_{s-3}}+ \left\|(t_{0}-t) \chi_{(t_{0}, x_{0}), \widetilde{\vec{\boldsymbol{\epsilon}}}}\langle \nabla_{t,x}\rangle^{s-3}\partial_{x}^{3}g\right\|_{\widehat{L}^{r}_{xt}}+\left\|\left[\langle \nabla_{t,x}\rangle^{s-3}\partial_{x}^{3}; (t_{0}-t)\chi_{(t_{0}, x_{0}), \widetilde{\vec{\boldsymbol{\epsilon}}}}\right]g\right\|_{\widehat{L}^{r}_{xt}}\right\}
				\\
				&=\Pi_{2,1}+\Pi_{2,2}+\Pi_{2,3}.
			\end{split}
		\end{equation*}
		The term $\Pi_{2,1}$ provides the required bound. To handle $\Pi_{2}$, applying Young's inequality for convolution yields
		\begin{equation*}
			\begin{split}
				\Pi_{2,2} &= c(t_{0},x_{0})\left\| \mathcal{F}_{t,x}\left( (t_{0}-t)\chi_{(t_{0}, x_{0}), \widetilde{\vec{\boldsymbol{\epsilon}}}} \right) \ast \mathcal{F}_{t,x}\left( \langle \nabla_{t,x} \rangle^{s-3} \partial_{x}^{3} g \right) \right\|_{L^{r'}_{\tau\xi}} \\
				&\leq c(t_{0},x_{0}) \left\| \mathcal{F}_{t,x}\left( (t_{0}-t)\chi_{(t_{0}, x_{0}), \widetilde{\vec{\boldsymbol{\epsilon}}}}\right) \right\|_{L^{1}_{\tau\xi}} \left\| \mathcal{F}_{t,x}\left( \langle \nabla_{t,x} \rangle^{s-3} \partial_{x}^{3} g \right) \right\|_{L^{r'}_{\tau\xi}} \\
				&\leq 2\epsilon c(t_{0},x_{0}) \|g\|_{\widehat{H}^{r}_{s}}.
			\end{split}
		\end{equation*}
		Next, we focus on the $\Pi_{2,3}$ term, which is estimated using a similar analysis as the one used in Lemma \ref{lemh} in the Appendix. From this estimation, we find that there exists a parameter $\delta_{1}>0$ (to be defined later) such that 
		\begin{equation*}
			\begin{split}
				\Pi_{2,3}&\leq \frac{c(t_{0},x_{0})}{\epsilon^{|s-3|+1}} \|g\|_{\widehat{H}^{r}_{s-1}(\mathbb{R}^{2})}\\
				&\leq \delta_{1}\epsilon \|g\|_{\widehat{H}^{r}_{s}(\mathbb{R}^{2})}+ \frac{1}{9}\left(\frac{2c(t_{0},x_{0})}{\delta_{1}\epsilon ^{|s-3|+2}}\right)^{2}\|g\|_{\widehat{H}^{r}_{s-3}(\mathbb{R}^{2})}.
			\end{split}
		\end{equation*}
		Finally, we handle $\Pi_{3},$ which we estimate after noticing the identity satisfied by $P$ and $P_{0}.$ More precisely,
		\begin{equation}\label{decompo2}
			P^{3}-P_{0}^{3}
			=
			\left(P-P_{0}\right)\left(P^{2}+P_{0}P+P_{0}^{2}\right)+3P_{0}P+6P_{0}^{2}-2x_{0}\partial_{x}P-4x_{0}\partial_{x}P_{0}.
		\end{equation} 
		Additionally, we fix a particular cutoff function that is decoupled in the $t$ and $x$ 
		variables, following the notation introduced in \eqref{cutoff1A}. More precisely, 
		for a given $\epsilon > 0$, we choose a positive real number $\gamma$ satisfying 
		$\gamma \ge 2$ and define the smooth indicator 
		function:
		\begin{equation*}
			\chi(t, x) = \chi_{1, t_{0}, \gamma \epsilon}(t) \, \chi_{2, x_{0}, \gamma \epsilon}(x), 
			\quad (t, x) \in \mathbb{R}^{2}.
		\end{equation*}
		It follows immediately that $g = \chi g$, which allows us to rewrite $\Pi_{3}$ according to \eqref{decompo2} as:
		It follows immediately that $g = \chi g$, which allows us to rewrite $\Pi_{3}$ according to \eqref{decompo2} as:
		\begin{equation}\label{Pi3}
			\begin{split}
				\Pi_{3} \leq c(t_{0},x_{0}) \left\{ 
				\vphantom{\left\|\mathcal{F}_{t,x}\left(\langle \nabla_{t,x}\rangle^{s-3}(P-P_{0})\underbrace{\left(P^{2}+P_{0}P+P_{0}^{2}\right)}_{:=\mathcal{Q}_{1}}(\chi g)\right)\right\|_{L^{r'}_{\tau\xi}}}
				\right. & \left. \left\|\mathcal{F}_{t,x}\left(\langle \nabla_{t,x}\rangle^{s-3}P^{3}g\right)\right\|_{L^{r'}_{\tau\xi}} \right. \\
				&+ \left. \left\|
				\mathcal{F}_{t,x}\left(\langle \nabla_{t,x}\rangle^{s-3}(P-P_{0})\underbrace{\left(P^{2}+P_{0}P+P_{0}^{2}\right)}_{=:\mathcal{Q}_{1}}(\chi g)\right)\right\|_{L^{r'}_{\tau\xi}} \right. \\
				&+ \left. \left\|\mathcal{F}_{t,x}\left(\langle \nabla_{t,x}\rangle^{s-3}\underbrace{\left(3P_{0}P+6P_{0}^{2}-2x_{0}\partial_{x}P-4x_{0}\partial_{x}P_{0}\right)}_{=:\mathcal{Q}_{2}}(\chi g)\right)\right\|_{L^{r'}_{\tau\xi}} \right\}\\
				&=\Pi_{3,1}+\Pi_{3,2}+\Pi_{3,3}.
			\end{split}
		\end{equation}
		The first term in the expression \eqref{Pi3}   clearly satisfies 
		\begin{equation*}
			\Pi_{3,1}=c(t_{0},x_{0}) \left\|P^{3}g\right\|_{\widehat{H}^{r}_{s-3}(\mathbb{R}^{2})},
		\end{equation*} 
		which provides the required bound.
		
		Next, we proceed to handle $\Pi_{3,2}$, which we rewrite as follows 
		\begin{equation*}
			\begin{split}
				\Pi_{3,2}&\leq c(t_{0},x_{0})\left\{\left\|\mathcal{F}_{t,x}\left(\langle \nabla_{t,x}\rangle^{s-3}\left[(P-P_{0})\mathcal{Q}_{1}; \chi\right]g\right)\right\|_{L^{r'}_{\tau,\xi}}+\left\|\mathcal{F}_{t,x}\left(\langle \nabla_{t,x}\rangle^{s-3}\left(\chi \,\left(P-P_{0}\right) \mathcal{Q}_{1}g\right)\right)\right\|_{L^{r'}_{\tau,\xi}}\right\}\\
				&\leq c(t_{0},x_{0})\left\{\left\|\mathcal{F}_{t,x}\left(\langle \nabla_{t,x}\rangle^{s-3}\left[(P-P_{0})\mathcal{Q}_{1}; \chi\right]g\right)\right\|_{L^{r'}_{\tau,\xi}}+\left\|\mathcal{F}_{t,x}\left(\left[\langle \nabla_{t,x}\rangle^{s-3}; (x-x_{0})\chi\right]\partial_{x}\mathcal{Q}_{1}g\right)\right\|_{L^{r'}_{\tau,\xi}}\right.\\
				&\left.+\left\|\mathcal{F}_{t,x}\left((x-x_{0})\chi \,\langle\nabla_{t,x}\rangle^{s-3}\partial_{x}\mathcal{Q}_{1}g\right)\right\|_{L^{r'}_{\tau,\xi}} +3\left\|\mathcal{F}_{t,x}\left(\left[\langle \nabla_{t,x}\rangle^{s-3}; (t-t_{0})\chi\right]\partial_{t}\mathcal{Q}_{1}g\right)\right\|_{L^{r'}_{\tau,\xi}} \right.\\
				&\left.+3\left\|\mathcal{F}_{t,x}\left((t-t_{0})\chi \,\langle\nabla_{t,x}\rangle^{s-3}\partial_{t}\mathcal{Q}_{1}g\right)\right\|_{L^{r'}_{\tau,\xi}}\right\}\\
				&=\Pi_{3,2,1}+\Pi_{3,2,2}+\Pi_{3,2,3}+\Pi_{3,2,4}+\Pi_{3,2,5}.
			\end{split}
		\end{equation*}
		Notice that by Lemma \ref{lemh}, stated in the appendix, we obtain
		\begin{equation*}
			\begin{split}
				\Pi_{3,2,1} &\leq \frac{c}{\epsilon^{y_{s}}} \|g\|_{\widehat{H}^{r}_{s-1}(\mathbb{R}^{2})} \\
				&\leq \delta_{2} \epsilon \|g\|_{\widehat{H}^{r}_{s}(\mathbb{R}^{2})} + \frac{c}{\epsilon^{\widetilde{y_{s}}}} \|g\|_{\widehat{H}^{r}_{s-3}(\mathbb{R}^{2})},
			\end{split}
		\end{equation*}
		where  $\widetilde{y_{s}}>0$ and $ \delta_{2}$ is a positive constant to be chosen. 
		
		With respect to $\Pi_{3,2,2}$ and $\Pi_{3,2,4}$, it follows from Lemma \ref{lemh} in the appendix that they satisfy
		\begin{equation*}
			\begin{aligned}
				\max\left\{\Pi_{3,2,2}, \Pi_{3,2,4}\right\} &\leq \frac{c}{\epsilon^{p_{s}}} \|g\|_{\widehat{H}^{r}_{s-1}(\mathbb{R}^{2})} \\
				&\leq \delta_{3}\epsilon \|g\|_{\widehat{H}^{r}_{s}(\mathbb{R}^{2})} + \dfrac{c}{\delta_{3}^{2}\epsilon^{p_{s}'}} \|g\|_{\widehat{H}^{r}_{s-3}(\mathbb{R}^{2})},
			\end{aligned}
		\end{equation*}
		where $\delta_{3}$ is a positive constant to be chosen, and $p_{s},p_{s}'>0.$
		
		The terms $\Pi_{3,2,3}$ and $\Pi_{2,3,5}$ are easily handled by using Lemma \ref{lemh}. More precisely, we have the following bound 
		\begin{equation*}
			\max\left\{\Pi_{3,2,3},\Pi_{2,3,5}\right\}\leq c\epsilon.
		\end{equation*}	
		Finally, collecting the estimates derived above and setting $\delta_{1} = \delta_{2} = \delta_{3} = 1/3$ and $0<\epsilon \ll1$ yields the desired result.
	\end{proof}
	\begin{obs}
		This lemma establishes a localization result within our newly defined Sobolev-Lebesgue Fourier spaces. By adapting the arguments presented in \cite{KO2000}, we extend the analysis to this broader class of spaces. Furthermore, in the specific case where $r=2$, this result recovers the classical estimates established by Kato and Ogawa \cite[Lemma 3.2]{KO2000}. This is precisely the estimate invoked in {\sc claim }~\ref{c3} of the proof of Theorem~\ref{main1}.
	\end{obs}
	For comparison with the classical, $L^2$-based ($r=2$) setting, we state below the corresponding estimate established by Kato and Ogawa \cite{KO2000}, which our Lemma~\ref{interpo1} generalizes to the Fourier-Lebesgue scale.
	\begin{lema}\label{LemaA}
		Let $P = 3t \partial_t + x \partial_x$ be the generator of the dilation and $|\nabla_{t,x}|$ be
		an operator defined by
		\begin{equation*}
			\widehat{|\nabla_{t,x}| f}(\tau, \xi) = \big( |\tau| + |\xi| \big)\,\widehat{f}(\tau, \xi).
		\end{equation*}
		We fix an arbitrary point $(t_0, x_0) \in \big( (-T,0) \cup (0,T) \big) \times \mathbb{R}$.
		\begin{itemize}
			\item Let $\epsilon >0$.  Suppose that $b \in (0,1]$, $r \in (-\infty,0]$ and $g \in X^{r, b-1}$ with  $\operatorname{supp} g \subset B_{2\epsilon}(t_0, x_0)$ and  
			$t \partial_x^3 g, \; P^3 g \in X^{r, b-1}$.  
			If $\epsilon > 0$ is sufficiently small, then we have
			\begin{equation}\label{EqA}
				\| \langle \nabla_{t,x} \rangle^{3b} g \|_{L^2(\mathbb{R}; H^r(\mathbb{R}))} 
				\;\leq\; C \Big( \| g \|_{X^{r, b-1}} 
				+ \| t \partial_x^3 g \|_{X^{r, b-1}} 
				+ \| P^3 g \|_{X^{r, b-1}} \Big),  
			\end{equation}
			where the constant $C$ depends on $(t_0, x_0)$ and $\varepsilon$.
			\item  If $g \in H^{\mu-3}(\mathbb{R}^2)$ with  
			$\operatorname{supp} g \subset B_{2\varepsilon}(t_0, x_0)$ and  
			$t \partial_x^3 g, \; P^3 g \in H^{\mu-3}(\mathbb{R}^2)$,  
			then for $0<\varepsilon\ll1$, we have
			\begin{equation}\label{EqB}
				\|\langle \nabla_{t,x} \rangle^{\mu} g\|_{L^2(\mathbb{R}^2)} 
				\;\leq\; C \Big( \| g \|_{H^{\mu-3}(\mathbb{R}^2)} 
				+ \| t \partial_x^3 g \|_{H^{\mu-3}(\mathbb{R}^2)} 
				+ \| P^3 g \|_{H^{\mu-3}(\mathbb{R}^2)} \Big),  
			\end{equation}
			where the constant $C$ depends on $(t_0, x_0)$ and $\varepsilon$.
		\end{itemize}
	\end{lema}
	\begin{proof}
		See \cite[p.~588, Lemma~3.2]{KO2000}.
	\end{proof}
	The estimates used in the proof of Lemma~\ref{interpo1} rely strongly on the following commutator estimates, which we prove separately for the sake of clarity.
	\begin{lema}\label{lemh}
		Let $s \in \mathbb{R}$ and $(t_{0}, x_{0}) \in \mathbb{R}^{2}$ with $t_{0} \neq 0$. Given $\epsilon \in (0,1)$, choose a positive real number $\gamma \geq 2$ and define the smooth indicator function $\chi(t, x) = \chi_{1, t_{0}, \gamma \epsilon}(t) \, \chi_{2, x_{0}, \gamma \epsilon}(x)$ for $(t, x) \in \mathbb{R}^{2}$. Let $g \in \mathcal{S}'(\mathbb{R}^{2})$ be a tempered distribution with $\operatorname{supp}(g) \subseteq B_{2\epsilon}(t_{0}, x_{0})$, and define the second-order differential operator $\mathcal{Q}_{1} := P^{2} + P_{0}P + P_{0}^{2}$. 
		
		Then, the following propositions hold:
		\begin{enumerate}
			\item[\rm(i)]The weighted derivative estimates satisfy:
			\begin{align*}
				\Big\|\mathcal{F}_{t,x}\left((x-x_{0})\chi \,\langle\nabla_{t,x}\rangle^{s}\partial_{x}\mathcal{Q}_{1}g\right)\Big\|_{L^{r'}_{\tau\xi}} &\leq c \epsilon \|g\|_{\widehat{H}^{r}_{s+3}(\mathbb{R}^{2})}+\frac{c}{\epsilon^{q_{s}}}\| g \|_{\widehat{H}^{r}_{s+2}}, \\
				\Big\|\mathcal{F}_{t,x}\left((t-t_{0})\chi \,\langle \nabla_{t,x}\rangle^{s} \partial_{t} \mathcal{Q}_{1} g\right)\Big\|_{L^{r'}_{\tau\xi}} &\leq c \epsilon \|g\|_{\widehat{H}^{r}_{s+3}(\mathbb{R}^{2})}+\frac{c}{\epsilon^{q_{s}}}\| g \|_{\widehat{H}^{r}_{s+2}},
			\end{align*}
			where $q_{s} > 1$, and $c$ denotes a constant depending on $s$ and the uniform norms of $\chi$ and its derivatives up to a finite order.
			\item[\rm(ii)] The commutator estimates with the spatial and temporal coordinate weights satisfy:
			\begin{align*}
				\Big\| \Big[ \langle \nabla_{t,x} \rangle^{s}; (x-x_{0})\chi \Big] \partial_{x}\mathcal{Q}_{1}g  \Big\|_{\widehat{L}^r_{xt}} &\leq \frac{c}{\epsilon^{p_{s}}}\|g\|_{\widehat{H}^{r}_{s+2}(\mathbb{R}^{2})}, \\
				\Big\| \Big[ \langle \nabla_{t,x} \rangle^{s} ; (t-t_{0})\chi \Big] \partial_{t}\mathcal{Q}_{1}g  \Big\|_{\widehat{L}^r_{xt}} &\leq \frac{c}{\epsilon^{p_{s}}}\|g\|_{\widehat{H}^{r}_{s+2}(\mathbb{R}^{2})}.
			\end{align*}
			where $p_{s} > 1$, and $c$ denotes a constant depending on $s$ and the uniform norms of $\chi$ and its derivatives up to a finite order.
			\item[\rm(iii)] The full commutator estimate with the vector field difference satisfies:
			\begin{equation*}
				\Big\| \big[(P-P_{0})\mathcal{Q}_{1}; \chi\big]g \Big\|_{\widehat{H}^{r}_{s}(\mathbb{R}^{2})} \leq \frac{c}{\epsilon^{y_{s}}}\|g\|_{\widehat{H}^{r}_{s+2}(\mathbb{R}^{2})},
			\end{equation*}
			where $y_{s} > 1$, and $c$ denotes a constant depending on $s$ and the uniform norms of $\chi$ and its derivatives up to a finite order.
		\end{enumerate}
		where $c$ is a positive constant independent of $\epsilon$ and $g$.
	\end{lema}
	\begin{proof}
		We divide the proof into several parts. To establish item \rm(i), we only provide the detailed argument for the first estimate, as the proof for the second estimate follows from a completely analogous analysis.
		
		Notice that 
		\begin{equation}\label{scale1}
			\begin{split}
				\left\|\mathcal{F}_{t,x}\left((x-x_{0})\chi \langle \nabla_{t,x}\rangle^{s} \partial_{x} \mathcal{Q}_{1} g\right)\right\|_{L^{r'}_{\tau\xi}} &= \left\|\mathcal{F}_{t,x}\left((x-x_{0})\chi\right)*\mathcal{F}_{t,x}\left(\langle \nabla_{t,x}\rangle^{s}\partial_{x}\mathcal{Q}_{1} g\right)\right\|_{L^{r'}_{\tau\xi}}\\
				&\leq \left\|\mathcal{F}_{t,x}\left((x-x_{0})\chi\right)\right\|_{L^{1}_{\tau\xi}}\,\left\|\mathcal{F}_{t,x}\left(\langle \nabla_{t,x}\rangle^{s}\partial_{x}\mathcal{Q}_{1} g\right)\right\|_{L^{r'}_{\tau\xi}}.
			\end{split}
		\end{equation}
		The differential operator $\partial_{x}\mathcal{Q}_{1}$ is local and can be represented as 
		\begin{equation}\label{Q1}
			\partial_{x}\mathcal{Q}_{1} = \sum_{\substack{\alpha=(\alpha_{x},\alpha_{t})\in \mathbb{N}_{0}^{2}\\ 1\leq |\alpha|\leq 3}}a_{\alpha}(t,x)\partial^{\alpha}
		\end{equation}
		where each coefficient $a_{\alpha}(t,x)$ is a polynomial in $(t,x)$ of degree $|\alpha|-1$.
		
		Since the cutoff function satisfies $\chi \equiv 1$ on the support of $g$, we have 
		\begin{equation*}
			\begin{split}
				&	\bigg\|\mathcal{F}_{t,x}\left(\langle \nabla_{t,x}\rangle^{s}\partial_{x}\mathcal{Q}_{1} g\right)\bigg\|_{L^{r'}_{\tau\xi}}\\
				&\qquad \leq \sum_{\substack{\alpha=(\alpha_{t},\alpha_{x})\in \mathbb{N}_{0}^{2}\\1\leq |\alpha|\leq 3}}\bigg\|\mathcal{F}_{t,x}\Big(\langle \nabla_{t,x}\rangle^{s}\left(a_{\alpha}\,\chi\,\partial^{\alpha}g \right)\Big)\bigg\|_{L^{r'}_{\tau\xi}} \\
				&\qquad\leq \sum_{\substack{\alpha=(\alpha_{t},\alpha_{x})\in \mathbb{N}_{0}^{2}\\1\leq |\alpha|\leq 3}}\left\{\bigg\|\mathcal{F}_{t,x}\Big(a_{\alpha}\chi\langle \nabla_{t,x}\rangle^{s}\partial^{\alpha}g\Big)\bigg\|_{L^{r'}_{\tau\xi}}+\left\|\mathcal{F}_{t,x}\Big(\big[\langle \nabla_{t,x}\rangle^{s}, a_{\alpha}\chi\big]\partial^{\alpha}g\Big)\right\|_{L^{r'}_{\tau\xi}}\right\} \\
				&\qquad \leq \sum_{\substack{\alpha=(\alpha_{t},\alpha_{x})\in \mathbb{N}_{0}^{2}\\1\leq |\alpha|\leq 3}}\left\{\underbrace{\big\|\mathcal{F}_{t,x}\big(a_{\alpha}\chi\big)\big\|_{L^{1}_{\tau\xi}} \bigg\|\mathcal{F}_{t,x}\Big(\langle \nabla_{t,x}\rangle^{s}\partial^{\alpha}g\Big)\bigg\|_{L^{r'}_{\tau\xi}}}_{\Xi_{1}} + \underbrace{\bigg\|\mathcal{F}_{t,x}\Big(\big[\langle \nabla_{t,x}\rangle^{s}, a_{\alpha}\chi\big]\partial^{\alpha}g\Big)\bigg\|_{L^{r'}_{\tau\xi}}}_{\Xi_{2}}\right\} .
			\end{split}
		\end{equation*}
		The term $\Xi_{1}$ clearly satisfies the following upper bound
		\begin{equation*}
			\Xi_{1}\leq c_{\alpha}\epsilon^{|\alpha|-1} \big\|\mathcal{F}_{t,x}\Big(\langle \nabla_{t,x}\rangle^{s+3}g\Big)\big\|_{L^{r'}_{\tau\xi}}.
		\end{equation*}
		As for $\Xi_{2}$ we have 
		\begin{equation*}
			\begin{split}
				&\left|\mathcal{F}_{t,x}\Big(\big[\langle \nabla_{t,x}\rangle^{s}, a_{\alpha}\chi\big]\partial^{\alpha}g\Big)(\tau, \xi)\right| \\
				&\leq \int_{\mathbb{R}^2} \left|\mathcal{F}_{t,x}(a_{\alpha}\chi)(\tau-\tau', \xi-\xi')\right| 
				\sigma(\tau,\xi,\lambda,\eta) 
				\left| (i\tau')^{\alpha_t}(i\xi')^{\alpha_x} \mathcal{F}_{t,x}(g)(\tau', \xi') \right| \, d\tau' \, d\xi',
			\end{split}
		\end{equation*}
		where 
		\begin{equation*}
			\sigma(\tau,\xi,\lambda,\eta):=\left\langle(\tau,\xi)\right\rangle^{s}-\left\langle(\lambda,\eta)\right\rangle^{s}.
		\end{equation*}
		An application of the Mean Value Theorem yields:
		\begin{equation*}
			|\sigma(\tau,\xi, \lambda,\eta)|\lesssim_{s}\left\langle (\tau-\lambda,\xi-\eta)\right\rangle^{|s-2|}\,\left\langle (\lambda, \eta)\right\rangle^{s-2}\left(|(\lambda,\eta)|+\left|\left(\tau-\lambda,\xi-\eta\right)\right|\right)\, \left|(\tau-\lambda, \xi-\eta)\right|.
		\end{equation*}
		We decompose the domain of integration using the two complementary regions:
		\begin{equation*}
			\Sigma_{1} := \left\{ (\lambda,\eta) \in \mathbb{R}^2 : \frac{1}{2} |(\lambda,\eta)| > |(\tau-\lambda, \xi-\eta)| \right\}
		\end{equation*}
		and
		\begin{equation*}
			\Sigma_{2} := \left\{ (\lambda,\eta) \in \mathbb{R}^2 : \frac{1}{2} |(\lambda,\eta)| \leq |(\tau-\lambda, \xi-\eta)| \right\}.
		\end{equation*}
		On the region $\Sigma_{1}$, the high-frequency variable $(\lambda, \eta)$ dominates and 
		\begin{align*}
			&\iint_{\Sigma_{1}} \left| \mathcal{F}_{t,x}(a_{\alpha}\chi)(\tau-\lambda, \xi-\eta) \sigma(\tau,\xi,\lambda,\eta) \mathcal{F}_{t,x}(\partial^{\alpha}g)(\lambda, \eta) \right| \,\mathrm{d}\lambda \, \mathrm{d}\eta \\
			&\quad \lesssim \iint_{\Sigma_{1}} \langle(\tau-\lambda,\xi-\eta)\rangle^{|s-2|+1} \left| \mathcal{F}_{t,x}(a_{\alpha}\chi)(\tau-\lambda, \xi-\eta) \right| \langle (\lambda,\eta)\rangle^{s-1} \left| \mathcal{F}_{t,x}(\partial^{\alpha}g)(\lambda, \eta) \right| \,\mathrm{d}\lambda \, \mathrm{d}\eta \\
			&\quad \le c \left( \left| \mathcal{F}_{t,x}\left(\langle \nabla_{t,x}\rangle^{|s-2|+1}(a_{\alpha}\chi)\right) \right| * \left| \mathcal{F}_{t,x}\left(\langle \nabla_{t,x}\rangle^{s-1}\partial^{\alpha}g\right) \right| \right)(\tau,\xi) \\
			&\quad \le c \left\| \mathcal{F}_{t,x}\left(\langle \nabla_{t,x}\rangle^{|s-2|+1}(a_{\alpha}\chi)\right) \right\|_{L^1_{\tau\xi}} \left\| \mathcal{F}_{t,x}\left(\langle \nabla_{t,x}\rangle^{s-1}\partial^{\alpha}g\right) \right\|_{L^{r'}_{\tau\xi}} \\
			&\quad \lesssim\frac{1}{\epsilon^{k_{1}}} \| g \|_{\widehat{H}^{r}_{s+|\alpha|-1}},
		\end{align*}
		with $k_{1}=k_{1}(s)\geq 1.$

		Conversely, on $ \Sigma_{2}$ we have 
		\begin{align*}
			&\iint_{\Sigma_{2}} \left| \mathcal{F}_{t,x}(a_{\alpha}\chi)(\tau-\lambda, \xi-\eta) \sigma(\tau,\xi,\lambda,\eta) \mathcal{F}_{t,x}(\partial^{\alpha}g)(\lambda, \eta) \right| \,\mathrm{d}\lambda \, \mathrm{d}\eta \\
			&\quad \lesssim \iint_{\Sigma_{2}} \langle(\tau-\lambda,\xi-\eta)\rangle^{|s-2|+2} \left| \mathcal{F}_{t,x}(a_{\alpha}\chi)(\tau-\lambda, \xi-\eta) \right| \langle (\lambda,\eta)\rangle^{s-2} \left| \mathcal{F}_{t,x}(\partial^{\alpha}g)(\lambda, \eta) \right| \,\mathrm{d}\lambda \, \mathrm{d}\eta \\
			&\quad \le c \left( \left| \mathcal{F}_{t,x}\left(\langle \nabla_{t,x}\rangle^{|s-2|+2}(a_{\alpha}\chi)\right) \right| * \left| \mathcal{F}_{t,x}\left(\langle \nabla_{t,x}\rangle^{s-2}\partial^{\alpha}g\right) \right| \right)(\tau,\xi) \\
			&\quad \le c \left\| \mathcal{F}_{t,x}\left(\langle \nabla_{t,x}\rangle^{|s-2|+2}(a_{\alpha}\chi)\right) \right\|_{L^1_{\tau\xi}} \left\| \mathcal{F}_{t,x}\left(\langle \nabla_{t,x}\rangle^{s-2}\partial^{\alpha}g\right) \right\|_{L^{r'}_{\tau\xi}} \\
			&\quad \lesssim \frac{1}{\epsilon^{k_{1}+1}} \| g \|_{\widehat{H}^{r}_{s+|\alpha|-2}}.
		\end{align*}
		Combining the estimates for $\Xi_{1}$ and $\Xi_{2}$, we obtain:
		\begin{equation}\label{com2}
			\begin{split}
				&\bigg\|\mathcal{F}_{t,x}\bigg(\langle \nabla_{t,x}\rangle^{s}\partial_{x}\mathcal{Q}_{1} g\bigg)\bigg\|_{L^{r'}_{\tau\xi}} \\
				&\quad \lesssim \sum_{\substack{\alpha=(\alpha_{x},\alpha_{t})\in \mathbb{N}_{0}^{2}\\1\leq |\alpha|\leq 3}} c_{\alpha}\bigg\{ \epsilon^{|\alpha|-1} \Big\|\mathcal{F}_{t,x}\Big(\langle \nabla_{t,x}\rangle^{s+3}g\Big)\Big\|_{L^{r'}_{\tau\xi}} +  \frac{1}{\epsilon^{k_{1}+1}} \| g \|_{\widehat{H}^{r}_{s+|\alpha|-1}} \bigg\}\\
				&\lesssim_{\alpha} \| g \|_{\widehat{H}^{r}_{s+3}}+\frac{1}{\epsilon^{k_{1}+1}}\| g \|_{\widehat{H}^{r}_{s+2}}.
			\end{split}
		\end{equation}
		Finally, by substituting the estimates derived above into \eqref{scale1}, we obtain the desired bound. More precisely,
		\begin{equation*}
			\bigg\|\mathcal{F}_{t,x}\left((x-x_{0})\chi \,\langle\nabla_{t,x}\rangle^{s}\partial_{x}\mathcal{Q}_{1}g\right)\bigg\|_{L^{r'}_{\tau\xi}} \leq c \epsilon \|g\|_{\widehat{H}^{r}_{s+3}(\mathbb{R}^{2})}+\frac{c}{\epsilon^{q_{s}}}\| g \|_{\widehat{H}^{r}_{s+2}},
		\end{equation*}
		where $q_{s}=k_{1}.$

		Next, we establish item \rm(ii). Here too, we restrict our attention to the detailed proof of the first estimate, as the temporal companion estimate can be handled by an identical argument.
		
		Notice that 
		\begin{equation}\label{com1}
			\begin{split}
				&\mathcal{F}_{t,x}\bigg(\big[ \langle \nabla_{t,x} \rangle^{s}; (x-x_{0})\chi \big] \partial_{x}\mathcal{Q}_{1}g \bigg)(\tau,\xi)= \\
				& \iint\limits_{\mathbb{R}^2} \sigma(\tau, \xi, \lambda, \eta) \mathcal{F}_{t,x}\big((x-x_{0})\chi\big)(\tau-\lambda, \xi-\eta) 
				\mathcal{F}_{t,x}\left(\partial_{x}\mathcal{Q}_{1}g\right)(\lambda, \eta) \,\mathrm{d}\lambda \, \mathrm{d}\eta.
			\end{split}		
		\end{equation}
		At this point, we shall point out that an  analysis similar to the one described above in item (i) yields
		\begin{align*}
			&\iint_{\Sigma_{1}} \left| \sigma(\tau, \xi, \lambda, \eta) \mathcal{F}_{t,x}\big((x-x_{0})\chi\big)(\tau-\lambda, \xi-\eta) \mathcal{F}_{t,x}\left(\partial_{x}\mathcal{Q}_{1}g\right)(\lambda, \eta) \right| \,\mathrm{d}\lambda \, \mathrm{d}\eta \\
			&\, \lesssim \iint_{\Sigma_{1}} \langle(\tau-\lambda,\xi-\eta)\rangle^{|s-2|+1} \left| \mathcal{F}_{t,x}\big((x-x_{0})\chi\big)(\tau-\lambda, \xi-\eta) \right| \langle (\lambda,\eta)\rangle^{s-1} \left| \mathcal{F}_{t,x}\left(\partial_{x}\mathcal{Q}_{1}g\right)(\lambda, \eta) \right| \,\mathrm{d}\lambda \, \mathrm{d}\eta \\
			&\, \le c \left( \left| \mathcal{F}_{t,x}\left(\langle \nabla_{t,x}\rangle^{|s-2|+1}\left((x-x_{0})\chi\right)\right) \right| * \left| \mathcal{F}_{t,x}\left(\langle \nabla_{t,x}\rangle^{s-1}\partial_{x}\mathcal{Q}_{1}g\right) \right| \right)(\tau,\xi).
		\end{align*}
		Conversely, on the region $\Sigma_{2}$, the low-frequency variable or comparable components dominate, yielding:
		\begin{align*}
			&\iint_{\Sigma_{2}} \left| \sigma(\tau, \xi, \lambda, \eta) \mathcal{F}_{t,x}\big((x-x_{0})\chi\big)(\tau-\lambda, \xi-\eta) \mathcal{F}_{t,x}\left(\partial_{x}\mathcal{Q}_{1}g\right)(\lambda, \eta) \right| \,\mathrm{d}\lambda \, \mathrm{d}\eta \\
			&\, \lesssim \iint_{\Sigma_{2}} \langle(\tau-\lambda,\xi-\eta)\rangle^{|s-2|+2} \left| \mathcal{F}_{t,x}\big((x-x_{0})\chi\big)(\tau-\lambda, \xi-\eta) \right| \langle (\lambda,\eta)\rangle^{s-2} \left| \mathcal{F}_{t,x}\left(\partial_{x}\mathcal{Q}_{1}g\right)(\lambda, \eta) \right| \,\mathrm{d}\lambda \, \mathrm{d}\eta \\
			&\, \le c \left( \left| \mathcal{F}_{t,x}\left(\langle \nabla_{t,x}\rangle^{|s-2|+2}\left((x-x_{0})\chi\right)\right) \right| * \left| \mathcal{F}_{t,x}\left(\langle \nabla_{t,x}\rangle^{s-2}\partial_{x}\mathcal{Q}_{1}g\right) \right| \right)(\tau,\xi).
		\end{align*}
		By combining the estimates obtained across both regions and applying Young's convolution inequality, we utilize the explicit formula for $\partial_{x}\mathcal{Q}_{1}$ given in \eqref{com1}. Following an argument analogous to the one used to derive \eqref{com2}, we obtain the desired bound:
		\begin{align*}
			&\bigg\| \big[ \langle \nabla_{t,x} \rangle^{s}; (x-x_{0})\chi \big] \partial_{x}\mathcal{Q}_{1}g \bigg\|_{\widehat{L}^r_{xt}} \\
			&\qquad \lesssim \left\|\langle \nabla_{t,x}\rangle^{|s-2|+2}\left((x-x_{0})\chi\right)\right\|_{\widehat{L}^r_{xt}} \sum_{\substack{\alpha=(\alpha_{x},\alpha_{t})\in \mathbb{N}_{0}^{2} \\ 1\leq |\alpha|\leq 3}} \left\|\mathcal{F}_{t,x}\left(\langle \nabla_{t,x}\rangle^{s-1}\left(a_{\alpha}\,\chi\,\partial^{\alpha}g \right)\right)\right\|_{L^{r'}_{\tau\xi}} \\
			&\qquad \leq c \left\|\langle \nabla_{t,x}\rangle^{|s-2|+2}\left((x-x_{0})\chi\right)\right\|_{\widehat{L}^r_{xt}} \sum_{\substack{\alpha=(\alpha_{x},\alpha_{t})\in \mathbb{N}_{0}^{2} \\ 1\leq |\alpha|\leq 3}} \big\|a_{\alpha}\,\chi \,\partial^{\alpha}g \big\|_{\widehat{H}^{r}_{s-1}(\mathbb{R}^{2})} \\
			&\qquad \lesssim_{\alpha} \frac{1}{\epsilon^{\widetilde{k_{1}}}}\| g \|_{\widehat{H}^{r}_{s+2}},
		\end{align*}
		where the implicit constant is independent of $\epsilon$, and $\widetilde{k_{1}} \geq 1.$
		
		Finally, we establish item \rm(iii), which addresses the full commutator estimate involving the vector field difference. By employing the classical algebraic identity for commutators, the operator can be decoupled as follows:
		\begin{equation*}
			\big[(P-P_{0})\mathcal{Q}_{1}; \chi \big] = (P-P_{0})\big[\mathcal{Q}_{1}; \chi\big] + \big[P-P_{0}; \chi\big]\mathcal{Q}_{1}.
		\end{equation*}
		A direct computation shows that the first inner commutator, $\big[\mathcal{Q}_{1}; \chi\big]$, is a first-order local differential operator of the form
		\begin{equation*}
			(P-P_{0})\big[\mathcal{Q}_{1}; \chi\big] = \sum_{\substack{\alpha=(\alpha_t, \alpha_x) \in \mathbb{N}_{0}^{2} \\ |\alpha| \le 2}}  b_{\alpha}(t,x)\partial^{\alpha}.
		\end{equation*}
		On the other hand, the second commutator, $\big[P-P_{0}; \chi\big]$, reduces to a zero-order multiplicative operator given exactly by 
		\begin{equation*}
			\big[P-P_{0}; \chi\big] = (P-P_{0})(\chi).
		\end{equation*}
		Since $\mathcal{Q}_{1}$ is linear and local, it can be written as 
		\begin{equation*}
			\mathcal{Q}_{1} = \sum_{\substack{\beta=(\beta_t, \beta_x) \in \mathbb{N}_{0}^{2} \\ |\beta| \le 2}} q_{\beta}(t,x) \partial^{\beta},
		\end{equation*}
		and thus, applying an argument analogous to the one used in \eqref{com2}, we obtain the desired bound:
		\begin{equation*}
			\begin{split}
				\Big\| \big[ (P-P_{0})\mathcal{Q}_{1}; \chi \big] g \Big\|_{\widehat{H}^r_{s}} 
				&\leq \sum_{\substack{\alpha=(\alpha_t, \alpha_x) \in \mathbb{N}_{0}^{2} \\ |\alpha| \le 2}} \Big\| \langle \nabla_{t,x} \rangle^{s} (b_{\alpha} \partial^{\alpha}g) \Big\|_{\widehat{L}^r_{t,x}} \\
				&\quad + \Big\| \langle \nabla_{t,x} \rangle^s \big( (P - P_0)(\chi) \mathcal{Q}_1 g \big) \Big\|_{\widehat{L}^r_{t,x}} \\
				&\lesssim_{\alpha} \frac{1}{\epsilon^{\widetilde{k_{1}}}}\| g \|_{\widehat{H}^{r}_{s+2}}.
			\end{split}
		\end{equation*}
	\end{proof}
	We close this appendix with an embedding lemma, used elsewhere in the paper to transfer $X^{r}_{s,b}$-norm control into pointwise-in-time bounds.
	\begin{lema}\label{stabilitylinfty}
		Let $\mathcal{Y} \subset \mathcal{S}'(\mathbb{R}^{n+1})$ be a Banach space that is stable under multiplication by $L^\infty_t$ functions, satisfying 
		\begin{equation*}
			\|\psi u\|_{\mathcal{Y}} \leq c \|\psi\|_{L^\infty_t} \|u\|_{\mathcal{Y}}, \quad \forall \psi \in L^\infty_t, \, u \in \mathcal{Y}.
		\end{equation*}
		Let $\{U_{\phi}(t)\}_{t \in \mathbb{R}}$ be the unitary group defined by $U_{\phi}(t) = e^{it\phi(D)}$, where the phase function $\phi: \mathbb{R}^n \to \mathbb{R}$ is of polynomial growth. Assume that the estimate 
		\begin{equation*}
			\|U_{\phi}(t)u_0\|_{\mathcal{Y}} \leq c \|u_0\|_{\widehat{L}^r_x}
		\end{equation*}
		holds for all $u_0 \in \widehat{L}^r_x$. Then, for any $b > \frac{1}{r}$, the following estimate holds for all $u \in X^r_{s,b}$:
		\begin{equation*}
			\|u\|_{\mathcal{Y}} \leq c(b) \|u\|_{X^{r}_{s,b}},
		\end{equation*}
		where the constant $c(b)$ depends only on $b$.
	\end{lema}
	\begin{proof}
		For a detailed proof of this estimate, we refer the reader to \cite[Lemma 2.1]{Gru2004}.
	\end{proof}
	Finally, for completeness, we record two classical estimates used repeatedly throughout Sections~\ref{sectA}, \ref{sect2}: the $L^8_{xt}$ Strichartz estimate for the Airy propagator, and a fractional Leibniz (Kato-Ponce type) inequality.
	\begin{teorema}\label{kpvstrichartz}
		There exists a constant $c > 0$ such that the following Strichartz estimate holds for the Airy propagator:
		\begin{equation}
			\left\|\int\limits_{\mathbb{R}}  e^{i(x\xi + t\xi^3)} \widehat{f}(\xi) \, d\xi \right\|_{L^{8}_{xt}(\mathbb{R}^2)} \leq c \|f\|_{L^{2}_{x}(\mathbb{R})}.
		\end{equation}
	\end{teorema}
	\begin{proof}
		For a detailed derivation of this $L^{8}_{xt}$ estimate, see Theorem 2.4 in \cite{KPV1991}.
	\end{proof}
	Together with the localized trace and commutator estimates established earlier in this appendix, this fractional Leibniz rule provides all the technical ingredients needed to close the proofs of Theorems~\ref{main1} and~\ref{main2}.
	\begin{teorema}\label{kp1}
		Let $\frac{1}{2} < r < \infty$, $1 < p_1, p_2, q_1, q_2 \le \infty$ satisfy $\frac{1}{r} = \frac{1}{p_1} + \frac{1}{q_1} = \frac{1}{p_2} + \frac{1}{q_2}$. Given $s > \max\left(0, \frac{n}{r} - n\right)$ or $s \in 2\mathbb{N}$, there exists a constant $c = c(n, s, r, p_1, q_1, p_2, q_2) < \infty$ such that for all $f, g \in \mathcal{S}(\mathbb{R}^n)$ we have:
		\begin{align*}
			\||\nabla|^s(fg)\|_{L^r(\mathbb{R}^n)} &\le c \left( \||\nabla|^s f\|_{L^{p_1}(\mathbb{R}^n)} \|g\|_{L^{q_1}(\mathbb{R}^n)} + \|f\|_{L^{p_2}(\mathbb{R}^n)} \||\nabla|^s g\|_{L^{q_2}(\mathbb{R}^n)} \right), \\[1ex]
			\|\langle \nabla\rangle^s(fg)\|_{L^r(\mathbb{R}^n)} &\le c(s, n) \left( \|f\|_{L^{p_1}(\mathbb{R}^n)} \|\langle \nabla\rangle^s g\|_{L^{q_1}(\mathbb{R}^n)} + \|\langle \nabla\rangle^s f\|_{L^{p_2}(\mathbb{R}^n)} \|g\|_{L^{q_2}(\mathbb{R}^n)} \right)
		\end{align*}
	\end{teorema}
	\begin{proof}
		The original case, in the $L^2$ setting, was proved by Kato and Ponce \cite{KP88}; the general version stated here is due to Grafakos and Oh \cite{GO2014}.
	\end{proof}
\section{Acknowledgment}
The figures in this paper were generated with the assistance of Claude (Anthropic). The authors would like to thank Felipe Linares for his feedback on a previous version of this work and his valuable suggestions regarding the references. We also thank Oscar Riaño for his comments and suggestions on this work.

	\end{document}